\documentclass[11pt]{amsart}
\usepackage{amssymb,amsmath,amsfonts,amsthm,graphics,mathrsfs,amscd,hyperref}
\usepackage[hmargin=1in,vmargin=1in]{geometry}
\usepackage[all,cmtip]{xy}

\theoremstyle{plain}

\newtheorem{definition}[equation]{Definition}
\newtheorem{corollary}[equation]{Corollary}
\newtheorem{lemma}[equation]{Lemma}
\newtheorem{proposition}[equation]{Proposition}
\newtheorem{theorem}[equation]{Theorem}

\newtheorem{example}{Example}

\newtheorem{remark}[equation]{Remark}

\numberwithin{equation}{subsection}

\title{Deformations of nonsingular Poisson varieties and Poisson invertible sheaves}
\author{Chunghoon Kim}
\date{Februry 21, 2015}
\thanks{The author was partially supported by NRF grant 2011-0027969.}
\address{Reserch Institute of Mathematics\\
               Seoul National University\\
               Seoul 151-747, Korea}
\email{ckim042@gmail.com}            

\begin{document}

\maketitle
\begin{abstract}
In this paper, we study deformations of nonsingular Poisson varieties, deformations of Poisson invertible sheaves and simultaneous deformations of nonsingular Poisson varieties and Poisson invertible sheaves, which extend flat deformation theory of nonsingular varieties and invertible sheaves. In an appendix, we study deformations of Poisson vector bundles. We identify first-order deformations and obstructions.
\end{abstract}

\tableofcontents

\section{Introduction}

In this paper, we study deformations of nonsingular Poisson varieties and Poisson invertible sheaves, which extend  flat deformation theory of algebraic schemes and invertible sheaves (see \cite{Ser06},\cite{Har10}). In other words, when we ignore Poisson structures, underlying deformation theory is exactly same to flat deformation theory of algebraic schemes and invertible sheaves. An algebraic Poisson scheme $X$ over $k$ is an algebraic scheme over $k$ whose structure sheaf $\mathcal{O}_X$ is a sheaf of Poisson $k$-algebras\footnote{For general information on Poisson geometry, we refer to \cite{Lau13}.}, where $k$ is an algebraically closed field with characteristic $0$.  Equivalently, a Poisson structure on an algebraic scheme $X$ is characterized by an element $\Lambda_0\in \Gamma(X,\mathscr{H}om_{\mathcal{O}_X}(\wedge^2\Omega_{X/k}^1,\mathcal{O}_X))$ with $[\Lambda_0,\Lambda_0]=0$, where $[-,-]$ is the Schouten bracket on $X$. In the sequel  we denote an algebraic Poisson scheme by $(X,\Lambda_0)$. It is known that infinitesimal deformations of a nonsingular variety $X$ is controlled by the tangent sheaf $T_X$ so that $H^1(X,T_X)$ represents first-order deformations and $H^2(X,T_X)$ represents obstructions (see \cite{Ser06} Proposition 1.2.9, Proposition 1.2.12). On the other hand, infinitesimal deformations of an invertible sheaf $L$ on a nonsingular variety $X$ is controlled by the structure sheaf $\mathcal{O}_X$ so that $H^1(X,\mathcal{O}_X)$ represents first-order deformations and $H^2(X,\mathcal{O}_X)$ represents obstructions (see \cite{Ser06} Theorem 3.3.1). In Poisson deformations of a nonsingular Poisson variety $(X,\Lambda_0)$,  the role of $T_X$ is replaced by degree-shifted (by $1$) truncated Lichnerowicz-Poisson complex $T_X^\bullet: T_X\to \wedge^2 T_X\to \cdots $ induced by $[\Lambda_0,-]$. We will denote the $i$-th hypercohomology group by $\mathbb{H}^i(X,\Lambda_0,T_X^\bullet)$\footnote{In \cite{Kim14},\cite{Kim15} , the author used the notation $HP^i(X,\Lambda_0)$ for the $i$-th hypercohomology group of unshifted truncated Lichnerowicz-Poisson complex $0\to T_X\to \wedge^2 T_X\to \wedge^3 T_X\to \cdots$ induced by $[\Lambda_0,-]$ in order to keep notational consistency with \cite{Nam09}, \cite{Gin04}. However the author decides to use the notation $\mathbb{H}^i(X,\Lambda_0, T_X^\bullet)$ to express the $i$-th hypercohomology group of shifted truncated Lichnerowicz-Poisson complex $T_X\to \wedge^2 T_X\to \wedge^3 T_X\to \cdots$ induced by $[\Lambda_0,-]$ since it looks more natural by the general philosophy of deformation theory.}. Then $\mathbb{H}^1(X,\Lambda_0,T_X^\bullet)$ represents first-order deformations and $\mathbb{H}^2(X,\Lambda_0,T_X^\bullet)$ represents obstructions (see Proposition \ref{3q}, Proposition \ref{lifting}). A  Poisson invertible sheaf $L$ is an invertible sheaf equipped with a flat Poisson connection $\nabla$. We will denote the Poisson invertible sheaf by $(L,\nabla)$ (see Definition \ref{pc}).  In Poisson deformations of a Poisson invertible sheaf $(L,\nabla)$, the role of $\mathcal{O}_X$ is replaced by Lichnerowicz-Poisson complex $\mathcal{O}_X^\bullet:\mathcal{O}_X\to T_X\to \wedge^2 T_X\to \cdots$ induced by $[\Lambda_0,-]$. We will denote the $i$-th hypercohomology group by $\mathbb{H}^i(X,\Lambda_0,\mathcal{O}_X^\bullet)$. Then $\mathbb{H}^1(X,\Lambda_0,\mathcal{O}_X^\bullet)$ represents first-order deformations of $(L,\nabla)$ and $\mathbb{H}^2(X,\Lambda_0,\mathcal{O}_X^\bullet)$ represents obstructions (see Proposition \ref{tpl}).

We will review simultaneous deformation theory of nonsingular varieties and invertible sheaves, and explain how the theory is extended in terms of simultaneous deformations of nonsingular Poisson varieties and Poisson invertible sheaves.  Let $X$ be a nonsingular projective variety and $L$ be an invertible sheaf on it. We can identify isomorphism classes of invertible sheaves on $X$ with the Picard group $H^1(X,\mathcal{O}_X^*)$. Then the image of $L$ under $H^1(X,\mathcal{O}_X^*)\to H^1(X,\Omega_X^1)$ induced by $d\log-$, where $d$ is the canonical derivation, defines the Chern class $c(L)\in H^1(X,\Omega_X^1)$, which  gives the Atiyah extension associated with $c(L)$,  $0\to \mathcal{O}_X\to \mathcal{E}_L^0\to T_X\to 0$. Then simultaneous deformations of a nonsingular projective variety $X$ and an invertible sheaf $L$ is controlled by the locally free sheaf $\mathcal{E}_L^0$ so that $H^1(X,\mathcal{E}_L^0)$ represents first-order deformations and $H^2(X,\mathcal{E}_L^0)$ represents obstructions (see \cite{Ser06} Theorem 3.3.11). Now assume that $(X,\Lambda_0)$ is a nonsingular projective Poisson variety over $\mathbb{C}$. We can identify isomorphism classes of Poisson invertible sheaves on $(X,\Lambda_0)$ with the Poisson Picard group $\mathbb{H}^1(X,\Lambda_0,\mathcal{O}_X^{*\bullet})$ which is the first hypercohomology group of the complex of sheaves $\mathcal{O}_X^{*\bullet} :\mathcal{O}_X^*\xrightarrow{[\Lambda_0,\log -]} T_X \xrightarrow{[\Lambda_0,-]} \wedge^2 T_X\xrightarrow{[\Lambda_0,-]}\cdots$ (see Definition \ref{pic}). We define the notion of the Poisson Chern class associated with a Poisson invertible sheaf $(L,\nabla)$ where the Chern class $c(L)$ of the invertible sheaf $L$ is realized as a component of the Poisson Chern class in the following way (see Definition \ref{pchern}). We have a morphism of complex of sheaves, $\mathcal{O}_X^{*\bullet} \to (\Omega_X^1\xrightarrow{i_{\Lambda_0}} T_X)$, which induces $\mathbb{H}^1(X,\Lambda_0,\mathcal{O}_X^{*\bullet})\to \mathbb{H}^1(\Omega_X^1\xrightarrow{i_{\Lambda_0}} T_X)$, where $i_{\Lambda_0}:\Omega_X^1\to T_X$ is the natural morphism induced by the bivector field $\Lambda_0\in H^0(X,\wedge^2 T_X)$ by contraction. We call the image of $(L,\nabla)$ the Poisson Chern class associated with $(L,\nabla)$ and denote it by $c(L,\nabla)$. Then the Poisson Chern class $c(L,\nabla)\in \mathbb{H}^1(\Omega_X^1\to T_X)$ gives the Poisson Atiyah extension $0\to \mathcal{O}_X^\bullet\to \mathcal{E}_L^\bullet\to T_X^\bullet \to 0$ which extends the  Atiyah extension $0\to \mathcal{O}_X\to \mathcal{E}_L^0\to T_X\to 0$ associated with the Chern class $c(L)$, where $\mathcal{E}_L^\bullet:\mathcal{E}_L^0\xrightarrow{d}\mathcal{E}_L^1\xrightarrow{d} \mathcal{E}_L^2\to \cdots$ is a complex of locally free sheaves and $d$ is induced by $[\Lambda_0,-]$ and the component of $c(L,\nabla)$ associated with $T_X$ (see Proposition \ref{patiyah}). In simultaneous deformations of a nonsingular projective Poisson variety $(X,\Lambda_0)$ and a Poisson invertible sheaf $(L,\nabla)$, the role of $\mathcal{E}_L^0$ in simultaneous deformations of a nonsingular projective variety $X$ and an invertible sheaf $L$ is replaced by $\mathcal{E}_L^\bullet:\mathcal{E}_L^0\xrightarrow{d}\mathcal{E}_L^1\xrightarrow{d}\mathcal{E}_L^1\to\cdots$.  We will denote the $i$-th hypercohomology group by $\mathbb{H}^i(X,\Lambda_0,\mathcal{E}_L^\bullet)$ so that $\mathbb{H}^1(X,\Lambda_0,\mathcal{E}_L^\bullet)$ represents first-order deformations and $\mathbb{H}^2(X,\Lambda_0,\mathcal{E}_L^\bullet)$ represents obstructions (see Proposition \ref{spo}).

In section \ref{section1}, we review the characterization of the Poisson structure of an algebraic Poisson scheme $X$ over a base scheme $S$ by an element $\Lambda\in \Gamma(X,\mathscr{H}om_{\mathcal{O}_X}(\wedge^2 \Omega_{X/S}^1,\mathcal{O}_X))$ with $[\Lambda,\Lambda]=0$ where $[-,-]$ is the Schouten bracket on $X$ over $S$ (see Remark \ref{remarkc}). We also review Lichnerowicz-Poisson complex and define degree-shifted (by $1$) truncated Lichnerowicz-Poisson complex with their cohomology groups (see Remark \ref{remarkd}).

In section \ref{section2}, we study deformations of algebraic Poisson schemes which extend flat deformation theory of algebraic schemes. We define the Poisson deformation functor $Def_{(X,\Lambda_0)}$ which is a functor of Artin rings for an algebraic Poisson scheme $(X,\Lambda_0)$ (see Definition \ref{definitionp}). We identify first order deformations of a nonsingular Poisson variety $(X,\Lambda_0)$ with $\mathbb{H}^1(X,\Lambda_0,T_X^\bullet)$ and obstructions with $\mathbb{H}^2(X,\Lambda_0,T_X^\bullet)$ (see Proposition \ref{3q}, Proposition \ref{lifting}). We show that for a nonsingular projective Poisson variety $(X,\Lambda_0)$ with $\mathbb{H}^0(X,\Lambda_0,T_X^\bullet)=0$, the Poisson deformation functor $Def_{(X,\Lambda_0)}$ is pro-representable (see Theorem \ref{ppro}).

In section \ref{section3}, we study Poisson invertible sheaves on an algebraic Poisson scheme $(X,\Lambda_0)$. A Poisson invertible sheaf $L$ on $(X,\Lambda_0)$ is an invertible sheaf equipped with a flat Poisson connection $\nabla$, which is denoted by $(L,\nabla)$ (see Definition \ref{pc}). The group of Poisson invertible sheaves on an algebraic Poisson scheme on $(X,\Lambda_0)$ can be identified with $\mathbb{H}^1(X,\Lambda_0,\mathcal{O}_X^{*\bullet})$ (see Definition \ref{pic}). A flat Poisson connection $\nabla$ on an invertible sheaf $L$ on a nonsingular Poisson variety $(X,\Lambda_0)$ is extended to a complex of sheaves $L^\bullet:L\xrightarrow{\nabla} T_X\otimes L\to \wedge^2 T_X\otimes L \to\cdots$ (see Remark \ref{lco}). We denote the $i$-th hypercohomology group by $\mathbb{H}^i(X,\Lambda_0,L^\bullet, \nabla)$. We define the notion of the Poisson Chern class associated with a Poisson invertible sheaf on a nonsingular Poisson variety over $\mathbb{C}$. (see Definition \ref{pchern}).

In section \ref{section4}, we study deformations of Poisson invertible sheaves under the trivial Poisson deformations. We define the associated deformation functor $Def_{(L,\nabla)}$ for a Poisson invertible sheaf $(L,\nabla)$ on a nonsingular Poisson variety $(X,\Lambda_0)$ (see Definition \ref{unt}). We identify first-order deformations of a Poisson invertible sheaf on a nonsingular Poisson variety with $\mathbb{H}^1(X,\Lambda_0,\mathcal{O}_X^\bullet)$ and obstructions with $\mathbb{H}^2(X,\Lambda_0,\mathcal{O}_X^\bullet)$ (see Proposition \ref{tpl}). We show that for a nonsingular projective Poisson variety, the functor $Def_{(L,\nabla)}$ is pro-representable (see Theorem \ref{ppro2}).

In section \ref{section5}, we study deformations of sections of a Poisson invertible sheaf $(L,\nabla)$ in the trivial Poisson deformations. We show that the formalism of deformations of sections of an invertible sheaf $L$ in the trivial deformations can be extended to Poisson deformations by replacing $H^1(X,\mathcal{O}_X)$ and $H^i(X,L)$ by $\mathbb{H}^1(X,\Lambda_0,\mathcal{O}_X^\bullet)$ and $\mathbb{H}^i(X,\Lambda_0,L^\bullet,\nabla), i=0,1$, respectively. For a global section $s\in \mathbb{H}^0(X,\Lambda_0,L^\bullet,\nabla)$ and a first-order deformation $a\in \mathbb{H}^1(X,\Lambda_0,\mathcal{O}_X^\bullet)$, we study a condition when $s$ can be extended to a section $\tilde{s}\in \mathbb{H}^0(X\times Spec(k[\epsilon]),\Lambda_0,\mathcal{L}_a^\bullet,\nabla_a)$ of the first-order deformation associated with $a$ (see Proposition \ref{sec1}).

In section \ref{section6}, we study simultaneous deformations of nonsingular Poisson varieties and Poisson invertible sheaves. We define the associated deformation functor $Def_{(X,\Lambda_0,L,\nabla)}$ for a nonsingular Poisson variety $(X,\Lambda_0)$ and a Poisson invertible sheaf $(L,\nabla)$ on $(X,\Lambda_0)$ (see Definition \ref{simul}). Given a nonsingular projective Poisson variety $(X,\Lambda_0)$ over $\mathbb{C}$ and a Poisson invertible sheaf $(L,\nabla)$, we show that the Poisson Chern class $c(L,\nabla)$ of $(L,\nabla)$ gives the Poisson Atiyah extension $0\to \mathcal{O}_X^\bullet\to \mathcal{E}_L^\bullet\to T_X^\bullet\to 0$ (see Proposition \ref{patiyah}). We identify first-order deformations with $\mathbb{H}^1(X,\Lambda_0,\mathcal{E}_L^\bullet)$ and obstructions with $\mathbb{H}^2(X,\Lambda_0,\mathcal{E}_L^\bullet)$ (see Proposition \ref{spo}).

In section \ref{section7}, we study deformations of sections of a Poisson invertible sheaf $(L,\nabla)$ in flat Poisson deformations. We show that the formalism of deformations of sections of an invertible sheaf $L$ in flat deformations can be extended to Poisson deformations by replacing $H^1(X,\mathcal{E}_L^\bullet)$ and $H^i(X,L)$ by $\mathbb{H}^1(X,\Lambda_0,\mathcal{E}_L^0)$ and $\mathbb{H}^i(X,\Lambda_0,L^\bullet,\nabla), i=0,1$, respectively. For a global section $s\in \mathbb{H}^0(X,\Lambda_0,L^\bullet,\nabla)$ and a first-order deformation $\eta\in \mathbb{H}^1(X,\Lambda_0,\mathcal{E}_L^\bullet)$, we study a condition when $s$ can be extended to a section $\tilde{s}\in\mathbb{H}^0(\mathcal{X},\Lambda, \mathcal{L}^\bullet,\nabla_{\mathcal{L}})$ of the first-order deformation associated with $\eta$ (see Proposition \ref{sec2}).

In Appendix \ref{appendixB}, more generally, we study deformations of Poisson vector bundles which extend deformations of vector bundles. It is known that infinitesimal deformations of a vector bundle $F$ on a nonsingular variety $X$ is controlled by the sheaf $\mathscr{H}om(F,F)$ so that $H^1(X,\mathscr{H}om(F,F))=Ext^1(F,F)$ represents first-order deformations and $H^2(X,\mathscr{H}om(F,F))=Ext^2(F,F)$ represents obstructions. Given a nonsingular Poisson variety $(X,\Lambda_0)$, a Poisson vector bundle $F$ is a locally free sheaf of finite rank equipped with a flat Poisson connection $\nabla$. We will denote the Poisson vector bundle by $(F,\nabla)$ (see Definition \ref{ve3}). Then $\mathscr{H}om(F,F)$ has a natural flat Poisson connection $\nabla_{\mathscr{H}om(F,F)}$ which defines a complex of sheaves $\mathscr{H}om(F,F)^\bullet:\mathscr{H}om(F,F)\to T_X\otimes \mathscr{H}om(F,F)\to \wedge^2 T_X\otimes \mathscr{H}om(F,F)\to \cdots$. We will denote the $i$-th hypercohomology group by $\mathbb{H}^i(X,\Lambda_0, \mathscr{H}om(F,F)^\bullet,\nabla_{\mathscr{H}om(F,F)})$ (see Remark \ref{ve5}). We define the associated deformation functor $Def_{(F,\nabla)}$ (see Definition \ref{ve4}). We show that infinitesimal deformations of a Poisson vector bundle $(F,\nabla)$ is controlled by $\mathscr{H}om(F,F)^\bullet$ so that $\mathbb{H}^1(X,\Lambda_0,\mathscr{H}om(F,F)^\bullet,\nabla_{\mathscr{H}om(F,F)})$ represents first-order deformations and $\mathbb{H}^2(X,\Lambda_0, \mathscr{H}om(F,F)^\bullet,\nabla_{\mathscr{H}om(F,F)})$ represents obstructions (see Proposition \ref{ve2}).

\section{Preliminaries}\label{section1}
In this paper, every algebra is a commutative $k$-algebra, where $k$ is an algebraically closed field with characteristic $0$. We review the characterization of a Poisson structure on a commutative algebra $A$ over $R$ in terms of an element $\Lambda \in Hom_A(\wedge^2 \Omega_{R/A}^1,A)$ with $[\Lambda,\Lambda]=0$ where $[-,-]$ is the Schouten bracket on $\bigoplus_{p\geq 0}Hom_A(\wedge^p \Omega_{A/R}^1,A)$ and $\Omega_{A/R}^1$ is the $A$-module of relative K\"{a}hler differential forms of $A$ over $R$. For the detail, we refer to \cite{Lau13} Chapter 3. Let $d:A\to \Omega_{A/R}^1$ be the canonical map.
\begin{definition}[Poisson algebras]
A commutative algebra $A$ over a commutative algebra $R$ is a Poisson algebra over $R$ if there is an operation $\{-,-\}:A\times A\to A$ such that for $F,G,H\in A$,
\begin{enumerate}
\item $\{-,-\}$ is a skew-symmetric $R$-bilinear $:$ $\{F,G\}=-\{G,F\}$.
\item $\{FG,H\}=F\{G,H\}+G\{F,H\}$ $($biderivation$)$
\item $\{F,\{G,H\}\}+\{G,\{H,F\}\}+\{H,\{F,G\}\}=0$ $($Jacobi identity$)$.
\end{enumerate}
$\{-,-\}$ is called the Poisson bracket of the Poisson algebra $A$ over $R$. The bracket $\{-,-\}$ defines an element $\Lambda\in Hom_A(\wedge^2 \Omega_{A/R}^1,A)$ so that for $F,G\in A$, $\{F,G\}=\Lambda(dF\wedge dG)$. Then we will denote by $(A,\Lambda)$ the Poisson algebra $A$ over $R$ with the Poisson bracket $\{-,-\}$.
\end{definition}

\begin{definition}
For $p,q \in \mathbb{N}$, a $(p,q)$-shuffle is a permutation $\sigma$ of the set $\{1,...,p+q\}$, such that $\sigma(1)< \cdots < \sigma(p)$ and $\sigma(p+1) < \cdots <\sigma(p+q)$. The set of all $(p,q)$-shuffles is denoted by $S_{p,q}$. For a shuffle $\sigma \in S_{p,q}$, we denote the signature of $\sigma$ by $sgn(\sigma)$. By convention, $S_{p,-1}:=\emptyset$ and $S_{-1,q}:= \emptyset $ for $p,q\in \mathbb{N}$.
\end{definition}

\begin{definition}
We define the Schouten bracket $[-,-]$ on $\bigoplus_{p\geq 0} Hom_A(\wedge^p \Omega_{A/R}^1,A)$, namely a family of maps
\begin{align*}
[-,-]:Hom_A(\wedge^p \Omega_{A/R}^1,A) \times Hom_A(\wedge^q \Omega_{A/R}^1,A) \to Hom_A(\wedge^{p+q-1}\Omega_{A/R}^1,A)
\end{align*}
for $p,q\in \mathbb{N}$ in the following way. Let $P\in Hom_A(\wedge^p \Omega_{A/R}^1,A)$ and $Q\in Hom_A(\wedge^q \Omega_{A/R}^1,A)$, and let $F_1,...,F_{p+q-1}\in A$. Then $[P,Q]\in Hom_A(\wedge^{p+q-1}\Omega_{A/R}^1,A)$ is defined by
\begin{align*}
[P,Q](dF_1\wedge \cdots\wedge dF_{p+q-1})=\sum_{\sigma\in S_{q,p-1}}sgn(\sigma)P(d(Q(dF_{\sigma(1)}\wedge...\wedge dF_{\sigma(q)}))\wedge dF_{\sigma(q+1)}\cdots\wedge dF_{\sigma(q+p-1)})\\
                                     -(-1)^{(p-1)(q-1)}\sum_{\sigma\in S_{p,q-1}}sgn(\sigma)Q(d(P(dF_{\sigma(1)}\wedge...\wedge dF_{\sigma(p)}))\wedge dF_{\sigma(p+1)}\wedge\cdots \wedge dF_{\sigma(p+q-1)})
\end{align*} 
\end{definition}

\begin{example}\label{3ex}
Let $P\in Hom_A(\wedge^2 \Omega_{A/R}^1,A)$, $Q\in Hom_A(\Omega_{A/R}^1,A)$ and $R\in Hom_A(\Omega_{A/R}^1,A)$. Then
\begin{enumerate}
\item $[P,Q](dF_1\wedge dF_2)=P(dQ(dF_1)\wedge dF_2)-P(d(Q(F_2))\wedge dF_1)-Q(d(P(dF_1\wedge dF_2)))$
\item $[Q,R](dF)=Q(dR(dF))-R(dQ(dF))$
\end{enumerate}
\end{example}

\begin{proposition}
Let $A$ be a commutative algebra over $R$. If $\Lambda$ is a skew symmetric biderivation of $A$ over $R$, i.e $\Lambda\in Hom_A(\wedge^2 \Omega_{A/R}^1,A)$, then $P$ defines a Poisson bracket $($i.e Jacobi identity holds$)$ if and only if $[\Lambda,\Lambda]=0$.
\end{proposition}

\begin{proof}
See \cite{Lau13} Proposition 3.5  page 80.
\end{proof}


\begin{remark}
Let $(A,\Lambda)$ be a Poisson algebra over $R$ with $\Lambda \in Hom_A(\wedge^2, \Omega_{A/R}^1,A)$ with $[\Lambda,\Lambda]=0$. Then we have the following properties: for $P\in Hom_A(\wedge^p \Omega_{A/R}^1,A)$ and $Q\in Hom_A(\wedge^q \Omega_{A/R}^1,A)$ and $S\in Hom_A(\wedge^r \Omega_{A/R}^1,A) $,
\begin{enumerate}
\item $[\Lambda,[\Lambda,P]]]=0$ and $[\Lambda,P]\in Hom_A(\wedge^{p+1} \Omega_{A/R}^1,A)$
\item $[P,Q]=-(-1)^{(p-1)(q-1)}[Q,P]$
\item $[[P,Q],S]=[P,[Q,S]]-(-1)^{(p-1)(q-1)}[Q,[P,S]]$
\item $[\Lambda,[P,Q]]=[[\Lambda,P],Q]+(-1)^{p-1}[P,[\Lambda,Q]]$
\item $[P,Q\wedge S]=Q\wedge [P,S]+(-1)^{(p-1)r}[P,Q]\wedge S $
\item $[\Lambda,\Lambda]=0$ define a complex
\begin{align*}
A\xrightarrow{[\Lambda,-]}  Hom_A(\Omega_{A/R}^1,A)\xrightarrow{[\Lambda,-]} Hom_A (\wedge^2 \Omega_{A/R}^1,A)\xrightarrow{[\Lambda,-]} Hom_A(\wedge^3 \Omega_{A/R}^1,A)\xrightarrow{[\Lambda,-]} \cdots
\end{align*}
which is known as Lichnerowicz-Poisson complex.
\end{enumerate}
\end{remark}

\begin{example}
Let $B\otimes_k A$ be a $A$-algebra and $B$ is a finitely generated $k$-algebra so that $\Omega_{B/k}^1$ is finitely presented. Then $Hom_{B\otimes_kA}(\wedge^p \Omega_{B\otimes_k A/A}^1,B\otimes_k A)\cong Hom_B(\wedge^p \Omega_{B/k}^1, B)\otimes _k A$. So the Schouten bracket $[-,-]_{B\otimes_k A}$ on $Hom_B(\wedge^p\Omega_{B/k}^1,B)\otimes A$ over $A$ can be seen as
\begin{align*}
[P\otimes a,Q\otimes b]_{B\otimes_k A}=[P,Q]_B\otimes ab
\end{align*}
\end{example}

We can globalize a Poisson algebra $(A,\Lambda)$ over $R$ to define a Poisson scheme over some base scheme. We note that we can globalize the Schouten bracket, and so characterize a Poisson scheme over some base scheme. (for the detail, see the third part of the author's Ph.D thesis \cite{Kim14})

\begin{definition}
Let $f:X\to S$ be a morphism of $k$-schemes. There is an operation 
\begin{align*}
[-,-]:\mathscr{H}om_{\mathcal{O}_X}(\wedge^p \Omega_{X/S}^1,\mathcal{O}_X )\times \mathscr{H}om_{\mathcal{O}_X}(\wedge^q \Omega_{X/S}^1,\mathcal{O}_X )\to\mathscr{H}om_{\mathcal{O}_X}(\wedge^{p+q-1} \Omega_{X/S}^1,\mathcal{O}_X )
\end{align*}
 which is called the Schouten bracket on a scheme $X$ over $S$.
\end{definition}

\begin{remark}\label{remarkc}
Let $f:X\to S$ be a morphsim of $k$-schemes. The following are equivalent.
\begin{enumerate}
\item $X$ is a Poisson scheme over $S$.
\item There exists a global section $\Lambda\in \Gamma(X,\mathscr{H}om_{\mathcal{O}_X}(\wedge^2 \Omega_{X/S}^1,\mathcal{O}_X))$ with $[\Lambda,\Lambda]=0$.
\end{enumerate}
We will denote the Poisson scheme by $(X,\Lambda)$.
\end{remark}

\begin{definition}
Let $(X,\Lambda_0)$ be an algebraic Poisson scheme over $S$. Then we define Lichnerowicz-Poisson complex by the following complex of sheaves
\begin{align*}
\mathcal{O}_X\xrightarrow{[\Lambda_0,-]} \mathscr{H}om_{\mathcal{O}_X}(\Omega_{X/S}^1,\mathcal{O}_X)\xrightarrow{[\Lambda_0,-]} \mathscr{H}om_{\mathcal{O}_X}(\wedge^2 \Omega^1_{X/S},\mathcal{O}_X)\xrightarrow{[\Lambda_0,-]}\mathscr{H}om_{\mathcal{O}_X}(\wedge^3\Omega^1_{X/S},\mathcal{O}_X)\xrightarrow{[\Lambda_0,-]}\cdots
\end{align*}
We define $i$-th shifted $($by $1$$)$ truncated Lichnerowicz-Poisson complex by the following complex of sheaves
\begin{align*}
\mathscr{H}om_{\mathcal{O}_X}(\Omega_{X/S}^1,\mathcal{O}_X)\xrightarrow{[\Lambda_0,-]} \mathscr{H}om_{\mathcal{O}_X}(\wedge^2 \Omega^1_{X/S},\mathcal{O}_X)\xrightarrow{[\Lambda_0,-]}\mathscr{H}om_{\mathcal{O}_X}(\wedge^3\Omega^1_{X/S},\mathcal{O}_X)\xrightarrow{[\Lambda_0,-]}\cdots
\end{align*}
\end{definition}

\begin{remark}\label{remarkd}
Let $(X,\Lambda_0)$ be a nonsingular Poisson variety over $k$. In this case we denote $\mathscr{H}om_{\mathcal{O}_X}(\wedge^i\Omega_{X/k}^1,\mathcal{O}_X)$ by $\wedge^i T_X$ so that the Lichnerowicz-Poisson complex is
\begin{align*}
\mathcal{O}_X^\bullet:\mathcal{O}_X\xrightarrow{[\Lambda_0,-]}T_X\xrightarrow{[\Lambda_0,-]} \wedge^2 T_X\xrightarrow{[\Lambda_0,-]}\cdots
\end{align*}
We will denote its $i$-th hypercohomology group by $\mathbb{H}^i(X,\Lambda_0,\mathcal{O}_X^\bullet)$. On the other hand, shiftd truncated Lichnerowicz-Poisson complex is
\begin{align*}
T_X^\bullet:T_X\xrightarrow{[\Lambda_0,-]}\wedge^2 T_X\xrightarrow{[\Lambda_0,-]} \wedge^3 T_X\xrightarrow{[\Lambda_0,-]}\cdots
\end{align*}
We will denote its $i$-th hypercohomology group by $\mathbb{H}^i(X,\Lambda_0,T_X^\bullet)$.
\end{remark}

\section{Deformations of algebraic Poisson schemes}\label{section2}
We denote by $\bold{Art}$ the category of local artinian $k$-algebras with residue field $k$, where $k$ is an algebraically closed field with characteristic $0$.
\begin{definition}[small extension]
We say that for $(\tilde{A},\tilde{\mathfrak{m}}), (A,\mathfrak{m})\in \bold{Art}$, an exact sequence of the form $0\to (t)\to \tilde{A}\to A\to 0$ is a small extension if $t\in \tilde{\mathfrak{m}}$ is annihilated by $\tilde{\mathfrak{m}}$ $($i.e, $t\cdot \tilde{\mathfrak{m}}=0)$ so that $(t)$ is an one dimensional $k$-vector space.
\end{definition}

\begin{lemma}[compare \cite{Ser06} Lemma 1.2.6 page 26]\label{3l}
Let $B_0$ be a Poisson $k$-algebra with the Poisson structure $\Lambda_0\in Hom_{B_0}(\wedge^2 \Omega_{B_0/k},B_0)$, and
\begin{align*}
e:0\to(t)\to \tilde{A}\to A\to 0
\end{align*}
a small extension in $\bold{Art}$. Let $\Lambda\in Hom_{B_0}(\wedge^2 \Omega_{B_0/k}^1,B_0)\otimes_k A$ be a Poisson structure on $B_0\otimes_k A$ over $A$  inducing $\Lambda_0$. Let $\Lambda_1,\Lambda_2\in Hom_{B_0}(\wedge^2 \Omega_{B_0/k}^1,B_0)\otimes_k \tilde{A}$ be skew-symmetric biderivations on $B_0\otimes_k \tilde{A}$ over $\tilde{A}$ which induces $\Lambda$. This implies that there exists a $\Lambda'\in Hom_{B_0}(\wedge^2 \Omega_{B_0/k}^1,B_0)$ such that $\Lambda_1-\Lambda_2=t\Lambda'$. Then there is one to one correspondence
\begin{align*}
\{\text{isomorphisms between} \,\,\,(B_0\otimes_k \tilde{A},\Lambda_1)\,\,\,\text{and}\,\,\,(B_0\otimes_k\tilde{A},\Lambda_2)\,\,\,\text{inducing the identity on}\,\,\,(B_0\otimes_k A,\Lambda)\}\\\to \{P\in Der_k(B_0,B_0)=Hom_{B_0}(\Omega_{B_0/k}^1,B_0)| \Lambda'-[\Lambda_0, P]=\Lambda'+[P,\Lambda_0]=0\}
\end{align*}
In particular, when $\Lambda_1=\Lambda_2$, there is a canonical isomorphism of groups
\begin{align*}
\{\text{automorphisms on} \,\,\,(B_0\otimes_k \tilde{A},\Lambda_1)\,\,\,\text{inducing the identity on}\,\,\,(B_0\otimes_k A,\Lambda)\}\to PDer_k(B_0,B_0)
\end{align*}
\end{lemma}
\begin{proof}
Let $\theta:(B_0\otimes_k \tilde{A},\Lambda_1)\to (B_0\otimes_k \tilde{A}, \Lambda_2)$ be an isomorphism compatible with skew symmetric biderivations which induces the identity on $(B_0\otimes_k A,\Lambda)$ so that $\theta$ is $\tilde{A}$-linear and induces the identity modulo by $t$. We have $\theta(x)=x+tPx$, where $P\in Der_{\tilde{A}}(B_0\otimes_k \tilde{A},B_0)=Der_k(B_0,B_0)=Hom_{B_0}(\Omega_{B_0/k}^1,B_0)$. When we think of $P$ as an element of $Hom_{B_0}(\Omega_{B_0/k}^1,B_0)$, we have $\theta(x)=x+tP(dx)$. We define the correspondence by $\theta \mapsto P$. Now we check that $\Lambda'-[\Lambda_0,P]=0$. Since $\theta$ is compatible with skew-symmetric biderivations $\Lambda_1,\Lambda_2$, for $x,y\in B_0$, we have by Example $(\ref{3ex})$,
\begin{align*}
&\theta(\Lambda_1(dx\wedge dy))=\Lambda_2(d(\theta x)\wedge d (\theta y))\\
&\Lambda_1(dx\wedge dy)+t P(d(\Lambda_1(dx\wedge dy)))=\Lambda_2((dx+td(P(dx)))\wedge (dy+td(P(dy))))\\
&\Lambda_1(dx\wedge dy)+t P(d(\Lambda_0(dx\wedge dy)))=\Lambda_2(dx\wedge dy)+t\Lambda_0(dx\wedge d(P(dy)))+t\Lambda_0(d(P(dx))\wedge dy)\\
&t[\Lambda'(dx\wedge dy)+P(d(\Lambda_0(dx\wedge dy)))-\Lambda_0(dx\wedge d(P(dy)))-\Lambda_0(d(P(dx))\wedge dy)]=0\\
&\Lambda'-[\Lambda_0,P]=0
\end{align*}
Since $\theta$ is determined by $P$, the correspondence is one to one.

Now we assume that $\Lambda_1=\Lambda_2$. So $\theta$ corresponds to $P$ with $[\Lambda_0,P]=0$. First we note that $P\in Hom_{B_0}(\Omega_{B_0/k}^1,B_0)$ with $[\Lambda,P]=0$ is a Poisson derivation. i.e $P\in PDer_k(B_0,B_0)$. In other words, $P(\{x,y\})=\{Px,y\}+\{x,Py\}$. Indeed, $0=[\Lambda_0,P](dx\wedge dy)=\Lambda_0(d(Px)\wedge dy)-\Lambda_0(d(Py)\wedge dx)-P(d(\Lambda_0(dx\wedge dy))$.

We show that the correspondence is a group isomorphism. Indeed, let $\theta(x)=x+tPx$ and $\sigma(y)=y+tQy$ with $[\Lambda_0,P]=[\Lambda_0,Q]=0$. Then $\sigma(\theta(x))=\theta(x)+tQ(\theta(x))=x+tPx+tQ(x+tPx)=x+tPx+tQx=x+t(P+Q)x$. Hence $\sigma\circ \theta$ corresponds to $P+Q$. Since $[\Lambda_0,P+Q]=0$ and identity map corresponds to $0$, the correspondence is a group isomorphism.
\end{proof}

\begin{lemma}\label{3n}
 Let $B_0$ be a $k$-algebra with $T_0\in Hom_{B_0}(\Omega_{B_0/k}^1,B_0)$, and $e:0\to (t)\to\tilde{A}\to A\to 0$ a small extension in $\bold{Art}$. Let $T\in Hom_{B_0}(\Omega_{B_0/k}^1,B_0)\otimes A$ inducing $T_0$, and $T_1,T_2\in  Hom_{B_0}(\Omega_{B_0/k},B_0)\otimes \tilde{A}$ which induce $T$. This implies that there is a $T'\in Hom_{B_0}(\Omega_{B_0/k}^1,B_0)$ such that $T_1-T_2=tT'$. Then there is one to one correspondence
 \begin{align*}
\{\text{isomorphisms between} \,\,\,(B_0\otimes_k \tilde{A},T_1)\,\,\,\text{and}\,\,\,(B_0\otimes_k\tilde{A},T_2)\,\,\,\text{inducing the identity on}\,\,\,(B_0\otimes_k A,T)\}\\\to \{P\in Der_k(B_0,B_0)=Hom_{B_0}(\Omega_{B_0/k},B_0)| T'-[T_0, P]=T'+[P,T_0]=0\}
\end{align*}
\end{lemma}

\begin{proof}
We keep the notations in the proof of Lemma \ref{3l}. Since $\theta(T_1(dx))=T_2(d\theta(x))$, we have $T_1(dx)+tP(dT_0(dx))=T_2(dx)+tT_0(dP(dx))$, which means $T'+[P,T_0]=0$.
\end{proof}

Now we discuss deformations of algebraic Poisson schemes. All schemes will be assumed to be defined over an algebraically closed field $k$ with characteristic $0$, locally noetherian and separated. 
\begin{definition}[flat Poisson deformations, compare \cite{Ser06} and see also \cite{Nam09}, \cite{Gin04}]\label{definitionp}
Let $A\in \bold{Art}$. Let $(X,\Lambda_0)$ be an algebraic Poisson scheme over $k$. An infinitesimal $($Poisson$)$ deformation of $(X,\Lambda_0)$ over $A$ is a cartesian diagram of morphisms of schemes
\begin{center}
$\xi:$
$\begin{CD}
(X,\Lambda_0) @>i>> (\mathcal{X},\Lambda)\\
@VVV @VV{\pi}V\\
Spec(k) @>>>Spec(A)
\end{CD}$
\end{center}
where $\pi$ is flat, $(\mathcal{X},\Lambda)$ is a Poisson scheme over $Spec(A)$ with $\Lambda\in \Gamma(\mathcal{X}, \mathscr{H}om_{\mathcal{O}_{\mathcal{X}}}(\wedge^2 \Omega_{\mathcal{X}/A}^1, \mathcal{O}_{\mathcal{X}}))$ and $(X,\Lambda_0) \cong (\mathcal{X},\Lambda) \times_{Spec(A)} Spec(k)$ as a Poisson isomorphism: in other words, $\Lambda_0$ is induced from $\Lambda$ so that $(X,\Lambda_0)$ is a closed Poisson subscheme of $(\mathcal{X},\Lambda)$. $\xi$ is called a first-order deformation if $A=k[\epsilon]$. Two deformations $(\mathcal{X},\Lambda)$ and $(\mathcal{X}',\Lambda')$ of $(X,\Lambda_0)$ is isomorphic if there is a Poisson isomorphism $\phi:(\mathcal{X},\Lambda)\to (\mathcal{X}',\Lambda')$ over $Spec(A)$ inducing $(X,\Lambda_0)$. Then we can define a functor of Artin rings
\begin{align*}
 Def_{(X,\Lambda_0)}:\bold{Art}&\to (sets)\\
                              A&\mapsto \{\text{infinitesimal deformations of $(X,\Lambda_0)$ over $A$}\}/\text{isomorphism}
\end{align*} 
\end{definition} 

\begin{definition}[trivial Poisson deformations]
Let $(X,\Lambda_0)$ be an algebraic Poisson scheme over $k$. An infinitesimal deformation of $(X,\Lambda_0)$ over $A\in \bold{Art}$ is called trivial if it is isomorphic to the following infinitesimal deformation 
\begin{center}
$\begin{CD}
(X,\Lambda_0) @>>> (X\times_{Spec(k)} Spec(A), \Lambda_0)\\
@VVV @VV{\pi}V\\
Spec(k) @>>>Spec(A)
\end{CD}$
\end{center}
\end{definition}

\begin{definition}[rigid Poisson deformations]
An algebraic Poisson scheme $(X,\Lambda_0)$ over $k$ is called rigid if every infinitesimal Poisson deformation of $(X,\Lambda_0)$ over $A$ is trivial for every $A$ in $\bold{Art}$.
\end{definition}

\begin{proposition}[compare \cite{Ser06} Proposition 1.2.9 page 29 and see also \cite{Nam09} Proposition 8]\label{3q}
Let $(X,\Lambda_0)$ be a nonsingular Poisson variety with $\Lambda_0\in \Gamma(X,\wedge^2 T_X)$. There is a canonical isomorphism
\begin{align*}
Def_{(X,\Lambda_0)}(k[\epsilon])=\{\text{first order Poisson deformations of $(X,\Lambda_0)$/isomorphism}\} \xrightarrow{\kappa} \mathbb{H}^1(X,\Lambda_0,T_X^\bullet)
\end{align*}

such that $\kappa(\xi)=0$ if and only $\xi$ is the trivial Poisson deformation class. 
\end{proposition}

\begin{proof}
Given a first-order Poisson deformation of a nonsingular Poisson variety $(X,\Lambda_0)$, 
\begin{center}
$\begin{CD}
(X,\Lambda_0)@>>> (\mathcal{X},\Lambda)\\
@VVV @VVV\\
Spec(k)@>>> Spec(k[\epsilon])
\end{CD}$
\end{center}
we choose an affine open covering $\mathcal{U}=\{U_i\}$ of $X$ such that $\mathcal{X}|_{U_i}\cong U_i\times Spec(k[\epsilon])$ is trivial for all $i$ with the induced Poisson structure $\Lambda_0+\epsilon\Lambda_i\in \Gamma(U_i,T_X)\otimes_k k[\epsilon]$ on $U_i\times Spec(k[\epsilon])$ from $\Lambda$. For each $i$, we have a Poisson isomorphism
\begin{align*}
\theta_i:(U_i\times Spec(k[\epsilon]),\Lambda_0+\epsilon\Lambda_i)\to (\mathcal{X}|_{U_i},\Lambda|_{U_i})
\end{align*}
Then for each $i,j$, $\theta_{ij}:=\theta_j^{-1}\theta_i:(U_{ij}\times Spec(k[\epsilon]),\Lambda_0+\epsilon\Lambda_i)\to (U_{ij}\times Spec(k[\epsilon]), \Lambda_0+\epsilon\Lambda_j)$ is a Poisson isomorphism inducing the identity on $(U_{ij},\Lambda_0)$ by modulo $\epsilon$. Hence by Lemma \ref{3l}, $\theta_{ij}$ corresponds to $Id+\epsilon p_{ij}:(\mathcal{O}_X(U_{ij})\otimes k[\epsilon],\Lambda_0+\epsilon \Lambda_j)\to (\mathcal{O}_X(U_{ij})\otimes k[\epsilon],\Lambda_0+\epsilon \Lambda_i)$ where $p_{ij}\in \Gamma(U_{ij}, T_X)$ where $T_X=\mathscr{H}om_{\mathcal{O}_X}(\Omega_{X/k}^1,\mathcal{O}_X)=Der_k(\mathcal{O}_X,\mathcal{O}_X)$ such that $\Lambda_j-\Lambda_i-[\Lambda_0,p_{ij}]=0$. We claim that $(\{p_{ij}\},\{-\Lambda_i\})\in C^1(\mathcal{U},T_X)\oplus C^0(\mathcal{U},\wedge^2 T_X)$ is a $1$-cocycle in the following diagram (here $\delta$ is the \v{C}ech map)
\begin{center}
$\begin{CD}
C^0(\mathcal{U},\wedge^3 T_X)\\
@A[\Lambda_0,-]AA \\
C^0(\mathcal{U},\wedge^2 T_X)@>\delta>> C^1(\mathcal{U},T_X)\\
@A[\Lambda_0,-]AA @A[\Lambda_0,-]AA \\
C^0(\mathcal{U},T_X)@>-\delta>>C^1(\mathcal{U},T_X)@>\delta>>C^2(\mathcal{U},T_X)
\end{CD}$
\end{center}

Since $[\Lambda_0+\epsilon \Lambda_i,\Lambda_0+\epsilon\Lambda_i]=0$, we have $[\Lambda_0,-\Lambda_i]=0$. Since on each $U_{ijk}$ we have $\theta_{ij}\theta_{jk}\theta_{ik}^{-1}=Id_{U_{ijk}\times Spec(k[\epsilon])}$, we have  $(Id+\epsilon p_{ij})\circ(Id+\epsilon p_{jk})\circ (Id+\epsilon p_{ki})=Id$ so that $p_{ij}+p_{jk}-p_{ik}=0$, and so $\delta(\{p_{ij}\})=0$. Since $\Lambda_j-\Lambda_i-[\Lambda_0,p_{ij}]=0$, we have $\delta(\{-\Lambda_i\})+[\Lambda_0,\{p_{ij}\}]=0$. Hence $(\{p_{ij}\},\{-\Lambda_i\})$ defines an element in $\mathbb{H}^1(X,\Lambda,T_X^\bullet)$.

Now we show that for two equivalent Poisson deformations of $(X,\Lambda_0)$, associated $1$-cocycles are equivalent. If we have another Poisson deformation
\begin{center}
$\begin{CD}
(X,\Lambda_0)@>>> (\mathcal{X}',\Lambda')\\
@VVV @VVV\\
Spec(k)@>>> Spec(k[\epsilon])
\end{CD}$
\end{center}
which induces a $1$-cocycle $(\{p_{ij}'\},\{-\Lambda_i'\})$ and $\Phi:(\mathcal{X},\Lambda)\to (\mathcal{X'},\Lambda')$ is a Poisson isomorphism of deformations, then for each $i$, there is an induced Poisson isomorphism:
\begin{align*}
\alpha_i:(U_i\times Spec(k[\epsilon]),\Lambda_0+\epsilon\Lambda_i) \xrightarrow{\theta_i} (\mathcal{X}|_{U_i},\Lambda|_{U_i} )\xrightarrow{\Phi|_{U_i}} (\mathcal{X}'|_{U_i},\Lambda'|_{U_i} )\xrightarrow{\theta_i^{'-1}} (U_i\times Spec(k[\epsilon]),\Lambda_0+\epsilon\Lambda_i')
\end{align*}
Then $\alpha_i$ corresponds to $a_i\in \Gamma(U_i,T_X)$ such that $\Lambda_i'-\Lambda_i-[\Lambda_0,a_i]=0$ by Lemma \ref{3l} . We have $\theta_i'\alpha_i=\Phi|_{U_i}\theta_i$ and therefore $(\theta_j'\alpha_j)^{-1}(\theta_i'\alpha_i)=\theta_j^{-1}\Phi|_{U_{ij}}^{-1}\Phi|_{U_{ij}}\theta_i=\theta_{ij}$ so that we have $\alpha_j^{-1}\theta_{ij}'\alpha_i=\theta_{ij}$ Hence $(Id+\epsilon a_i)(Id+\epsilon p_{ij}')(Id-\epsilon a_j')=Id+\epsilon p_{ij}$ which means $a_i-a_j=p_{ij}-p_{ij}'$.

 Since $-\delta(\{a_i\})=a_i-a_j=p_{ij}-p_{ij}'$ and $\Lambda'_i - \Lambda_i=[\Lambda_0,a_i]$, $(\{p_{ij}\},\{-\Lambda_i\})$ and $(\{p_{ij}'\},\{-\Lambda_i'\})$ are cohomologous. 

Now we define an inverse map. Given an element in $\mathbb{H}^1(X,\Lambda_0,T_X^\bullet)$, we represent it by a \v{C}ech $1$-cocylce $(\{p_{ij}\}, \{-\Lambda_i\})$ for an affine open cover $\mathcal{U}=\{U_i\}$ of $X$. So we have $[\Lambda_0,-\Lambda_i]=0$, $p_{ij}+p_{jk}-p_{ik}=0$ and $\Lambda_i-\Lambda_j=[\Lambda_0,p_{ij}]=0$. By reversing the above process, the cohomology class gives a glueing condition to make a Poisson deformation of $(X,\Lambda_0)$.
\end{proof}

\begin{definition}
Let $(X,\Lambda_0)$ be a nonsingular Poisson variety. Consider a small extension
\begin{align*}
e:0\to (t)\to \tilde{A}\to A\to 0
\end{align*}
in $\bold{Art}$. let
\begin{center}
$\xi:
\begin{CD}
(X,\Lambda_0)@>>> (\mathcal{X},\Lambda)\\
@VVV @VVV\\
Spec(k)@>>> Spec(A)
\end{CD}$
\end{center}
be an infinitesimal Poisson deformation of $(X,\Lambda_0)$ over $A$. A lifting of $\xi$ to $\tilde{A}$ is an infinitesimal Poisson deformation $\tilde{\xi}$ over $\tilde{A}$ 
\begin{center}
$\tilde{\xi}:
\begin{CD}
(X,\Lambda_0)@>>> (\tilde{\mathcal{X}},\tilde{\Lambda})\\
@VVV @VVV\\
Spec(k)@>>> Spec(\tilde{A})
\end{CD}$
\end{center}
inducing $\xi$ up to isomorphism.
\end{definition}

\begin{proposition}[compare \cite{Ser06} Proposition 1.2.12]\label{lifting}
Let $(X,\Lambda_0)$ be a nonsingular Poisson variety. Let $A\in \bold{Art}$ and an infinitesimal Poisson deformation $\xi=(\mathcal{X},\Lambda)$ of $(X,\Lambda_0)$ over $A$. To every small extension $e:0\to (t)\to \tilde{A}\to A\to 0$,  we can associate an element $o_{\xi}(e)\in \mathbb{H}^2(X,\Lambda_0,T_X^\bullet)$ called the obstruction lifting of $\xi$ to $\tilde{A}$, which is $0$ if and only if a lifting of $\xi$ to $\tilde{A}$ exists.
\end{proposition}

\begin{proof}
Let $\mathcal{U}=\{U_i\}$ be an affine open covering of $X$ such that we have Poisson isomorphisms $\theta_i:(U_i\times Spec(A),\Lambda_i)\to (\mathcal{X}|_{U_i},\Lambda|_{U_i})$, where $\Lambda_i\in \Gamma(U_i,\wedge^2 T_X)\otimes A$ with $[\Lambda_i,\Lambda_i]=0$, and $\theta_{ij}:=\theta_j^{-1}\theta_i$ is a Poisson isomorphism with $\theta_{ij}\theta_{jk}=\theta_{ik}$ on $U_{ijk}\times Spec(A)$. To give a lifting $\tilde{\xi}$ of $\xi$ to $\tilde{A}$ is equivalent to give a collection of $\{\tilde{\Lambda}_i\}$ where $\tilde{\Lambda}_i \in \Gamma(U_i,\wedge^2 T_X)\otimes_k \tilde{A}$ with $[\tilde{\Lambda}_i,\tilde{\Lambda}_i]=0$ is a Poisson structure on $U_i\times Spec(\tilde{A})$ and a collection of Poisson isomorphisms $\{\tilde{\theta}_{ij}\}$ where $\tilde{\theta}_{ij}:(U_{ij}\times Spec(\tilde{A}),\tilde{\Lambda}_i)\to (U_{ij}\times Spec(\tilde{A}),\tilde{\Lambda}_j)$ such that
\begin{enumerate}
\item $\tilde{\theta}_{ij}\tilde{\theta}_{jk}=\tilde{\theta}_{ik}$ as a Poisson isomorphism.
\item $\tilde{\theta}_{ij}$ restricts to $\theta_{ij}$ on $U_{ij}\times Spec(A)$.
\item $\tilde{\Lambda}_i$ restrits to $\Lambda_i$.
\end{enumerate}

From such data, we can glue together $(U_i\times Spec(\tilde{A}),\tilde{\Lambda}_i)$ to make a Poisson deformation $(\tilde{\mathcal{X}},\tilde{\Lambda})$ inducing $(\mathcal{X},\Lambda)$. Now given a Poisson deformation $\xi=(\mathcal{X},\Lambda)$ over $A$ and a small extension $e:0\to (t)\to \tilde{A}\to A\to0$, we associate an element $o_{\xi}(e)\in \mathbb{H}^2(X,\Lambda_0,T_X^\bullet)$. Choose arbitrary automorphisms $\{\tilde{\theta}_{ij}\}$ satisfying $(2)$ (for the existence of lifting, see \cite{Ser06} Lemma 1.2.8) and arbitrary $\tilde{\Lambda}_i\in \Gamma(U_i,\wedge^2 T_X)$ satisfying $(3)$ (not necessarily $[\tilde{\Lambda}_i,\tilde{\Lambda}_i]=0$). The lifting exists since $\Gamma(U_i,\wedge^2 T_X)\otimes_k \tilde{A} \to \Gamma(U_i,\wedge^2 T_X)\otimes_k A$ is surjective. Let $\tilde{\theta}_{ijk}=\tilde{\theta}_{ij}\tilde{\theta}_{jk}\tilde{\theta}_{ik}^{-1}$. Since $\tilde{\theta}_{ijk}$ is an automorphism on $U_{ijk}\times Spec(\tilde{A})$ inducing the identity on $U_{ijk}\times Spec(A)$, $\tilde{\theta}_{ijk}$ corresponds to $\tilde{d}_{ijk}\in \Gamma(U_{ijk},T_X)$ and $\tilde{d}_{jkl}-\tilde{d}_{ikl}+\tilde{d}_{ijl}-\tilde{d}_{jkl}=0$. So we have $-\delta(\{\tilde{d}_{ijk}\})=0$. Since $[\tilde{\Lambda}_i,\tilde{\Lambda}_i]$ is zero  modulo $(t)$ by $[\Lambda_i,\Lambda_i]=0$, there exists $\Pi_i\in \Gamma(U_i,\wedge^3 T_X)$ such that $[\tilde{\Lambda}_i,\tilde{\Lambda}_i]=t\Pi_i$. Since $0=[\tilde{\Lambda}_i,[\tilde{\Lambda}_i,\tilde{\Lambda}_i]]=[\tilde{\Lambda}_i,t\Pi_i]=t[\Lambda_0,\Pi_i]=0$, we have $[\Lambda_0,\Pi_i]=0$. 

 Let $\tilde{f}_{ij}: \mathcal{O}_X(U_{ij})\otimes_k \tilde{A} \to \mathcal{O}_X(U_{ij})\otimes_k \tilde{A}$ be the ring homomorphism corresponding to $\tilde{\theta}_{ij}$. We will denote by $\tilde{f}_{ij}\Lambda_j$ be the induced skew symmetric biderivation structure on $\mathcal{O}_X(U_{ij})\otimes_k \tilde{A}$ such that $\tilde{f}_{ij}:(\mathcal{O}_X(U_{ij})\otimes_k\tilde{A},\tilde{\Lambda}_j)\to (\mathcal{O}_X(U_{ji})\otimes_k\tilde{A}, \tilde{f}_{ij}\Lambda_j)$ is skew symmetric biderivation-preserving.  Since $\tilde{f}_{ij}\tilde{\Lambda}_j$ and $\tilde{\Lambda}_i$ are same modulo $(t)$ by $(3)$, there exists $\Lambda_{ij}'\in \Gamma(U_{ij},\wedge^2 T_X)$ such that $t\Lambda_{ij}'=\tilde{f}_{ij}\tilde{\Lambda}_j-\tilde{\Lambda}_i$. Then $t\Lambda_{ji}'=\tilde{f}_{ji}\Lambda_i-\Lambda_j$. By applying $\tilde{f}_{ij}$ on both sides, we have $t\Lambda_{ji}'=\tilde{\Lambda}_i-\tilde{f}_{ij}\tilde{\Lambda}_j=-t\Lambda_{ij}'$. Hence $\Lambda_{ji}'=-\Lambda_{ij}'$. We note that $t\Pi_i-t\Pi_j=t\Pi_i-\tilde{f}_{ij}(t\Pi_j)=[\tilde{\Lambda}_i,\tilde{\Lambda}_i]-\tilde{f}_{ij}[\tilde{\Lambda}_j,\tilde{\Lambda}_j]=[\tilde{\Lambda}_i,\tilde{\Lambda}_i]-[\tilde{f}_{ij}\tilde{\Lambda}_j,\tilde{f}_{ij}\tilde{\Lambda}_j]=[\tilde{\Lambda}_i,\tilde{\Lambda}_i]-[\tilde{\Lambda}_i+t\Lambda_{ij}',\tilde{\Lambda}_i+t\Lambda_{ij}']=t[\Lambda_0,-2\Lambda_{ij}']$. Hence we have $\Pi_i-\Pi_j+[\Lambda_0, 2\Lambda_{ij}']=0$. So we have $-\delta(\{\Pi_i\})+[\Lambda_0,\{2\Lambda_{ij}'\}]=0$. In the following isomorphism
\begin{align*}
\tilde{\alpha}_{ijk}:U_{ijk}\times Spec(\tilde{A}) \xrightarrow{\tilde{\theta}_{ij}}U_{ijk}\times Spec(\tilde{A}) \xrightarrow{\tilde{\theta}_{jk}}U_{ijk}\times Spec(\tilde{A})\xrightarrow{\tilde{\theta}_{ki}}U_{ijk}\times Spec(\tilde{A})
\end{align*}
which corresponds to a $\tilde{d}_{ijk} \in \Gamma(U_{ijk},T_X)$. Then we have
\begin{align*}
Id+t\tilde{d}_{ijk}:\mathcal{O}_X(U_{ijk})\otimes_k \tilde{A} \xrightarrow{\tilde{f}_{ki}}\mathcal{O}_X(U_{ijk})\otimes_k \tilde{A} \xrightarrow{\tilde{f}_{jk}}\mathcal{O}_X(U_{ijk})\otimes_k \tilde{A} \xrightarrow{\tilde{f}_{ij}} \mathcal{O}_X(U_{ijk})\otimes_k \tilde{A}
\end{align*}
 $Id+t\tilde{d}_{ijk}: (\mathcal{O}_X(U_{ijk})\otimes \tilde{A} ,\tilde{\Lambda}_i) \to (\mathcal{O}_X(U_{ijk})\otimes_k \tilde{A} ,\tilde{f}_{ij}\tilde{f}_{jk}\tilde{f}_{ki}\tilde{\Lambda}_i)$ is an isomorphism compatible with skew-symmetric bidervations. We note that $\tilde{\Lambda}_i-\tilde{f}_{ij}\tilde{f}_{jk}\tilde{f}_{ki}\tilde{\Lambda}_i=\tilde{\Lambda}_i-\tilde{f}_{ij}\tilde{f}_{jk}(\tilde{\Lambda}_k+t\Lambda_{ki}')=\tilde{\Lambda}_i-\tilde{f}_{ij}(\tilde{\Lambda}_j+t\Lambda'_{jk}+\Lambda_{ki}')=\tilde{\Lambda}_i-(\tilde{\Lambda}_i+t\Lambda_{ij}'+t\Lambda_{jk}'+t\Lambda_{ki}')=-t(\Lambda_{ij}'+\Lambda_{jk}'+\Lambda_{ki}')$.  Hence by Lemma \ref{3l}, we have $-(\Lambda_{ij}'+\Lambda_{jk}'+\Lambda_{ki}')-[\Lambda_0, \tilde{d}_{ijk}]=0$. So we have $-\delta(\{\Lambda_{ij}'\})+[\Lambda_0,-\{\tilde{d}_{ijk}\}]=0$. Hence $\alpha=(\{\Pi_i\}, \{2\Lambda_{ij}'\},\{-2\tilde{d}_{ijk}\})\in C^0(\mathcal{U},\wedge^3 T_X)\oplus C^1(\mathcal{U},\wedge^2 T_X)\oplus C^2(\mathcal{U},T_X)$ is a $2$-cocyle in the following diagram 
\begin{center}
$\begin{CD}
\mathcal{C}^0(\mathcal{U},\wedge^4 T_X)\\
@A[\Lambda_0,-]AA\\
\mathcal{C}^0(\mathcal{U},\wedge^3 T_X)@>-\delta>>\mathcal{C}^1(\mathcal{U},\wedge^3 T_X)\\
@A[\Lambda_0,-]AA @A[\Lambda_0,-]AA\\
C^0(\mathcal{U},\wedge^2 T_X)@>\delta>> C^1(\mathcal{U},\wedge^2 T_X)@>-\delta>>\mathcal{C}^2(\mathcal{U},\wedge^2 T_X)\\
@A[\Lambda_0,-]AA @A[\Lambda_0,-]AA @A[\Lambda_0,-]AA\\
C^0(\mathcal{U},T_X)@>-\delta>>C^1(\mathcal{U},T_X)@>\delta>>C^2(\mathcal{U},T_X)@>-\delta>>\mathcal{C}^3(\mathcal{U},T_X)\\
\end{CD}$
\end{center}

We claim that given a different choice $\{\tilde{\theta}_{ij}' \}$ and $\{\tilde{\Lambda}_i'\}$ satisfying $(1),(2),(3)$, the associated $2$-cocycle $\beta=(\{\Pi_i'\}, \{2\Lambda_{ij}''\},\{-2\tilde{d}'_{ijk}\})\in C^0(\mathcal{U},\wedge^3 T_X)\oplus C^1(\mathcal{U},\wedge^2 T_X)\oplus C^2(\mathcal{U},T_X)$ is cohomologous to the $2$-cocycle $\alpha$ associated with $\{\tilde{\theta}_{ij}\}$ and $\{\tilde{\Lambda}_i\}$. Let $\tilde{f}'_{ij}:\mathcal{O}_X(U_{ij})\otimes \tilde{A}\to \mathcal{O}_X(U_{ij})\otimes \tilde{A}$ corresdpong to $\tilde{\theta}_{ij}'$. Then $\tilde{f}_{ij}'=\tilde{f}_{ij}+tp_{ij}$ for some $p_{ij}\in \Gamma(U_{ij},T_X)$ \footnote{Since $\tilde{f}'_{ij}-\tilde{f}_{ij}$ is zero modulo $t$, we have $(\tilde{f}'_{ij}-\tilde{f}_{ij})(x)=0+tp_{ij}(x)$ for some map $p_{ij}$. We show that $p_{ij}$ is a derivation. Indeed,  $tp_{ij}(xy)=(\tilde{f}_{ij}-f_{ij})(xy)=\tilde{f}_{ij}(x)(\tilde{f}_{ij}-f_{ij})(y)+(\tilde{f}_{ij}-f_{ij})(x)f_{ij}(y)=\tilde{f}_{ij}(x)tp_{ij}(y)+tp_{ij}(y)f_{ij}(y)=t(xp_{ij}(y)+yp_{ij}(x))$. So $p_{ij}$ is a derivation and so an element in $\Gamma(U_{ij},T_X)$.} and $\tilde{\Lambda}'_i=\tilde{\Lambda}_i+t\Lambda_i'$ for some $\Lambda_i'\in \Gamma(U_i,\wedge^2 T_X)$. For each $i,j,k$, $\tilde{\theta}_{ij}'\tilde{\theta}_{jk}'\tilde{\theta'}_{ik}^{-1}$ corresponds to the derivation $\tilde{d}'_{ijk}=\tilde{d}_{ijk}+(p_{ij}+p_{jk}-p_{ik})$. Hence $\delta(2\{p_{ij}\})=\{-2\tilde{d}_{ijk}-(-2\tilde{d}_{ijk}')\}$. We also note that  $t\Pi_i'=[\tilde{\Lambda}'_i,\tilde{\Lambda}'_i]=[\tilde{\Lambda}_i+t\Lambda_i',\tilde{\Lambda}_i+t\Lambda_i']=[\tilde{\Lambda}_i,\tilde{\Lambda}_i]+t[2\Lambda_i',\Lambda_0]=t\Pi_i+t[2\Lambda_i',\Lambda_0]$. Hence we have $[\Lambda_0,\{-2\Lambda_i'\}]=\{\Pi_i-\Pi_i'\}$. Since $t\Lambda_{ij}'=\tilde{f}_{ij}\tilde{\Lambda}_j-\tilde{\Lambda}_i$, and $ t\Lambda_{ij}''=\tilde{f}_{ij}'\tilde{\Lambda}_j'-\tilde{\Lambda}_i'=\tilde{f}_{ij}\tilde{\Lambda}_j'+t[p_{ij},\tilde{\Lambda}_j']-\tilde{\Lambda}_i'=\tilde{f}_{ij}\tilde{\Lambda}_j+t\Lambda_j'+t[p_{ij},\Lambda_0]-\tilde{\Lambda}_i-t\Lambda_i'$, we have $\Lambda_{ij}'-\Lambda_{ij}''=-\Lambda_j'+[\Lambda_0,p_{ij}]+\Lambda_i'$. So $\delta(\{-2\Lambda_i'\})+[\Lambda_0,\{2p_{ij}\}]=2\Lambda_{ij}'-2\Lambda_{ij}''.$ Hence $(\{-2\Lambda_i'\}, \{2p_{ij}\})$ is mapped to $\alpha-\beta$ so that $\alpha$ and $\beta$ are cohomologous. So given a deformation $\xi$ and a small extension $e:0\to (t)\to \tilde{A}\to A\to 0$, we can associate an element $o_{\xi}(e):=$ the cohomology class of $\alpha \in \mathbb{H}^2(X,\Lambda_0,T_X^\bullet)$. We also note that $o_{\xi}(e)=0$ if and only if there exists a collection of $\{\tilde{\theta}_{ij}\}$ and $\{\tilde{\Lambda}_i\}$ satisfying $(2),(3)$ with $[\tilde{\Lambda}_i,\tilde{\Lambda}_i]=0$ (which means $\tilde{\Lambda}_i$ defines a Poisson structure), $\Lambda_{ij}'=0$ (which implies $\tilde{f}_{ij}\tilde{\Lambda}_j=\tilde{\Lambda}_i$) and $\tilde{d}_{ijk}=0$ (which means $(1)$) if and only if there is a lifting $\tilde{\xi}$.

\end{proof}

\begin{definition}
The Poisson deformation $\xi$ is called unobstructed if $o_{\xi}$ is the zero map, otherwise $\xi$ is called obstructed. $(X,\Lambda_0)$ is unobstructed if every infinitesimal deformation of $(X,\Lambda_0)$ is unobstructed, otherwise $(X,\Lambda_0)$ is obstructed.
\end{definition}

\begin{corollary}
A nonsingular Poisson variety $(X,\Lambda_0)$ is unobstructed if $\mathbb{H}^2(X,\Lambda_0,T_X^\bullet)=0$.
\end{corollary}

\begin{example}
Let $(X,\Lambda_0)$ be a nondegenerate Poisson $K3$ surface. In other words, $(X,\Lambda_0)$ is symplectic. Since $i_{\Lambda_0}:\Omega_X^1\to T_X$ is isomorphic, $\mathbb{H}^i(X,\Lambda_0,T_X^\bullet)$ is isomorphic to $\mathbb{H}^i(X,\Omega_X^\bullet)$, where $\Omega_X^\bullet:\Omega_X^1\xrightarrow{\partial} \wedge^2 \Omega_X^1\xrightarrow{\partial} \wedge^3\Omega_X^1\xrightarrow{\partial}\cdots.$ By using the exact sequence of complex of sheaves $0\to \Omega_X^{\bullet-1}\to\mathcal{O}_X^\bullet \to \mathcal{O}_X\to 0$, where $\Omega_X^{\bullet-1}:0\to \Omega_X^1\xrightarrow{\partial }\wedge^2 \Omega_X^1\xrightarrow{\partial}\cdots$, and $\mathcal{O}_X^\bullet:\mathcal{O}_X\xrightarrow{\partial}\Omega_X^1\xrightarrow{\partial}\wedge^2 \Omega_X^1\xrightarrow{\partial}\cdots$, we get $\mathbb{H}^1(X,\Omega_X^\bullet)=21$ and $\mathbb{H}^2(X,\Omega_X^\bullet)=0$. Hence a symplectic $K3$ surface is unobstructed. Given a cocycle $\{p_{ij}\}\in \mathcal{C}^1(\mathcal{U},\Omega_X^1)$, there exists $\{\Lambda_i \}\in \mathcal{C}^0(\mathcal{U},\wedge^2 \Omega_X^1)$ such that $\delta(\{\Lambda_i\})=\{\partial{p}_{ij}\}$ since $H^1(X,\wedge^2 \Omega_X^1)=0$. Hence we get a surjection $\mathbb{H}^1(X,\Omega_X^\bullet)\to H^1(X,\Omega_X^1) $. In other words, any first-order flat deformation extends to a first-order Poisson deformation. Similarly, for a trivial Poisson $K3$ surface, we get the same situation. 
\end{example}

\begin{proposition}\label{3rigid}
A nonsingular Poisson variety $(X,\Lambda_0)$ is rigid if and only if $\mathbb{H}^1(X,\Lambda_0,T_X^\bullet)=0$.\footnote{The author could not find any example of rigid Poisson varieties.}
\end{proposition}

\begin{proof}
Assume that $(X,\Lambda_0)$ is rigid. Since any infinitesimal Poisson deformation  (in particular, any first order Poisson deformations) are trivial, $\mathbb{H}^1(X,\Lambda_0,T_X^\bullet)=0$ by Proposition \ref{3q}. Assume that $\mathbb{H}^1(X,\Lambda_0,T_X^\bullet)=0$. First we claim that given an infinitesimal Poisson deformation $\eta$ of $(X,\Lambda_0)$ over $A\in \bold{Art}$ and a small extension $e:0\to (t)\to \tilde{A}\to A\to 0$, any two liftings $ \xi, \tilde{\xi}$ of $\eta$ to $\tilde{A}$ are equivalent. Let $\{U_i\}$ be an affine open covering of $\xi=(\mathcal{X},\Lambda)$ and $\tilde{\xi}=(\tilde{\mathcal{X}},\tilde{\Lambda})$. Let $\{\theta_i\}$ where $\theta_i :U_i\times Spec(\tilde{A})\to \mathcal{X}|_{U_i} $, $\{\Lambda_i\}$ where $\Lambda_i$ is the Poisson structure on $U_i\times Spec(\tilde{A})$ induced from from $\Lambda|_{U_i}$ and let $\theta_{ij}=\theta_j^{-1}\theta_i$. Let $\{\tilde{\theta}_i\}$  where $\tilde{\theta}_i :U_i\times Spec(\tilde{A})\to \mathcal{\tilde{X}}|_{U_i} $, $\{\tilde{\Lambda}_i\}$ where the induced Poisson structure from $\tilde{\Lambda}$ on $U_i\times Spec(\tilde{A})$ and let $\tilde{\theta}_{ij}=\tilde{\theta}_j^{-1}\tilde{\theta}_i$. Let $f_{ij}:(\mathcal{O}_X({U_{ij}})\otimes\tilde{A},{\Lambda}_j )\to (\mathcal{O}_X(U_{ij})\otimes \tilde{A},\Lambda_i)$ be the homomorphism corresponding to $\theta_{ij}$ and $\tilde{f}_{ij}:(\mathcal{O}_X(U_{ij})\otimes \tilde{A},\tilde{\Lambda}_j)\to (\mathcal{O}_X(U_{ij})\otimes \tilde{A},\tilde{\Lambda}_i)$ corresponding to $\tilde{\theta}_{ij}$. Since $\xi,\tilde{\xi}$ induce the same Poisson deformation $\eta$ over $A$, we have
\begin{center}
$\tilde{f}_{ij}=f_{ij}+tp_{ij}$,$\,\,\,\,\,\,\,\,\,\,\, \tilde{\Lambda}_i=\Lambda_i+t\Lambda_i'$
\end{center}
for some $p_{ij}\in \Gamma(U_{ij},T_X)$.
Then for all $i,j,k$ we have $p_{ij}+p_{jk}-p_{ik}=0$. Since $0=[\tilde{\Lambda}_i,\tilde{\Lambda}_i]=[\Lambda_i+t\Lambda_i',\Lambda_i+t\Lambda_i']=2t[\Lambda'_i,\Lambda_0]$, we have $[\Lambda_0,\Lambda_i']=0$. Since $f_{ij}\Lambda_j=\Lambda_i$ and $\tilde{f}_{ij}\tilde{\Lambda}_j=\tilde{\Lambda}_i$,  we have ${\Lambda}_i+t\Lambda_i'=\tilde{\Lambda}_i=\tilde{f}_{ij}\tilde{\Lambda}_j=(f_{ij}+tp_{ij})(\Lambda_j+t\Lambda_j')=\Lambda_i-t[\Lambda_0,p_{ij}]+t\Lambda_j'$. Hence we have $\Lambda_j'-\Lambda_i'+[\Lambda_0,-p_{ij}]=0$. Hence $(\{\Lambda_i'\},\{- p_{ij}\})$ defines a cocylce. Since $\mathbb{H}^1(X,\Lambda_0,T_X^\bullet)=0$, there exists $\{a_i\}\in \mathcal{C}^0(\mathcal{U},T_X)$ such that $[\Lambda_0,a_i]=\Lambda_i'$ and $a_j-a_i=p_{ij}$. Now we explicitly construct a Poisson isomorphism $(\tilde{\mathcal{X}},\tilde{\Lambda}) \cong (\mathcal{X},\Lambda)$. We define a Poisson isomorphism locally on $U_i\times Spec(\tilde{A})$, and show that each map glue together to give a Poisson isomorphism $(\tilde{\mathcal{X}},\tilde{\Lambda}) \cong (\mathcal{X},\Lambda)$. We claim that $(U_i\times Spec(\tilde{A}),\Lambda_i)\to ( U_i\times Spec(\tilde{A}), \tilde{\Lambda}_i)$ is a Poisson isomorphism induced from $Id+ta_i:(\mathcal{O}_X(U_i)\otimes_k \tilde{A},\tilde{\Lambda}_i)\to (\mathcal{O}_X(U_i)\otimes_k \tilde{A},\Lambda_i)$. The inverse map is $Id-ta_{i}$. Since $\tilde{\Lambda}_i+t[a_i,\tilde{\Lambda}_i]=\tilde{\Lambda}_i+t[a_i,\Lambda_0]=\Lambda_i+t\Lambda_i'+t[a_i,\Lambda_0]=\Lambda_i$, $Id+ta_i$ is Poisson. We show that each Poisson isomorphism $\{Id+ta_i\}$ glues together to give a Poisson isomorphism $(\tilde{\mathcal{X}},\tilde{\Lambda}) \cong (\mathcal{X},\Lambda)$. Indeed, it is sufficient to show that the following diagram commutes.
\begin{center}
$\begin{CD}
(\mathcal{O}_X(U_{ij})\otimes_k Spec(\tilde{A}), \tilde{\Lambda}_i )@> Id+ta_i >> (\mathcal{O}_X(U_{ij})\otimes_k \tilde{A},\Lambda_i)\\
@A\tilde{f}_{ij}AA@AA f_{ij}A\\
(\mathcal{O}_X(U_{ij})\otimes Spec(\tilde{A}), \tilde{\Lambda}_j)@> Id+ta_j >> (\mathcal{O}_X(U_{ij})\otimes_k \tilde{A},\Lambda_j)
\end{CD}$
\end{center}
Indeed, the diagram commutes if and only if $(Id+ta_i)\circ \tilde{f}_{ij}=f_{ij}\circ (Id+ta_j)$ if and only if $\tilde{f}_{ij}+ta_i=f_{ij}+ta_j$ if and only if $p_{ij}=a_j-a_i$. Hence there is at most one lifting of $\eta$.

Now we prove that if $\mathbb{H}^1(X,\Lambda_0,T_X^\bullet)=0$, then $(X,\Lambda_0)$ is rigid. We will prove by induction on the dimension on $(A,\mathfrak{m})\in \bold{Art}$. For $A$ with $dim_k A=2$, then any first order Poisson deformation is trivial. Let's assume that any infinitesimal Poisson deformation of $(X,\Lambda_0)$ over $A$ with $dim_k A\leq n-1$ is trivial.  Let $\xi$ be an infinitesimal Poisson deformation of $(X,\Lambda_0)$ over $A$ with  $dim_k A=n$ such that $\mathfrak{m}^{p-1}\ne 0$ and $\mathfrak{m}^p=0$. Choose an element $t\ne 0\in \mathfrak{m}^{p-1}$. Then $0\to (t)\to A\to A/(t)\to 0$ is a small extension and $dim_k A/(t)\leq n-1$. Hence induced Poisson deformation $\bar{\xi}$ over $A/(t)$ from $\xi$ is trivial by induction hypothesis. Since $\xi$ is a lifting of $\bar{\xi}$, and trivial Poisson deformation over $A$ is also a lifting of $\bar{\xi}$, $\xi$ is trivial since we have at most one lifting of $\bar{\xi}$.
\end{proof}

\begin{proposition}\label{3lm}
Let $(X,\Lambda_0)$ be nonsingular projective Poisson variety with $\mathbb{H}^0(X,\Lambda_0,T_X^\bullet)=0$. Then for any infinitesimal Poisson deformation $(\mathcal{X},\Lambda)$ of $(X,\Lambda_0)$ over $A$ for any $A\in \bold{Art}$,
\begin{align*}
Aut((\mathcal{X},\Lambda)/(X,\Lambda_0))=Id,
\end{align*}
where $Aut((\mathcal{X},\Lambda)/(X,\Lambda_0)):=$ the set of Poisson automorphisms of $(\mathcal{X},\Lambda)$ restricting to the identity Poisson automorphism of $(X,\Lambda_0)$.
\end{proposition}

\begin{proof}
We prove by the induction on the dimension of $A$. Let $dim_k\,A=1$. Then $A=k$. So we have nothing to prove. Let's assume that the proposition holds for $A$ with $dim_k\,A\leq n-1$. Let $dim_k A=n$ and $(\mathcal{X},\Lambda)$ be an infinitesimal Poisson deformation of $(X,\Lambda_0)$ over $A$. Assume that the maximal ideal $\mathfrak{m}$ of $A$ satisfies $\mathfrak{m}^{p-1}\ne 0$ and $\mathfrak{m}^p=0$. Choose $t\ne 0\in \mathfrak{m}^{p-1}$. Then $A/(t) \in \bold{Art}$ with $dim_k\,A/(t)\leq n-1$ and $0\to (t)\to A\to A/(t)\to 0$ is a small extension. Now let $g:(\mathcal{X},\Lambda)\to (\mathcal{X},\Lambda)$ be a Poisson automorphism restricting to the identity Poisson automorphism of $(X,\Lambda_0)$. Let $\{U_i\}$ be an affine cover of $(\mathcal{X},\Lambda)$. Let $\{\theta_i\}$ where $\theta_i:U_i\times Spec({A})\to \mathcal{X}|_{U_i}$, $\{\Lambda_i\}$ where $\Lambda_i$ is the Poisson structure on $U_i\times Spec(A)$ induced from $\Lambda|_{U_i}$ via $\theta_i$ and let $\theta_{ij}={\theta}_j^{-1}{\theta}_i$ which corresponds to a Poisson homomorphism ${f}_{ij}:(\mathcal{O}_X(U_{ij})\otimes_k {A},{\Lambda}_j)\to (\mathcal{O}_X(U_{ij})\otimes_k {A}, {\Lambda}_i)$. Then $f$ can be described by the data $\{g_i\}$, where $g_{i}: (\mathcal{O}_X(U_i)\otimes A,\Lambda_i)  \to (\mathcal{O}_X(U_i)\otimes A,\Lambda_i)$ which is a Poisson automorphism and $g_i\cdot f_{ij}=f_{ij}\cdot g_j$.
Since by the induction hypothesis, $g_i$ induce the identity on $\mathcal{O}_X(U_i)\otimes A/(t)$. $g_i$ is of the form $g_i=Id+td_i$, where $d_i\in Der_k(\mathcal{O}_X(U_i),\mathcal{O}_X(U_i))$ with $d_i=d_j$ and $[\Lambda_0,d_i]=0$ by Lemma \ref{3l}. Hence $\{d_i\}\in \mathbb{H}^0(X,\Lambda_0,T_X^\bullet)$. Since $\mathbb{H}^0(X,\Lambda_0,T_X^\bullet)=0$, we have $d_i=0$. Hence $g_i$ is the identity. So $g$ is the identity. This proves the proposition \ref{3lm}.
\end{proof}

\begin{proposition}\label{prorepre1}
Let $e:0\to (t)\to \tilde{A}\xrightarrow{\mu} A\to 0$ be a small extension in $\bold{Art}$. Let $p:=Def_{(X,\Lambda_0)}(\mu):Def_{(X,\Lambda_0)}(\tilde{A})\to Def_{(X,\Lambda_0)}(\tilde{A})$. Then Given $\xi=(\mathcal{L},\nabla_{L})\in Def_{(L,\nabla)}(A)$, there is a transitive action of $\mathbb{H}^1(X,\Lambda_0,\mathcal{O}_X^\bullet)$ on $p^{-1}(\xi)$. Moreover, if $\mathbb{H}^0(X,\Lambda_0,T_X^\bullet)=0$, the action is free.
\end{proposition}

\begin{proof}
We will define a group action $G:\mathbb{H}^1(X,\Lambda_0,T_X^\bullet)\times p^{-1}(\xi)\to p^{-1}(\xi)$. Let $\tilde{\xi}=(\tilde{\mathcal{X}},\tilde{\Lambda})$ be a lifting of $\xi=(\mathcal{X},\Lambda)$ which is represented by $\tilde{f}_{ij}:(\mathcal{O}_X\otimes_k \tilde{A},\tilde{\Lambda}_i)\to (\mathcal{O}_X(U_{ij})\otimes_k \tilde{A},\tilde{\Lambda}_j)$ for an affine open cover $\{U_i\}$. Let $v=(\{-\Lambda_i'\},\{p_{ij}\} )\in \mathbb{H}^1(X,\Lambda_0,T_X^\bullet)$. Then we define $G(v,\tilde{\xi}):=\tilde{\xi}'$ which is represented by $\tilde{f}_{ij}+tp_{ij}$ and $\Lambda_i+t\Lambda_i'$. Then we can show that $G$ is well-defined and transitive. If $\mathbb{H}^0(X,\Lambda_0,T_X^\bullet)=0$, the action is free by Proposition \ref{3lm}. For the detail, we refer to the third part of the author's Ph.D thesis \cite{Kim14}.
\end{proof}

\begin{theorem}\label{ppro}
Let $(X,\Lambda_0)$ be a nonsingular projective Poisson variety with $\mathbb{H}^0(X,\Lambda_0,T_X^\bullet)=0$. Then the functor $Def_{(X,\Lambda_0)}$ is pro-representable.
\end{theorem}
\begin{proof}
We can check Schlessinger's criterion $(H_0),(H_1),(H_2)$, and since $X$ is projective, $\mathbb{H}^1(X,\Lambda_0,T_X^\bullet)$ is finite-dimensional so that $(H_3)$ is satisfied. Since $\mathbb{H}^0(X,\Lambda_0,T_X^\bullet)=0$, $(H_4)$ follows from Proposition \ref{prorepre1}.
\end{proof}

\begin{example}
For any Poisson $K3$ surface $(X,\Lambda_0)$, $\mathbb{H}^0(X,\Lambda_0,T_X^\bullet)=0$ so that $Def_{(X,\Lambda_0)}$ is pro-representable.
\end{example}

\section{Poisson invertible sheaves}\label{section3}

\begin{definition}[\cite{Pol97}]
Let $(X,\Lambda_0)$ be an algebraic Poisson scheme over $k$. A Poisson connection on a $\mathcal{O}_X$-module $\mathcal{F}$ is a $k$-linear bracket $\{-,-\}:\mathcal{O}_X\otimes_k \mathcal{F}\to \mathcal{F}$ which is a derivation in the first argument and satisfies the Leibnitz identity
\begin{align*}
&\{fg,s\}_{\mathcal{F}}=f\{g,s\}+g\{f,s\}_\mathcal{F}\\
&\{f,gs\}_\mathcal{F}=\{f,g\}s+g\{f,s\}_\mathcal{F}
\end{align*}
where $f,g\in \mathcal{O}_X$, $s$ is a local section of $\mathcal{F}$. Equivalently, a Poisson connection is given by a homomorphism $v:\mathcal{F}\to \mathscr{H}om_{\mathcal{O}_X}(\Omega_{X/k}^1,\mathcal{F})=Der(\mathcal{O}_X,\mathcal{F})$ which satisfies the identity
\begin{align*}
v(gs)=-i_{\Lambda_0}(dg)\otimes s+g\cdot v(s)=-[\Lambda_0,g]\otimes s +g\cdot v(s)
\end{align*}
where $g\in \mathcal{O}_X$. Namely, $v(s)\in Der(\mathcal{O}_X,\mathcal{F})$ is defined by the formula
\begin{align*}
v(s)f=\{f,s\}_\mathcal{F}
\end{align*}
\end{definition}

\begin{definition}[\cite{Pol97}]\label{pc}
A Poisson connection is flat if the the bracket above gives a Lie action of $\mathcal{O}_X$ on $\mathcal{F}$, where $\mathcal{O}_X$ is considered as a Lie algebra via the Poisson bracket. In other words, $\{f,\{g,s\}_\mathcal{F}\}_\mathcal{F}=\{\{f,g\},s\}_\mathcal{F}+\{g,\{f,s\}_\mathcal{F}\}_\mathcal{F}$. For given a Poisson connection $v:\mathcal{F}\to Der(\mathcal{O}_X,\mathcal{F})$, one can define a homomorphism $\tilde{v}:Der(\mathcal{O}_X,\mathcal{F})\to Der^2(\mathcal{O}_X,\mathcal{F})=\mathscr{H}om_{\mathcal{O}_X}(\wedge^2 \Omega_{X/k}^1,\mathcal{F})$ by the formula,
\begin{align*}
\tilde{v}(\delta)(f,g)=\{f,\delta(g)\}_\mathcal{F}-\{g,\delta(f)\}_\mathcal{F}-\delta(\{f,g\})=v(\delta(g))f-v(\delta(f))g-\delta(\{f,g\})
\end{align*}
Let $c(v):=\tilde{v}\circ v$. Then $c(v)$ is $\mathcal{O}_X$-linear and $v$ is flat if and only if $c(v)=\tilde{v}\circ v=0$.
\end{definition}

\begin{remark}\label{lco}
Let $(X,\Lambda_0)$ be a nonsingular Poisson variety. Given a flat Poisson connection $\nabla:\mathcal{F}\to T_X\otimes \mathcal{F}$, we can extend $v_0:=\nabla$ inductively to define
\begin{align*}
v_k:\wedge^k T_X\otimes\mathcal{F}\to \wedge^{k+1}T_X \otimes \mathcal{F}
\end{align*}
by the property
\begin{align*}
v_k(\alpha \otimes s)=-[\alpha,\Lambda_0]\otimes s+(-1)^k\alpha\wedge v_{0}(s)
\end{align*}
where $\alpha$ is a local $k$-vector and $s$ is a local section of $\mathcal{F}$. Then $v_1=\tilde{v}_0$ so that $v_1\circ v_0=0$ and so $v_{k+1}\circ v_k=0$. Hence we have a complex of sheaves
\begin{align}
\mathcal{F}^\bullet:\mathcal{F}\xrightarrow{\nabla:=v_0} T_X\otimes \mathcal{F}\xrightarrow{v_1} \wedge^2 T_X\otimes \mathcal{F}\xrightarrow{v_2} \wedge^3 T_X\otimes \mathcal{F}\xrightarrow{v_3} \cdots
\end{align}
We denote the $i$-th hypercohomology group of this complex of sheaves by $\mathbb{H}^i(X,\Lambda_0,\mathcal{F}^\bullet,\nabla)$.
\end{remark}

\begin{definition}\label{lco2}
Let $(X,\Lambda_0)$ be an algebraic Poisson scheme. An $\mathcal{O}_X$-module $\mathcal{F}$ equipped with a flat Poisson connection is called a Poisson $\mathcal{O}_X$-module. Given two Poisson $\mathcal{O}_X$-modules $(\mathcal{F},\{-,-\}_\mathcal{F})$ and $(G,\{-,-\}_\mathcal{G})$, a morphism $\alpha:\mathcal{F}\to \mathcal{G}$ of Poisson $\mathcal{O}_X$-modules is a morphism of $\mathcal{O}_X$-modules such that $\alpha(\{f,v\}_\mathcal{F})=\{f,\alpha(v)\}_\mathcal{G}$.
\end{definition}

\begin{definition}
Let $(X,\Lambda_0)$ be an algebraic Poisson scheme.  A Poisson invertible sheaf $L$ on $(X,\Lambda_0)$ is an invertible Poisson $\mathcal{O}_X$-module. In other words, $L$ is equipped with a flat Poisson connection $\nabla$. In this case, we denote the Poisson invertible sheaf by $(L,\nabla)$.
\end{definition}

\begin{remark}

 Let $(X,\Lambda_0)$ be an algebraic Poisson scheme with the Poisson bracket $\{-,-\}$. A flat Poisson connection on $\mathcal{O}_X$ is same to giving an element  $T\in \Gamma(X,\mathscr{H}om_{\mathcal{O}_X}(\Omega_{X/k}^1,\mathcal{O}_X))$ with $[\Lambda_0,T]=0$ which defines a Poisson derivation. Let $\{-,-\}_L:\mathcal{O}_X\otimes_k \mathcal{O}_X\to \mathcal{O}_X$ be a flat Poisson connection defined by $v:\mathcal{O}_X\to \mathscr{H}om_{\mathcal{O}_X}(\Omega_{X/k}^1,\mathcal{O}_X)=Der(\mathcal{O}_X,\mathcal{O}_X)$. Let $v(1)=T\in \Gamma(X,\mathscr{H}om(\Omega_{X/k}^1,\mathcal{O}_X))$ . Then $v(g)=v(g1)=-i_{\Lambda_0}(dg)\otimes 1+gT$. Hence $\{f,g\}_L=\{f,g\}+gT(f)$. Since $v$ is flat, $\tilde{v}(v(1))(f,g)=0$ so that $\{f,T(g)\}+\{T(f),g\}-T(\{f,g\})=0$. Hence $T$ is a Poisson derivation so that $[\Lambda_0,T]=0$. In particular, when $v(1)=0$, the Poisson $\mathcal{O}_X$-module structure on $\mathcal{O}_X$ is exactly the Poisson structure on $\mathcal{O}_X$.

\end{remark}

Let $(X,\Lambda_0)$ be an algebraic Poisson scheme with the Poisson bracket $\{-,-\}$. Let $\{f_{ij}\}$ be transition functions defining a Poisson invertible sheaf $(L,\nabla)$ for an open covering of $\mathcal{U}=\{U_i\}$ of $X$. The flat Poisson connection $\nabla$ on $L$, $\{-,-\}_L:\mathcal{O}_X\otimes_k L\to L$ is locally expressed as a Poisson connection on $U_i$ which is equivalent to giving an element $T_i \in \Gamma(U_i,\mathscr{H}om_{\mathcal{O}_X}(\Omega_{X/k}^1,\mathcal{O}_X))$ with $[\Lambda_0, T_i]=0$. Given a non-vanishing section $s\in \Gamma(U_{ij},L)$ which is locally expressed as $s_i$ on $U_i$ with $s_i=f_{ij}s_j$ and for any $a\in \mathcal{O}_X$, $\{a, s\}_L$ is locally expressed as $\{a,s_i\}+s_iT_ia $ on $U_i$ with $f_{ij}(\{a,s_j\}+s_jT_ja)=\{a,s_i\}+s_iT_ia$. Then
\begin{align*}
f_{ij}(\{a,s_j\}+s_jT_ja)&=\{a,s_i\}+s_iT_ia=\{a,f_{ij}s_j\}+f_{ij}s_jT_ia=\{a,f_{ij}\}s_j+f_{ij}\{a,s_j\}+f_{ij}s_j T_ia\\
&\iff f_{ij}T_ja=-\{f_{ij},a\}+f_{ij}T_ia=-[\Lambda_0,f_{ij}]a+f_{ij}T_ia\\
&\iff T_j-T_i+\frac{1}{f_{ij}}[\Lambda_0,f_{ij}]=T_j-T_i+[\Lambda_0,\log f_{ij}]=0
\end{align*}
This show that a Poisson invertible sheaf $(L,\nabla)$ gives a $1$-cocycle 
\begin{align*}
(\{T_i\},\{f_{ij}\})\in \mathcal{C}^0(\mathcal{U},\mathscr{H}om_{\mathcal{O}_X}(\Omega_{X/k}^1,\mathcal{O}_X))\oplus \mathcal{C}^1(\mathcal{U},\mathcal{O}_X^*)  
\end{align*}
in the following \v{C}ech resolution 
\begin{center}
$\begin{CD}
\mathcal{C}^0(\mathcal{U}, \mathscr{H}om_{\mathcal{O}_X}(\wedge^2 \Omega_{X/k}^1,\mathcal{O}_X))\\
@A[\Lambda_0,-]AA\\
\mathcal{C}^0(\mathcal{U},\mathscr{H}om_{\mathcal{O}_X}(\Omega_{X/k}^1,\mathcal{O}_X))@>\delta>> \mathcal{C}^1(\mathcal{U},\mathscr{H}om_{\mathcal{O}_X}(\Omega_{X/k}^1,\mathcal{O}_X))\\
@A[\Lambda_0,\log-]AA @A[\Lambda_0,\log-]AA\\
\mathcal{C}^0(\mathcal{U},\mathcal{O}_X^*)@>-\delta>> \mathcal{C}^1(\mathcal{U},\mathcal{O}_X^*)@>\delta>> \mathcal{C}^2(\mathcal{U},\mathcal{O}_X^*)
\end{CD}$
\end{center}
of the complex of sheaves $\mathcal{O}_X^{*\bullet}:\mathcal{O}_X^*\xrightarrow{[\Lambda_0,\log-]}\mathscr{H}om_{\mathcal{O}_X}(\Omega_{X/k}^1,\mathcal{O}_X)\xrightarrow{[\Lambda_0,-]} \mathscr{H}om_{\mathcal{O}_X}(\wedge^2\Omega_{X/k}^1,\mathcal{O}_X)\to\cdots$. In the sequel we will denote its $i$-th hypercohomology group by $\mathbb{H}^i(X,\Lambda_0,\mathcal{O}_X^{*\bullet})$.

Conversely, a $1$-cocycle $(\{T_i\},\{f_{ij}\})$ define a Poisson invertible sheaf $(L,\nabla)$. For given two $(\{T_i\},\{f_{ij}\})$ and $(\{T_i'\},\{f_{ij}'\})$ defining the same cohomology class so that $\frac{b_i}{b_j}=\frac{f_{ij}}{f_{ij}'}$ and $[\Lambda_0,\log b_i]=T_i-T_i'$ for some $\{b_i\}\in \mathcal{C}^0(\mathcal{U},\mathcal{O}_X^*)$, let $(L,\nabla)$ be the Poisson invertible sheaf defined by $(\{T_i\},\{f_{ij}\})$ and $(L',\nabla')$ the Poisson invertible sheaf defined by $(\{T_i'\},\{f_{ij}'\})$. Then we can define an isomorphism $(L',\nabla')\to (L,\nabla)$ as Poisson $\mathcal{O}_X$-modules in the following way. On each $U_i$, we define $\mathcal{O}_X(U_i)\xrightarrow{\cdot b_i} \mathcal{O}_X(U_i)$ which define an isomorphism $L'\to L$ as $\mathcal{O}_X$-modules by the condition $\frac{b_i}{b_j}=\frac{f_{ij}}{f_{ij}'}$. On the other hand, we note that $b_iT_i'=-[\Lambda_0,b_i]+b_iT_i$. Then $b_i\{a,s_i\}_{L'}=b_i(\{a_i,s_i\}+s_iT_i'a_i)=b_i\{a,s_i\}+s_i(-[\Lambda_0,b_i]+b_iT_i)(a)=b_i\{a,s_i\}+s_i\{a,b_i\}+s_ib_iT_ia=\{a,b_is_i\}+b_is_iT_ia=\{a,b_is_i\}_L$ so that $\cdot b_i$ is a Poisson $\mathcal{O}_X$-module homomorphism.

This show that we can identify $\mathbb{H}^1(X,\Lambda_0,\mathcal{O}_X^{*\bullet})$ with isomorphism classes of Poisson invertible sheaves on $(X,\Lambda_0)$. Moreover $\mathbb{H}^1(X,\Lambda_0,\mathcal{O}_X^{*\bullet})$ forms a group. We simply note that given two Poisson invertible sheaves $(L,\nabla_L)$ and $(L',\nabla_{L'})$ which are represented by $(\{T_i\},\{f_{ij}\})$ and $(\{T_i'\},\{f_{ij}'\})$ respectively, we can define  Poisson invertible sheaves $(L\otimes L', \nabla_{L\otimes L'})$ by $(\{T_i+T_i'\},\{f_{ij}f_{ij}'\})$ and $(L^{-1},\nabla_{L^{-1}})$ by $(\{-T_i\},\{f_{ij}^{-1}\})$. The Poisson structure on $\mathcal{O}_X$ define the $0$ element.

\begin{definition}\label{pic}
We call $\mathbb{H}^1(X,\Lambda_0,\mathcal{O}_X^{*\bullet})$ the Poisson Picard group of an algebraic Poisson scheme $(X,\Lambda_0)$ and denote it by $Pic_k(X,\Lambda_0)$.
\end{definition}

\begin{remark}
Let $(X,\Lambda_0)$ be a Poisson scheme over $S$ with $\Lambda_0\in \Gamma(X,\mathscr{H}om_{\mathcal{O}_X}(\wedge^2 \Omega_{X/S}^1,\mathcal{O}_X))$. 
Let $(L,\nabla)$ be a Poisson invertible sheaf on $(X,\Lambda)$. $(L,\nabla)$ is called a Poisson invertible sheaf over $S$ if the associated connection $\{-,-\}_L:\mathcal{O}_X\otimes_k L\to L$ is $\mathcal{O}_S$-linear. In other words, $\{-,-\}_L:\mathcal{O}_X\otimes_{\mathcal{O}_S}L\to L$.
Then we can identify the first cohomology group of the following complex of sheaves
\begin{align*}
\mathcal{O}_X^*\xrightarrow{[\Lambda_0,\log-]}\mathscr{H}om_{\mathcal{O}_X}(\Omega_{X/S}^1,\mathcal{O}_X)\xrightarrow{[\Lambda_0,-]} \mathscr{H}om_{\mathcal{O}_X}(\wedge^2\Omega_{X/S}^1,\mathcal{O}_X)\to\cdots
\end{align*}
with isomorphism classes of Poisson invertible sheaves over $S$ on $(X,\Lambda_0)$ and we will denote the group by $Pic_S(X,\Lambda_0)$.
\end{remark}

\begin{definition}\label{pchern}
Let $(X,\Lambda_0)$ be a nonsingular Poisson variety over $\mathbb{C}$ and $d:\mathcal{O}_X\to \Omega_X^1$ be the canonical derivation. We can define a homomorphism of complex of sheaves 
\begin{center}
$\begin{CD}
\mathcal{O}_X^{*\bullet}:@.\mathcal{O}_X^*@>[\Lambda_0,\log-]>> T_X @>[\Lambda_0,-]>> \wedge^2 T_X@>>> \cdots\\
@.@Vd\log - VV @VidVV@VVV\\
@.\Omega_X^1@>i_{\Lambda_0}>> T_X@>>> 0@>>>\cdots
\end{CD}$
\end{center}
So we have an induced group homomorphism:
\begin{align*}
c:\mathbb{H}^1(X,\Lambda_0,\mathcal{O}_X^{*\bullet})\to \mathbb{H}^1(\Omega_X^1\xrightarrow{i_{\Lambda_0}} T_X)
\end{align*}
Let $(L,\nabla)$ be a Poisson invertible sheaf on $(X,\Lambda_0)$. Let $[L,\nabla]$ be the associated element of $\mathbb{H}^1(X,\Lambda_0,\mathcal{O}_X^*)$. We call $c(L,\nabla):=c([L,\nabla]) \in \mathbb{H}^1(\Omega_X^1\xrightarrow{i_{\Lambda_0}} T_X)$ the Poisson Chern class of $(L,\nabla)$. 
\end{definition}

\begin{remark}\label{remark12}
Let $X$ be a compact K\"ahler manifold with a holomorphic Poisson structure $\Lambda_0$. Let $(L,\nabla)$ be a Poisson invertible sheaf defined by the $1$-cocycle $\{f_{ij}\}$ and Poisson vector fields $\{T_i\}$ for an open covering $\mathcal{U}$ of $X$. In this remark, we describe $c(\{f_{ij}\},\{T_j\})=(\{d\log f_{ij}\},\{ T_j\})$ under the map $c:\mathbb{H}^1(X,\Lambda_0,\mathcal{O}_X^{*\bullet})\to \mathbb{H}^1(X,\Omega_X^1\xrightarrow{i_{\Lambda_0}} T_X)$ in terms of the Deaulbault resolution. Given the K\"ahler form of $X$, choose a Hermitian form $\langle-,-\rangle$ on the fibers of $L$ so that $\langle\xi,\xi\rangle=a_j|\xi_j|^2$, where $\xi_j$ is a fiber coordinate of $\xi$, and $a_j(z)$ is a real positive $C^{\infty}$ function on $U_j$. Since $a_j|\xi_j|^2=a_k|\xi_k|^2$ and $\xi_j=f_{jk}\xi_k$, we have $|f_{jk}|^2=\frac{a_k}{a_j}$. Then $-\delta(\{-\partial\log a_i\})=\{d\log f_{ij}\}$, where $\{-\partial\log a_i\}\in \mathcal{C}^0(\mathcal{U},\mathscr{A}^{0,0}(\Omega_X^1))$\footnote{$\mathscr{A}^{0,0}(\Omega_X^1)$ is the sheaf of germs of $C^{\infty}$-sections of $\Omega_X^1$.}. We note that $i_{\Lambda_0}(-\partial \log a_i)=-[\Lambda_0,\log a_i]$. Hence, in the following two resolutions of $\Omega_X^1\xrightarrow{i_{\Lambda_0}} T_X\to0\to0\to \cdots$,

\begin{center}
$\begin{CD}
0\\
@AAA\\
\mathcal{C}^0(\mathcal{U},T_X)@>\delta>> \mathcal{C}^1(\mathcal{U},T_X)\\
@Ai_{\Lambda_0}AA @Ai_{\Lambda_0} AA\\
\mathcal{C}^0(\mathcal{U},\Omega^1_X)@>-\delta>> \mathcal{C}^1(\mathcal{U},\Omega^1_X)@>>> \cdots,
\end{CD}$
$\begin{CD}
0\\
@AAA\\
A^{0,0}(X,T_X)@>-\bar{\partial}>> A^{0,1}(X,T_X)\\
@Ai_{\Lambda_0}AA @Ai_{\Lambda_0} AA\\
A^{0,0}(X,\Omega^1_X)@>\bar{\partial}>> A^{0,1}(X,\Omega^1_X)@>>> \cdots
\end{CD}$
\end{center}
$(\{T_i\},\{d\log f_{ij}\})\in \mathcal{C}^0(\mathcal{U},T_X)\oplus \mathcal{C}^1(\mathcal{U},\Omega_X^1)$ in the \v{C}ech resolution corresponds to 
\begin{align*}
(\{-\bar{\partial}{\partial}\log a_i\},\{-[\Lambda_0,\log a_i]-T_i]\})\in A^{0,1}(X,\Omega_X^1)\oplus A^{0,0}(X,T_X) 
\end{align*}
\end{remark}

\section{Deformations of a Poisson invertible sheaf $(L,\nabla)$ under trivial Poisson deformations}\label{section4}

\begin{definition}\label{unt}
Let $(X,\Lambda_0)$ be a nonsingular Poisson variety. Let $A$ be in $\bold{Art}$. An infinitesimal deformation of a Poisson invertible sheaf $(L,\nabla)$ over $A$ consists of the trivial Poisson deformation $(X\times_{Spec(k)} Spec(A),\Lambda_0)$ and a Poisson invertible sheaf $(\mathcal{L},\nabla_{\mathcal{L}})$ over $A$ on $(X\times_{Spec(k)} Spec(A),\Lambda_0)$ such that $(L,\nabla)=(\mathcal{L},\nabla_{\mathcal{L}})|_{(X,\Lambda_0)}$
 Two deformations $(\mathcal{L},\nabla_\mathcal{L})$ and $(\mathcal{L}', \nabla_\mathcal{L'})$ of $(L,\nabla)$ over $A$ is called isomorphic  if there is an isomorphism $(\mathcal{L},\nabla_{\mathcal{L}})\cong (\mathcal{L}',\nabla_{\mathcal{L}'})$ as Poisson $\mathcal{O}_{X\times_{Spec(k)} Spec(A)}$-modules. Then we can define a functor of Artin rings
\begin{align*}
 Def_{(L,\nabla_L)}:\bold{Art}&\to (sets)\\
                              A&\mapsto \{\text{infinitesimal deformations of $(L,\nabla)$ over $A$}\}/\text{isomorphism}
\end{align*} 
\end{definition}

\begin{remark}
Given a nonsingular Poisson variety $(X,\Lambda_0)$, isomorphism classes of Poisson invertible sheaves over $A$ on $(X\times Spec(A),\Lambda_0)$ can be identified with the first hypercohomlogy group of the following complex of sheaves
\begin{align*}
\mathcal{O}_{X\times Spec(A)}^*\xrightarrow{[\Lambda_0,\log-]} T_X\otimes_k A\xrightarrow{[\Lambda_0,-]}\wedge^2 T_X\otimes_k A\xrightarrow{[\Lambda_0,-]} \cdots
\end{align*}

\end{remark}

\begin{proposition}\label{tpl}
Let $(X,\Lambda_0)$ be a nonsingular Poisson variety. Then
\begin{enumerate}
\item There is a $1-1$ correspondence
\begin{align*}
\kappa: Def_{(L,\nabla)}(k[\epsilon])=\{\text{isomorphism classes of first-order deformations of $(L,\nabla)$}\}\cong \mathbb{H}^1(X,\Lambda_0,\mathcal{O}_X^\bullet)
\end{align*}
\item Let $A\in \bold{Art}$ and $\eta=(\mathcal{L},\nabla_\mathcal{L})$ be an infinitesimal Poisson deformation of $(L,\nabla)$ over $A$. Then, to every small extension $e:0\to (t)\to\tilde{A}\to A\to 0$, we can associate an element $o_\eta(e)\in\mathbb{H}^2(X,\Lambda_0,\mathcal{O}_X^\bullet)$ called the obstruction lifting of $\eta$ to $\tilde{A}$, which is $0$ if and only if a lifting of $\eta$ exist.
\end{enumerate}
\end{proposition}

\begin{proof}
Let $\mathcal{U}=\{U_i\}$ be an affine open covering such that $(L,\nabla)$ is given by a system of transition functions $\{f_{ij}\}\in \mathcal{C}^1(\mathcal{U},\mathcal{O}_X^*)$ and $\{T_i\}\in \mathcal{C}^0(\mathcal{U},T_X)$ so that $f_{ij}f_{jk}=f_{ik},[\Lambda_0,T_i]=0$ and $T_j-T_i+[\Lambda_0,\log f_{ij}]=0$. Then a first order deformation $(\mathcal{L},\nabla_{\mathcal{L}})$ of $(L,\nabla)$ over $Spec(k[\epsilon])$ is given by $\{F_{ij}=f_{ij}+\epsilon g_{ij}\}\in \mathcal{C}^0(\mathcal{U},\mathcal{O}_{X\times spec(k[\epsilon])}^*)$ and $\{Y_i=T_i+\epsilon W_i\}\in \mathcal{C}^1(\mathcal{U},T_X\otimes k[\epsilon])$. The cocycle condition $F_{ij}F_{jk}=F_{ik}$ gives $(f_{ij}+\epsilon g_{ij})(f_{jk}+\epsilon g_{jk})=(f_{ik}+\epsilon g_{ik})$ so that $g_{ij}f_{jk}+f_{ij}g_{jk}=g_{ik}$, equivalently, $\frac{g_{ij}}{f_{ij}}+\frac{g_{jk}}{f_{jk}}=\frac{g_{ik}}{f_{ik}}$. $[\Lambda_0,T_i+\epsilon W_i]$ gives $[\Lambda_0,W_i]=0$.  $Y_j-Y_i+[\Lambda_0,\log F_{ij}]=0$ gives $T_j+\epsilon W_j-T_i-\epsilon W_i+\frac{f_{ij}-\epsilon g_{ij}}{f_{ij}^2}[\Lambda_0,f_{ij}+\epsilon g_{ij}]=0$ so that $W_j-W_i+[\Lambda_0,\frac{g_{ij}}{f_{ij}}]=0$. Then $(\{\frac{g_{ij}}{f_{ij}}\},\{W_i\})\in \mathcal{C}^0(\mathcal{U},\mathcal{O}_X)\oplus \mathcal{C}^0(\mathcal{U},T_X)$ define an element of $\mathbb{H}^1(X,\Lambda_0,\mathcal{O}_X^\bullet)$ in the following \v{C}ech resolution
\begin{center}
$\begin{CD}
\mathcal{C}^0(\mathcal{U},\wedge^3 T_X)\\
@A[\Lambda_0,-]AA\\ 
\mathcal{C}^0(\mathcal{U},\wedge^2 T_X)@>-\delta>>\mathcal{C}^1(\mathcal{U},\wedge^2 T_X)\\
@A[\Lambda_0,-]AA @A[\Lambda_0,-]AA\\
\mathcal{C}^0(\mathcal{U},T_X)@>\delta>>\mathcal{C}^1(\mathcal{U},T_X)@>-\delta>>\mathcal{C}^2(\mathcal{U},T_X)\\
@A[\Lambda_0,-]AA @A[\Lambda_0,-]AA @A[\Lambda_0,-]AA\\
\mathcal{C}^0(\mathcal{U},\mathcal{O}_X)@>-\delta>>\mathcal{C}^1(\mathcal{U},\mathcal{O}_X)@>\delta>>\mathcal{C}^2(\mathcal{U},\mathcal{O}_X)@>-\delta>>\mathcal{C}^3(\mathcal{U},\mathcal{O}_X)
\end{CD}$
\end{center}
Given two equivalent first-order deformations $(\mathcal{L},\nabla_{\mathcal{L}})$ and $(\mathcal{L}',\nabla_{\mathcal{L}'})$ of $(L,\nabla)$ which are represented by $(\{T_i+\epsilon W_i\},\{f_{ij}+\epsilon g_{ij}\})$ and $(\{T_i+\epsilon W_i'\},\{ f_{ij}+\epsilon g_{ij}'\})$ respectively so that we have a $\{a_i\}\in \mathcal{C}^0(\mathcal{U},\mathcal{O}_X)$ such that $(1+\epsilon a_i)(f_{ij}+\epsilon g_{ij})=(f_{ij}+\epsilon g_{ij}')(1+\epsilon a_j)$. Then $a_if_{ij}+g_{ij}=f_{ij}a_j+g_{ij}'$, equivalently, $a_i-a_j=\frac{g_{ij}'}{f_{ij}}-\frac{g_{ij}}{f_{ij}}$. On the other hand, since multiplication by $1+\epsilon a_i$ is a Poisson $\mathcal{O}_X$-module homomorphism, we have $T_i+\epsilon W_i-T_i-\epsilon W_i'+[\Lambda_0, \log (1+\epsilon a_i)]$ so that $W_i-W_i'+[\Lambda_0,a_i]=0$. Hence $(\{\frac{g_{ij}}{f_{ij}}\},\{W_i\})$ and $(\{\frac{g_{ij}'}{f_{ij}}\},\{W_i'\})$ are cohomologous so that we get $(1)$.

Second we identify obstructions.  
Let us consider a small extension $e:0\to (t)\to \tilde{A}\to A\to 0$ in $\bold{Art}$ and let $\eta=(\mathcal{L},\nabla_{\mathcal{L}})$ be an infinitesimal deformation of $(L,\nabla)$ over $Spec(A)$. Let $\eta=(\mathcal{L},\nabla_{\mathcal{L}})$ be represented by $\{F_{\ij}\}$, where $F_{ij}\in \Gamma(U_{ij},\mathcal{O}_{X\times spec(A)}^*)$ and $\{Y_i\}$, where $Y_i\in \Gamma(U_i,T_X)\otimes_k A$. In order to see if a lifting $\tilde{\eta}$ of $\eta$ to $Spec(\tilde{A})$ exists, we choose an arbitrary collection $\{\tilde{F}_{ij}\}, \{\tilde{Y}_i\}$, where $\tilde{F}_{ij}$ is a nowhere zero function on $U_{ij}\times_{Spec(k)} Spec(\tilde{A})$ which restricts to $F_{ij}$ on $U_{ij}\times Spec(A)$ and $\tilde{Y}_i\in\Gamma(U_i,T_X)\otimes \tilde{A}$ restricts to $Y_i\in \Gamma(U_i,T_X)\otimes A$. Then $\tilde{F}_{ij}\tilde{F}_{jk}\tilde{F}_{ik}^{-1}=1+tg_{ijk}$ for some $g_{ijk}\in \Gamma(U_{ijk}, \mathcal{O}_X)$, $[\Lambda_0, \tilde{Y}_i]=tS_i$ for some $S_i\in \Gamma(U_i,\wedge^2 T_X)$, and $\tilde{Y}_j-\tilde{Y}_i+[\Lambda_0,\log \tilde{F}_{ij}]=tQ_{ij}$ for some $Q_{ij}\in\Gamma(U_{ij},T_X)$. Now we claim that $\alpha:=(\{S_i\},\{Q_{ij}\}, \{g_{ijk}\})\in \mathcal{C}^0(\mathcal{U},\wedge^2 T_X)\oplus \mathcal{C}^1(\mathcal{U}, T_X)\oplus \mathcal{C}^2(\mathcal{U},\mathcal{O}_X)$ is a \v{C}ech $2$-cocycle in the above \v{C}ech resolution. Since $0=[\Lambda_0,[\Lambda_0,\tilde{Y}_i]]=t[\Lambda_0, S_i]$, we have $[\Lambda_0, S_i]=0$. $tS_i-tS_j=[\Lambda_0,\tilde{Y}_i-\tilde{Y}_j]=[\Lambda_0,[\Lambda_0,\log \tilde{F}_{ij}]-tQ_{ij}]=-t[\Lambda_0, Q_{ij}]$ so that we have $S_i-S_j+[\Lambda_0, Q_{ij}]=0$. $0=\tilde{Y}_i-\tilde{Y}_j+\tilde{Y}_j-\tilde{Y}_k+\tilde{Y}_k-\tilde{Y}_i=-t(Q_{ij}+Q_{jk}+Q_{ki})+[\Lambda_0, \log \tilde{F}_{ij}\tilde{F}_{jk}\tilde{F}_{ki}]=-t(Q_{ij}+Q_{jk}+Q_{ki})+[\Lambda_0, \log ( 1+tg_{ijk})]=-t(Q_{ij}+Q_{jk}+Q_{ki})+(1-tg_{ijk})[\Lambda_0, 1+tg_{ijk}]=t(-(Q_{ij}+Q_{jk}+Q_{ki})+[\Lambda_0, g_{ijk}])$. This proves the claim.

Let $\{\tilde{F}_{ij}'\},\{\tilde{Y}_i'\}$ be another such arbitrary collection inducing $\tilde{\eta}'=(\mathcal{L},\nabla_{\mathcal{L}})$. Then we claim that the $2$-cocycle $\beta:=(\{S_i'\},\{Q_{ij}'\},\{g_{ijk}'\})$ associated with $\{F_{ij}'\},\{\tilde{Y}_i'\}$ is cohomologous to the $2$-cocycle $\alpha=(\{S_i\},\{Q_{ij}\},\{g_{ijk}\})$ associated with $\{\tilde{F}_{ij}\},\{\tilde{Y}_i\}$. Note that $\tilde{F}_{ij}'=\tilde{F}_{ij}+tF_{ij}'$ for some $F_{ij}'\in \Gamma(U_{ij},\mathcal{O}_X)$ and $\tilde{Y}_i'=\tilde{Y}_i+tY_i'$ for some $Y_i\in \Gamma(U_i,T_X)$. Then $tS_i'=[\Lambda_0,\tilde{Y}_i']=[\Lambda_0,\tilde{Y}_i+tY_i']=tS_i+t[\Lambda_0,Y_i']$ so that $S_i'-S_i=[\Lambda_0,Y_i']$. $1+tg_{ijk}'=\tilde{F}_{ij}'\tilde{F}_{jk}'\tilde{F'}_{ik}^{-1}=(\tilde{F}_{ij}+tF_{ij}')(\tilde{F}_{jk}+tF_{jk}')\frac{\tilde{F}_{ik}-tF_{ik}'}{\tilde{F}_{ik}^2}=1+tg_{ijk}+tF_{ij}'f_{jk}f_{ki}+f_{ij}F_{jk}'f_{ki}-f_{ij}f_{jk}F_{ik}'f_{ki}^2$ so that we have $g_{ijk}'-g_{ijk}=\frac{F_{ij}'}{f_{ij}}+\frac{F_{jk}'}{f_{jk}}-\frac{F_{ik}'}{f_{ik}}$. $tQ_{ij}'=\tilde{Y}_j'-\tilde{Y}_i'+\frac{1}{\tilde{F}_{ij}'}[\Lambda_0,\tilde{F}_{ij}']=\tilde{Y}_j+tY_j'-\tilde{Y}_i-tY_i'+\frac{\tilde{F}_{ij}-tF_{ij}'}{\tilde{F}_{ij}^2}[\Lambda_0,\tilde{F}_{ij}+tF_{ij}']=tQ_{ij}+t(Y_j'-Y_i')+\frac{1}{f_{ij}}[\Lambda_0,F_{ij}']-\frac{F_{ij}'}{f_{ij}^2}[\Lambda_0,f_{ij}]$ so that we have $Q_{ij}'-Q_{ij}=Y_j'-Y_i'+[\Lambda_0,\frac{F_{ij}'}{f_{ij}}].$ $(\{Y_i'\},\{\frac{F_{ij}'}{f_{ij}}\})\in \mathcal{C}^0(\mathcal{U},T_X)\oplus \mathcal{C}^1(\mathcal{U},\mathcal{O}_X)$ is mapped to $\beta-\alpha$. This proves that $\alpha$ is cohomologous to $\beta$. So given a small extension $e:0\to (t)\to \tilde{A}\to A\to 0$ and an infinitesimal deformation $\eta$ of $(L,\nabla)$ over $A$, we can associate an element $o_{\eta}(e):=$ the cohomology class of $\alpha\in \mathbb{H}^2(X,\Lambda_0,\mathcal{O}_X^\bullet)$. We note that $o_{\eta}(e)=0$ if and only if there exists a collection of $\{\tilde{F}_{ij}\}$ and $\{\tilde{Y}_i\}$ satisfying the cocycle condition defining a Poisson invertible sheaf over $\tilde{A}$ on $(X\times Spec(\tilde{A}),\Lambda_0)$ which induces $\eta$.
\end{proof}

\begin{remark}
In Appendix \ref{appendixB}, more generally, we study deformations of Poisson vector bundles $($see Definition $\ref{ve4}$ and Proposition $\ref{ve2}$$)$.
\end{remark}

\begin{remark}
We can rephrase Poisson deformations of $(X,\Lambda_0,L,\nabla)$ under trivial Poisson deformations in the holomorphic setting in the following way. Let $(X,\Lambda_0)$ be a compact holomorphic Poisson manifold. By a family of deformations of Poisson invertible sheaves in the trivial Poisson deformation of $(X,\Lambda_0)$ over $M$, we mean a pair of the trivial Poisson analytic family $(\mathcal{X}=X\times M,\Lambda_0, M)$ with the projection $p:\mathcal{X}\to M$ $($see \cite{Kim14}$)$ and a Poisson invertible sheaf $(\mathcal{L},\nabla_\mathcal{L})$ of $\mathcal{X}$ over $M$. Let $\{U_i\}$ be a finite open covering of $X$, where $z_i=(z_1,...,z_n)$ is a coordinate of $U_i$, such that $\mathcal{X}$ is covered by a finite number of coordinate systems $U_i\times M$, and $(L,\nabla)$ is represented by transition functions $\{\Psi_{ij}(z_j,t)\}\in \mathcal{C}^1(\mathcal{U},\mathcal{O}_{X\times M}^*)$ with $\Psi_{ij}(z_j,t)\cdot \Psi_{jk}(z_k,t)=\Psi_{ij}(z_i,t)$ and $\{T_i\}\in \mathcal{C}^0(\mathcal{U},T_{X/M})$ where $T_i=\sum_{l=1}^n T_i^l(z_i,t)\frac{\partial}{\partial z_i^l}$ with $[\Lambda_0, T_i]=0$. Since $T_j-T_i+[\Lambda_0,\log \Psi_{ij}]=0$, by taking the derivative with respect to $t$, we get a cohomology class $(\{T_i'=\sum_{l=1}^n \frac{\partial T_i^l}{\partial t}\frac{\partial}{\partial z_i^l}\},\{\frac{1}{\Psi_{ij}}\frac{\partial \Psi_{ij}}{\partial t}\})\in \mathcal{C}^0(\mathcal{U},T_X)\oplus \mathcal{C}^1(\mathcal{U},\mathcal{O}_X)$ so that we get a characteristic map $\rho:T_t(M)\to \mathbb{H}^1(X,\Lambda_0,\mathcal{O}_X^\bullet)$. 
\end{remark}

\begin{remark}
Let $(L,\nabla)$ be Poisson invertible sheaf of on a nonsingular Poisson variety $(X,\Lambda_0)$. Then the group $Aut((L,\nabla))$ of automorphisms on $(L,\nabla)$ as Poisson $\mathcal{O}_X$-modules can be identified with $\mathbb{H}^0(X,\Lambda_0,\mathcal{O}_X^{*\bullet})=\mathbb{H}^0(X,\Lambda_0,\mathcal{O}_X^\bullet)^*$.
\end{remark}

\begin{proposition}\label{prorepre2}
Let $e:0\to (t)\to \tilde{A}\xrightarrow{\mu} A\to 0$ be a small extension in $\bold{Art}$. Let $p:=Def_{(L,\nabla)}(\mu):Def_{(L,\nabla)}(\tilde{A})\to Def_{(L,\nabla)}(\tilde{A})$. Then given $\eta=(\mathcal{L},\nabla_{L})\in Def_{(L,\nabla)}(A)$, there is a transitive action of $\mathbb{H}^1(X,\Lambda_0,\mathcal{O}_X^\bullet)$ on $p^{-1}(\eta)$. Moreover, if $\mathbb{H}^0(X\times Spec(\tilde{A}),\Lambda_0,\mathcal{O}_X^{*\bullet})\to \mathbb{H}^0(X\times Spec(A),\Lambda_0,\mathcal{O}_X^{*\bullet})$ is surjective, the action is free.
\end{proposition}

\begin{proof}
We will define a map $G:\mathbb{H}^1(X,\Lambda_0,\mathcal{O}_X^\bullet)\times p^{-1}(\eta)\to p^{-1}(\eta)$. Let $\tilde{\eta}=(\tilde{\mathcal{L}},\nabla_{\tilde{\mathcal{L}}})\in p^{-1}(\eta)$, which is represented by $\tilde{F}_{ij}\in \Gamma(U_{ij},\mathcal{O}_{X\times Spec(\tilde{A})}^*)$ and $\tilde{Y}_i\in \Gamma(U_i,T_X)\otimes Spec(\tilde{A})$ such that $\tilde{F}_{ij}\tilde{F}_{jk}=\tilde{F}_{ik}$, $[\Lambda_0,\tilde{Y}_i]=0$ and $Y_j-Y_i+[\Lambda_0,\log \tilde{F}_{ij}]=0$. Let $v=(\{W_i\},\{\frac{g_{ij}}{f_{ij}}\})\in \mathbb{H}^1(X,\Lambda_0,\mathcal{O}_X^\bullet)$. From $\tilde{\eta}$ and $v=(\{W_i\},\{\frac{g_{ij}}{f_{ij}}\})$, we define an another lifting $\tilde{\eta}'$, which is represented by $\tilde{F}_{ij}':=\tilde{F}_{ij}+t g_{ij}$ and $\tilde{Y}_i':=\tilde{Y}_i+t W_i$. Set $\tilde{\eta}':=G(v,\tilde{\eta})$. Then we can show that $G$ is well-defined and transitive.

Now assume that $\mathbb{H}^0(X\times Spec(\tilde{A}),\Lambda_0,\mathcal{O}_X^{*\bullet})\to \mathbb{H}^0(X\times Spec(A),\Lambda_0,\mathcal{O}_X^{*\bullet})$ is surjective. We show that the group action is free. Assume that for given $v=(\{W_i\},\{\frac{g_{ij}}{f_{ij}}\})\in \mathbb{H}^1(X,\Lambda_0,\mathcal{O}_X^\bullet)$, we have $G(v,\tilde{\eta})=\tilde{\eta}$. Let $\tilde{\eta}$ be represented by $(\tilde{\mathcal{L}},\nabla_{\tilde{\mathcal{L}}})$ and $G(v,\tilde{\eta})$ be represented by $(\tilde{\mathcal{L}}',\nabla_{\tilde{\mathcal{L}}'})$. This means that we have an isomorphism by multiplication $\cdot \tilde{a}_i:\mathcal{O}_X(U_i)\otimes \tilde{A}\to \mathcal{O}_X(U_i)\otimes \tilde{A}$ such that $\tilde{a}_j\cdot (\tilde{F}_{ij}+tg_{ij})=\tilde{F}_{ij}\cdot \tilde{a}_i$ and $\tilde{a}_i$ induces an automorphism $\cdot a:=\cdot a_i:\mathcal{O}_X(U_i)\otimes A\to \mathcal{O}_X(U_i)\otimes A$ on $(\mathcal{L},\nabla_{\mathcal{L}})$ so that $a\in \mathbb{H}^0(X\times Spec(A),\Lambda_0,\mathcal{O}_X^{*\bullet})$. By assumption, there is a lifting of an automorphism $\cdot \tilde{b}$ on $(\tilde{\mathcal{L}}',\nabla_{\tilde{\mathcal{L}}'})$ which induces $\cdot a^{-1}$. By replacing $\cdot \tilde{a}$ by $\tilde{b}\cdot \tilde{a}_i$, we may assume that a Poisson $\mathcal{O}_X$-module homomorphism $\cdot \tilde{a}_i$ induces an identity on $L$. Hence $a_i=1+tb_i$ for some $b_i\in \mathcal{O}_X(U_i)$. Then $\tilde{Y}_i+tW_i-\tilde{Y}_i+[\Lambda_0,\log 1+tb_i]=0$ so that $W_i=[\Lambda_0,b_i]$. $(1+tb_j)(\tilde{F}_{ij}+tg_{ij})=\tilde{F}_{ij}(1+tb_i)$ so that $b_jf_{ij}+ g_{ij}=f_{ij}b_i$. Hence $b_i-b_j=\frac{g_{ij}}{f_{ij}}$ so that $v=0$.

\end{proof}

\begin{theorem}\label{ppro2}
Let $(X,\Lambda_0)$ be a nonsingular projective Poisson variety and $(L,\nabla)$ be a Poisson invertible sheaf on $(X,\Lambda_0)$. Then $Def_{(L,\nabla)}$ is pro-representable.
\end{theorem}

\begin{proof}
We can check Schlessinger's criterion $(H_0),(H_1),(H_2)$, and since $X$ is projective, $\mathbb{H}^1(X,\Lambda_0,\mathcal{O}_X^\bullet)$ is finite-dimensional so that $(H_3)$ is satisfied. Let $e:0\to (t)\to \tilde{A}\to A\to 0$ be a small extension in $\bold{Art}$. Since $H^0(X,\mathcal{O}_X)=k$, we have $H^0(X\times Spec(\tilde{A}))=\tilde{A}$ and $H^0(X\times Spec(A))=A$. Hence $\mathbb{H}^0(X\times Spec(\tilde{A}),\Lambda_0,\mathcal{O}_X^{*\bullet})=\tilde{A}^*\to \mathbb{H}^0(X\times Spec(A),\Lambda_0,\mathcal{O}_X^{*\bullet})=A^*$ is surjective so that $(H_4)$ follows from Proposition \ref{prorepre2}.
\end{proof}

\section{Deformations of sections of a Poisson invertible sheaf $(L,\nabla)$ in trivial Poisson deformations}\label{section5}

The formalism of deformations sections of an invertible sheaf presented in \cite{Ser06} (see p.141) can be extended to Poisson deformations. Let $(X,\Lambda_0)$ be a nonsingular projective Poisson variety, and let $(L,\nabla)$ be a Poisson invertible sheaf on $(X,\Lambda_0)$. We can define a homomorphism of complex of sheaves in the following way (see Remark \ref{lco})

\begin{center}
$\begin{CD}
\mathcal{O}_X @>-[-,\Lambda_0]>> T_X@>-[-,\Lambda_0]>> \wedge^2 T_X\\
@Vm_0VV @Vm_1VV @Vm_2VV \\
\mathbb{H}^0(X,\Lambda_0,L^\bullet,\nabla)^\vee\otimes L@>\nabla:=v_0>> \mathbb{H}^0(X,\Lambda_0,L^\bullet,\nabla)^\vee\otimes L\otimes T_X@>v_1>> \mathbb{H}^0(X,\Lambda_0,L^\bullet,\nabla)^\vee\otimes L\otimes \wedge^2 T_X
\end{CD}$
\end{center}
where for every open set $U\subset X$
\begin{align*}
m_i(U):\Gamma(U_,\wedge^iT_X)&\to \mathbb{H}^0(X,\Lambda_0,L^\bullet,\nabla)^\vee\otimes\Gamma(U,L\otimes \wedge^i T_X)=Hom(\mathbb{H}^0(X,\Lambda_0, L^\bullet,\nabla),\Gamma(U,L\otimes \wedge^i T_X))\\
                            a &\mapsto [s\mapsto a\otimes s|_U]
\end{align*}
for $s\in \mathbb{H}^0(X,\Lambda_0,L^\bullet,\nabla)$. This induces $m_i:\mathbb{H}^i(X,\Lambda_0,\mathcal{O}_X^\bullet)\to Hom(\mathbb{H}^0(X,\Lambda_0,L^\bullet,\nabla),\mathbb{H}^i(X,\Lambda_0,L^\bullet,\nabla))$.

\begin{definition}
Given an infinitesimal deformation $(\mathcal{L},\nabla_{\mathcal{L}})$ of $(L,\nabla)$ over $A\in \bold{Art}$, we have an induced restriction map
\begin{align*}
\rho:\mathbb{H}^0(X\times Spec(A),\Lambda_0,\mathcal{L}^\bullet,\nabla_{\mathcal{L}})\to \mathbb{H}^0(X,\Lambda_0,L^\bullet,\nabla)
\end{align*}
We say that a section $\sigma\in \mathbb{H}^0(X,\Lambda_0,L^\bullet,\nabla)$ extends to $(\mathcal{L},\nabla_\mathcal{L})$ if $\sigma\in Im(\rho)$
\end{definition}

\begin{proposition}\label{sec1}
Let $(\mathcal{L}_a,\nabla_a)$ be a first-order deformation of $(L,\nabla)$, corresponding to an element $a\in \mathbb{H}^1(X,\Lambda_0,\mathcal{O}_X^\bullet)$. A section $s\in \mathbb{H}^0(X,\Lambda_0,L^\bullet,\nabla)$ extends to a section $\tilde{s}\in \mathbb{H}^0(X\times Spec(k[\epsilon]),\Lambda_0,\mathcal{L}_a^\bullet, \nabla_a)$ of $(\mathcal{L}_a,\nabla_a)$ if and only if
$s\in ker[m_1(a)]$ where
\begin{align*}
m_1:\mathbb{H}^1 (X,\Lambda_0,\mathcal{O}_X^\bullet)\to Hom(\mathbb{H}^0(X,\Lambda_0,L^\bullet,\nabla),\mathbb{H}^1(X,\Lambda_0,L^\bullet,\nabla))
\end{align*}
defined as above.
\end{proposition}
\begin{proof}
We keep the notation in the proof of Proposition \ref{tpl}. Let $\mathcal{U}=\{U_i\}$ be an affine open covering of $X$ such that $(L,\nabla)$ is represented by a system of transition functions $\{f_{ij}\}$, $f_{ij}\in \Gamma(U_{ij},\mathcal{O}_X^*)$ and Poisson vector fields $\{T_i\}$, $T_i\in \Gamma(U_i,T_X)$. Let $a\in \mathbb{H}^1(X,\Lambda_0,\mathcal{O}_X^\bullet)$ be represented by $\{\frac{g_{ij}}{f_{ij}},W_i\}\in \mathcal{C}^1(\mathcal{U},\mathcal{O}_X)\oplus \mathcal{C}^0(\mathcal{U}, T_X)$. Then the first order deformation $(\mathcal{L}_a, \nabla_a)$ of  $(L,\nabla)$ is represented by $\{f_{ij}+\epsilon g_{ij}\}$ and $\{T_i+\epsilon W_i\}$. Then $m_1(a)(s)$ is represented by $(\{\frac{g_{ij}}{f_{ij}} s_j\},\{s_iW_i  \})\in \mathcal{C}^1(\mathcal{U},L)\oplus \mathcal{C}^0(\mathcal{U},T_X\otimes L)$.

Let's assume that $s\in \mathbb{H}^0(X,\Lambda_0,L^\bullet,\nabla)$ is represented by the cocycle $\{s_i\}$, $s_i\in \Gamma(U_i, \mathcal{O}_X)$, such that $s_i=f_{ij}s_j$ and $-[\Lambda_0,s_{i}]+s_{i}T_{i}=0$. In order for $s$ to extend to a section $\tilde{s}\in \mathbb{H}^0(X\times Spec(k[\epsilon]),\Lambda_0,\mathcal{L}_a^\bullet,\nabla_a)$, it is necessary and sufficient that there exist $\{t_i\}, t_i\in \Gamma(U_i,\mathcal{O}_X)$ such that $s_i+\epsilon t_i=(f_{ij}+\epsilon g_{ij})(s_j+\epsilon t_j)$ on $U_{ij}$ and $-[\Lambda_0, s_i+\epsilon t_i]+(s_{i}+\epsilon t_i)(T_i+\epsilon W_i)=0$ which are equivalent to $g_{ij}s_j=t_i-f_{ij}t_j$ and $-[\Lambda_0,t_i]+s_{i}W_i+t_i T_i=0$ so that $\frac{g_{ij}}{f_{ij}}s_j=f_{ji}t_i-t_j$ and $-[\Lambda_0,-t_i]-t_i T_i=s_i W_i$. Hence the $1$-cocyle 
$(\{\frac{g_{ij}}{f_{ij}}s_j\},\{ s_i W_i\})\in \mathcal{C}^1(\mathcal{U},L)\oplus \mathcal{C}^0(U, T\otimes L)$ is a coboundary 
in the following \v{C}ech reolsution 
\begin{center}
$\begin{CD}
\mathcal{C}^0(\mathcal{U},\wedge^2 T_X\otimes L)\\
@Av_1 AA\\
\mathcal{C}^0(\mathcal{U},T_X\otimes L)@>-\delta>>\mathcal{C}^1(\mathcal{U},T_X\otimes L)\\
@Av_0:=\nabla AA @Av_0:=\nabla AA\\
\mathcal{C}^0(\mathcal{U},L)@>\delta>>\mathcal{C}^1(\mathcal{U},L)@>-\delta>>\mathcal{C}^2(\mathcal{U},L)
\end{CD}$
\end{center}
Hence $m_1(a)(s)=0$.
\end{proof}

\section{Simultaneous deformations of a nonsingular Poisson variety $(X,\Lambda_0)$ and a Poisson invertible sheaf $(L,\nabla)$}\label{section6}

\begin{proposition}\label{patiyah}

Let $(X,\Lambda_0)$ be a nonsingular projective Poisson variety over $\mathbb{C}$ and $(L,\nabla)$ be a Poisson invertible sheaf on $(X,\Lambda_0)$. Then the Poisson Chern class $c(L,\nabla)$ gives an element in $Ext^1(T_X^\bullet,\mathcal{O}_X^\bullet)$ for which we call `Poisson Atiyah extension'  associated with the Poisson Chern class $c(L,\nabla)$, which extends `Atiyah extension' associated with the Chern class $c(L)$. The Poisson Aityah extension associated with $c(L,\nabla)$ is described by $0\to \mathcal{O}_X^\bullet \to \mathcal{E}_L^\bullet\to T_X^\bullet \to 0$, where
\begin{center}
$\begin{CD}
0@>>>\cdots @>>>\cdots @>>> \cdots@>>> 0\\
@. @AAA @AAA @AAA @.\\
0@>>> \wedge^2 T_X@>>>\mathcal{E}_L^2@>>> \wedge^3 T_X@>>> 0\\
@. @A[\Lambda_0,-]AA @AdAA @AA[\Lambda_0,-]A @. \\
0@>>> T_X@>>> \mathcal{E}_L^1@>>> \wedge^2 T_X@>>> 0\\
@. @A[\Lambda_0,-]AA @Ad AA @AA[\Lambda_0,-]A @.\\
0@>>> \mathcal{O}_X@>>> \mathcal{E}_L^0@>>> T_X@>>> 0
\end{CD}$
\end{center}
The sheaf $\mathcal{E}_L^i$ is locally free of rank $\binom{n}{i}+\binom{n}{i+1}=\binom{n+1}{i}$, where $\dim(X)=n$. We note that the complex $0\to \mathcal{O}_X\to \mathcal{E}_L^0\to T_X\to 0$ is the Atiyah extension associated with the Chern class $c(L)$. We denote the $i$-th hypercohomology group of the complex of sheaves $\mathcal{E}_L^\bullet:\mathcal{E}_L^0\xrightarrow{d}\mathcal{E}_L^2\xrightarrow{d}\mathcal{E}_L^3\xrightarrow{d} \cdots$ by $\mathbb{H}^i(X,\Lambda_0,\mathcal{E}_L^\bullet)$. 
\end{proposition}

\begin{proof}
Let $\mathcal{U}=\{U_i\}$ be an affine open covering of $X$ such that $(L,\nabla)$ is represented by a system of transition functions $\{h_{ij}\}$, where $h_{ij}\in \Gamma(U_{ij},\mathcal{O}_X^*)$ and Poisson vector fields $\{T_i^0\}$, where $T_i^0\in \Gamma(U_i,T_X)$ so that $h_{ij}h_{jk}=h_{ik},[\Lambda_0,T_i^0]=0$ and $T_j^0-T_i^0+[\Lambda_0,\log h_{ij}]=0$. Then $c(L,\nabla)$ is represented by the \v{C}ech 1-cocycle
\begin{align*}
(\{d\log h_{ij}=\frac{dh_{ij}}{h_{ij}}\},\{T_i^0\})\in \mathcal{C}^1(\mathcal{U},\Omega_X^1)\oplus\mathcal{C}^0(\mathcal{U},T_X)
\end{align*}

We define the sheaf $\mathcal{E}_L^p$ in the following way. The sheaf $\mathcal{E}_L^p|_{U_i}$ is locally isomorphic to $\wedge^p T_{X}|_{U_i}\oplus \wedge^{p+1}T_X|_{U_i}$. A section $(a_i,b_i)$ of $\wedge^p T_X|_{U_i}\oplus \wedge^{p+1}T_X|_{U_i}$ and a section $(a_j,b_j)$ of $\wedge^p T_X|_{U_j}\oplus \wedge^{p+1}T_X|_{U_j}$ are identified on $U_{ij}$ if and only if $b_i=b_j$ and $a_j-a_i=[b_i,\log h_{ij}]$. Then we can check $\mathcal{E}_L^p$ is well-defined.

We define the differential $d:\mathcal{E}_L^p\to \mathcal{E}_L^{p+1}$. The differntial $d$ is locally defined in the following way. $d:\wedge^p T_X|_{U_i}\oplus\wedge^{p+1}T_X|_{U_i}\to \wedge^{p+1}T_X|_{U_i}\oplus\wedge^{p+2}T_X|_{U_i}$ is defined by $(a_i, b_i)\mapsto ([\Lambda_0,a_i]+(-1)^{p+1}[ T_i^0,b_i],[\Lambda_0,b_i])$. We check this define a differential $(d^2=0)$. Indeed, we have $[\Lambda_0,[\Lambda_0,a_i]]+(-1)^{p+1}[\Lambda_0,[T_i^0,b_i]]+(-1)^{p+2}[T_i^0,[\Lambda_0,b_i]]=0$ since $[\Lambda_0,T_i^0]=0$. We show that  $d$ is welld-defined. In other words, $([\Lambda_0,a_i]+(-1)^{p+1}[ T_i^0, b_i],[\Lambda_0,b_i])$ on $U_i$ is identified with $([\Lambda_0,a_j]+(-1)^{p+1}[T_j^0, b_j],[\Lambda_0,b_j])$ on $U_j$. Note that $b_i=b_j$ and $a_j-a_i=[b_i,\log h_{ij}]$. Then $[\Lambda_0,a_j-a_i]+(-1)^{p+1}[T_j^0-T_i^0, b_i]=[\Lambda_0,[b_i, \log h_{ij}]]-(-1)^{p+1}[[\Lambda_0,\log h_{ij}],b_i]=[[\Lambda_0,b_i],\log h_{ij}]+(-1)^{p}[b_i,[\Lambda_0,\log h_{ij}]]-(-1)^{p+1}[[\Lambda_0,\log h_{ij}],b_i]=[[\Lambda_0,b_i],\log h_{ij}].$

Let $(L,\nabla)$ and $(L',\nabla')$ be two Poisson invertible sheaves represented by $(\{f_{ij}\},\{T_i\})$ and $(\{f_{ij}'\},\{T_i'\})$ respectively. Assume that $c(L,\nabla)=c(L',\nabla')$ so that there exists $\{h_i\}\in \mathcal{C}^0(\mathcal{U},\Omega_X^1)$ such that $h_i-h_j=d\log f_{ij}-d\log f_{ij}'$ and $i_{\Lambda_0}h_i=T_i-T_i'$. We show that the complex $(\mathcal{E}_L^\bullet,d)$ associated with $(\{d\log f_{ij}\},\{T_i\})$ is isomorphic to the complex $(\mathcal{E}_{L'}^\bullet,d')$ associated with $(\{d\log f_{ij}'\},\{T_i'\})$. Since $X$ is a compact K\"{a}hler and $L\otimes L^{'-1}$ has the trivial Chern class, we may assume that $f_{ij}\cdot f_{ij}^{'-1}=a_{ij}$ is a constant so that $h_i-h_j=0$. Hence $\{h_i\}$ define a global $1$-form $H\in H^0(X,\Omega_X^1)$ with $dH=0$. So we may assume that there exists $\{c_i\}\in \mathcal{C}^0(\mathcal{U},\mathcal{O}_X)$ such that $dc_i=h_i$. Hence $i_{\Lambda_0}h_i=[\Lambda_0, c_i]=T_i-T_i'$. Now we show that $\phi:(\mathcal{E}_L^\bullet,d)\cong (\mathcal{E}_{L'}^{\bullet},d')$. $\phi=\{\phi_p\}$ is locally defined in the following way:  $\phi_p:\mathcal{E}_L^p|_{U_i}=\wedge^p T_X|_{U_i}\oplus \wedge^{p+1}T_X|_{U_i}\to \mathcal{E}_{L'}^p|_{U_i}= \wedge^p T_X|_{U_i}\oplus \wedge^{p+1}T_X|_{U_i}$, $(a,b)\mapsto (a+[b,c_i], b)$.
We claim that this is well-defined. First we show that $\phi=\{\phi_p\}$ is independent of $U_i$. We have to show that the following diagram commutes.
\begin{center}
$\begin{CD}
\mathcal{E}_L^{p}|_{U_j}=\wedge^p T_X|_{U_j}\oplus \wedge^{p+1}T_X|_{U_j} @>\phi>> \mathcal{E}_{L'}^p|_{U_j}=\wedge^p T_X|_{U_j}\oplus \wedge^{p+1}T_X|_{U_j}\\
@A\cong AA @AA\cong A\\
\mathcal{E}_L^{p}|_{U_i}=\wedge^p T_X|_{U_i}\oplus \wedge^{p+1}T_X|_{U_i} @>\phi>> \mathcal{E}_{L'}^p|_{U_i}=\wedge^p T_X|_{U_i}\oplus \wedge^{p+1}T_X|_{U_i}
\end{CD}$
\end{center}

Indeed, $(a,b)\in \mathcal{E}_L^p |_{U_i}\mapsto ( a+[b,\log f_{ij}],b)\in \mathcal{E}_L^p|_{U_j}\mapsto (a+[b,\log f_{ij}]+[b,c_j],b)\in \mathcal{E}_{L'}^p|_{U_j}$ and $(a,b)\in \mathcal{E}_L^p|_{U_i}\mapsto(a+[b,c_i],b)\in \mathcal{E}_{L'}^p|_{U_i}\mapsto (a+[b,c_i]+[b,\log f_{ij}'],b)\in \mathcal{E}_{L'}^p|_{U_j}$ which are same since $[b,c_i-c_j]=[b,\log f_{ij}-\log f_{ij}']=0$. 

Next we show that $\phi\circ d=d'\circ \phi$.
\begin{center}
$\begin{CD}
\mathcal{E}_L^{p}|_{U_i}=\wedge^p T_X|_{U_i}\oplus \wedge^{p+1}T_X|_{U_i} @>\phi>> \mathcal{E}_{L'}^p|_{U_i}=\wedge^p T_X|_{U_i}\oplus \wedge^{p+1}T_X|_{U_i}\\
@AdAA @AAd'A\\
\mathcal{E}_L^{p-1}|_{U_i}=\wedge^{p-1}T_X|_{U_i}\oplus \wedge^{p}T_X|_{U_i} @>\phi>> \mathcal{E}_{L'}^{p-1}|_{U_i}=\wedge^{p-1} T_X|_{U_i}\oplus \wedge^{p}T_X|_{U_i}
\end{CD}$
\end{center}
$(a,b)\in \mathcal{E}_L^{p-1}|_{U_i}\mapsto ([\Lambda_0,a]+(-1)^p[T_i,b],[\Lambda_0,b] )\in \mathcal{E}_L^p|_{U_i}\mapsto ([\Lambda_0,a]+(-1)^p[T_i,b]+[[\Lambda_0,b],c_i],[\Lambda_0,b])\in \mathcal{E}_{L'}^p|_{U_i}$ and $(a,b)\in \mathcal{E}_L^{p-1}|_{U_i}\mapsto (a+[b,c_i],b)\in \mathcal{E}_{L'}^{p-1}|_{U_i}\mapsto ([\Lambda_0,a]+[\Lambda_0,[b,c_i]]+(-1)^p[T_i',b],[\Lambda_0,b])\in \mathcal{E}_{L'}^p|_{U_i}$ which are same since $(-1)^p[T_i-T_i',b]+[[\Lambda_0,b],c_i]-[\Lambda_0,[b,c_i]]=(-1)^p[[\Lambda_0,c_i],b]+[[\Lambda_0,b],c_i]-[\Lambda_0,[b,c_i]]$ and $[\Lambda_0,[b,c_i]]=[[\Lambda_0,b],c_i]+(-1)^{p+1}[b,[\Lambda_0,c_i]]=[[\Lambda_0,b],c_i]+(-1)^p[[\Lambda_0,c_i],b]$. Hence we get the claim.

When $c(L,\nabla)=0$, $\mathcal{E}_L^\bullet$ is isomorphic to $\mathcal{O}_X\oplus T_X\xrightarrow{[\Lambda_0,-]} T_X\oplus \wedge^2 T_X\xrightarrow{[\Lambda_0,-]}\wedge^2 T_X\oplus \wedge^3 T_X\to \cdots$ so that $\mathbb{H}^i(X,\Lambda_0,\mathcal{E}_L^\bullet)=\mathbb{H}^i(X,\Lambda_0,\mathcal{O}_X^\bullet)\oplus \mathbb{H}^i(X,\Lambda_0, T_X^\bullet)$.

\end{proof}

\begin{definition}\label{simul}
Let $A$ be in $\bold{Art}$. Let $(X,\Lambda_0, L,\nabla)$ be a pair of a nonsingular Poisson variety $(X,\Lambda_0)$ and a Poisson invertible sheaf $(L,\nabla)$ on $(X,\Lambda_0)$. An infinitesimal deformation of $(X,\Lambda_0, L,\nabla)$ over $A$ consists of a pair $(\mathcal{X},\Lambda,\mathcal{L},\nabla_\mathcal{L})$
\begin{center}
$\xi:$$\begin{CD}
(X,\Lambda_0)@>>> (\mathcal{X},\Lambda)\\
@VVV @VVV\\
Spec(k)@>>> Spec(A)
\end{CD}$
\end{center}
is an infinitesimal Poisson deformation of $(X,\Lambda)$ over $A$ and $(\mathcal{L},\nabla_\mathcal{L})$ is a Poisson invertible sheaf on $(\mathcal{X},\Lambda)$ over $A$ such that $(L,\nabla)=(\mathcal{L}|_X,\nabla_{\mathcal{L}}|_X)$. We say that $(\mathcal{L},\nabla_{\mathcal{L}})$ is a $($Poisson$)$ deformation of $(L,\nabla)$ along $\xi$. Two deformations $(\mathcal{X},\Lambda, \mathcal{L},\nabla_\mathcal{L})$ and $(\mathcal{X}',\Lambda',\mathcal{L}',\nabla_\mathcal{L'})$ of $(X,\Lambda_0,L,\nabla)$ over $A$ is called isomorphic  if there is a Poisson isomorphism of deformation $f:(\mathcal{X},\Lambda)\to (\mathcal{X}',\Lambda')$ over $A$ and an isomorphism $(\mathcal{L},\nabla_{\mathcal{L}})\to (f^*\mathcal{L}',f^*\nabla_{\mathcal{L}'})$. Then we can define a functor of Artin rings
\begin{align*}
Def_{(X,\Lambda_0,L,\nabla)}:\bold{Art}&\to (sets)\\
A&\mapsto Def_{(X,\Lambda_0,L,\nabla)}(A)=\{\text{deformations of $(X,\Lambda_0,L,\nabla)$ over $A$}\}/\text{isomorphism}
\end{align*}
\end{definition}

\begin{proposition}[compare \cite{Ser06} Theorem 3.3.11 page 146]\label{spo}
Let $(X,\Lambda_0, L,\nabla)$ be a pair of a nonsingular projective Poisson variety $(X,\Lambda_0)$ over $k=\mathbb{C}$ and a Poisson invertible sheaf $(L,\nabla)$ on $(X,\Lambda_0)$. Then
\begin{enumerate}
\item There is a canonical isomorphism
\begin{align*}
Def_{(X,\Lambda_0,L,\nabla)}(k[\epsilon])= \frac{\text{first-order deformations of $(X,\Lambda_0,L,\nabla)$}}{\text{isomorphism}}\cong \mathbb{H}^1(X,\Lambda_0,\mathcal{E}_L^\bullet)
\end{align*}
where $\mathcal{E}_L^\bullet$ is a complex from `Poisson Atiyah extension' associated with the Poisson Chern class $c(L,\nabla)$.
\item Let $A\in \bold{Art}$ and $\eta=(\mathcal{X},\Lambda,\mathcal{L},\nabla_{\mathcal{L}})$ be an infinitesimal deformation of $(X,\Lambda_0,L,\nabla)$ over $A$. Then, to every small extension $e:0\to (t)\to \tilde{A}\to A\to 0$, we can associate an element $o_\eta(e)\in \mathbb{H}^2(X,\Lambda_0,\mathcal{E}_L^\bullet)$ called the obstruction lifting of $\eta$ to $\tilde{A}$, which is $0$ if and only if a lifting of $\eta$ exist.
\item The Poisson Chern class $c(L,\nabla)$ defines a map $\mathbb{H}^1(X,\Lambda_0,T_X^\bullet)\xrightarrow{c(L,\nabla)} \mathbb{H}^2 (X,\Lambda_0,\mathcal{O}_X^\bullet)$ such that given a first-order deformation $\xi$ of $(X,\Lambda_0)$, there is a first-order deformation of $(L,\nabla)$ along $\xi$ if and only if 
\begin{align*}
c(L,\nabla)(\kappa(\xi))=0.
\end{align*}
Recall that $\kappa(\xi)$ is the element of $\mathbb{H}^1(X,\Lambda_0,T_X^\bullet)$ associated with the first-order deformation $\xi$ $($see Proposition $\ref{3q}$ $)$.
\end{enumerate}

\end{proposition}

\begin{proof}

Let $\eta=(\xi,\mathcal{L},\nabla_\mathcal{L})$ be a first-order deformation of $(X,\Lambda_0,\nabla,L)$ over $k[\epsilon]$, where
\begin{center}
$\xi:\begin{CD}
(X,\Lambda_0)@>>> (\mathcal{X},\Lambda)\\
@VVV @VVV\\
Spec(k)@>>> Spec(k[\epsilon])
\end{CD}$
\end{center}
Let $\mathcal{U}=\{U_i\}$ be an affine open covering such that $(L,\nabla)$ is given by a system of transition functions $\{h_{ij}\}\in \mathcal{C}^1(\mathcal{U},\mathcal{O}_X^*)$ and $\{T_i^0\}\in \mathcal{C}^0(\mathcal{U},T_X)$, and $\kappa(\xi)\in \mathbb{H}^1(X,\Lambda_0,T_X)$ is given by a \v{C}ech $1$-cocycle $(\{p_{ij})\},\{-\Lambda_i\})\in \mathcal{C}^1(\mathcal{U},T_X)\oplus \mathcal{C}^0(\mathcal{U},\wedge^2 T_X)$ as in the proof of Proposition \ref{3q}. We keep the notations in the proof of Proposition \ref{3q} so that $Id+\epsilon p_{ij}:(\mathcal{O}_X(U_j)\otimes k[\epsilon],\Lambda_0+\epsilon \Lambda_j)\to (\mathcal{O}_X(U_i),\Lambda_0+\epsilon \Lambda_i)$ be a Poisson isomorphism defining $(\mathcal{X},\Lambda)$. 

Let the Poisson invertible sheaf $(\mathcal{L},\nabla_\mathcal{L})$ be represented by a system of transition functions $\{F_{ij}\}$, where $F_{ij}\in \Gamma(U_{ij},\mathcal{O}_{X\times Spec(k[\epsilon])}^*)$, and $\{Y_i\}$, where $Y_i\in \Gamma(U_i, T_X)\otimes k[\epsilon]$ which reduces to $\{h_{ij}\}$ and $\{T_i^0\} \mod\,\epsilon$. Therefore it can be represented on $U_{ij}\times Spec(k[\epsilon])$ as
\begin{align*}
F_{ij}=h_{ij}+\epsilon g_{ij},\,\,\,\,\, g_{ij}\in \Gamma(U_{ij}, \mathcal{O}_X)
\end{align*}
(here we see $F_{ij}$ be a function defined on $U_{ij}\times Spec(k[\epsilon])\subset U_i\times Spec(k[\epsilon])$) and on $U_i$ as
\begin{align*}
Y_i=T_i^0+\epsilon W_i,\,\,\,\,\, W_i\in \Gamma(U_i, T_X)
\end{align*}
We note that since $F_{ji}=h_{ji}+\epsilon g_{ji}$ which is considered to be a function on $U_{ij}\times Spec(k[\epsilon])\subset U_j\times Spec(k[\epsilon])$, $(Id+\epsilon p_{ij})(F_{ji})=F_{ij}^{-1}=\frac{h_{ij}-\epsilon g_{ij}}{h_{ij}^2}$ so that $h_{ji}+\epsilon g_{ji}+\epsilon p_{ij}(h_{ji})=h_{ji}-\epsilon \frac{g_{ij}}{h_{ij}^2}$. So $g_{ji}-\frac{1}{h_{ij}^2}p_{ij}(h_{ij})=-\frac{g_{ij}}{h_{ij}^2}$. Hence $g_{ij}=-h_{ij}^2 g_{ji}+[p_{ij},h_{ij}].$
Now we consider $(\frac{g_{ij}}{h_{ij}},p_{ij})$ to be on $U_{ij}\subset U_j$ (i.e. in $\mathcal{O}_X|_{U_j}\oplus T_X|_{U_j})$. Then we have $\{(\frac{g_{ij}}{h_{ij}},p_{ij})\}\in \mathcal{C}^1(\mathcal{U},\mathcal{E}_L^0)$. Indeed, $(\frac{g_{ij}}{h_{ij}},p_{ij})$ on $U_j$ is identified with $(\frac{g_{ij}}{h_{ij}}-[p_{ij},\log h_{ij}],p_{ij})=(\frac{-h_{ij}^2 g_{ji}+[p_{ij},h_{ij}]}{h_{ij}}-\frac{1}{h_{ij}}[p_{ij},h_{ij}],-p_{ji})=(-\frac{g_{ji}}{h_{ji}},-p_{ji})$ on $U_i$. We also consider $\{(W_i,-\Lambda_i)\}\in \mathcal{C}^0(\mathcal{U},\mathcal{E}_L^1)$. Then we claim that $(\{(W_i, -\Lambda_i)\},\{(\frac{g_{ij}}{h_{ij}},p_{ij})\})\in \mathcal{C}^0(\mathcal{U},\mathcal{E}_L^1)\oplus \mathcal{C}^1(\mathcal{U},\mathcal{E}_L^0)$ define a \v{C}ech $1$-cocycle in the following \v{C}ech resolution
\begin{center}
$\begin{CD}
\mathcal{C}^0(\mathcal{U},\mathcal{E}_L^3)\\
@AdAA\\ 
\mathcal{C}^0(\mathcal{U},\mathcal{E}_L^2)@>-\delta>>\mathcal{C}^1(\mathcal{U},\mathcal{E}_L^2)\\
@AdAA @AdAA\\
\mathcal{C}^0(\mathcal{U},\mathcal{E}_L^1)@>\delta>>\mathcal{C}^1(\mathcal{U},\mathcal{E}_L^1)@>-\delta>>\mathcal{C}^2(\mathcal{U},\mathcal{E}_L^1)\\
@AdAA @AdAA @AdAA\\
\mathcal{C}^0(\mathcal{U},\mathcal{E}_L^0)@>-\delta>>\mathcal{C}^1(\mathcal{U},\mathcal{E}_L^0)@>\delta>>\mathcal{C}^2(\mathcal{U},\mathcal{E}_L^0)@>-\delta>>\mathcal{C}^3(\mathcal{U},\mathcal{E}_L^0)
\end{CD}$
\end{center}

We have $F_{ij}(1+\epsilon p_{ij})(F_{jk})=F_{ij}$ on $U_i$, which induces (for the detail, \cite{Ser06} p.148)
\begin{align}\label{cocycle1}
\frac{g_{ij}}{h_{ij}}+\frac{g_{jk}}{h_{jk}}-\frac{g_{ik}}{h_{ik}}+\frac{p_{ij}h_{jk}}{h_{jk}}=\frac{g_{ij}}{h_{ij}}+\frac{g_{jk}}{h_{jk}}-\frac{g_{ik}}{h_{ik}}+[p_{ij},\log h_{jk}]=0
\end{align}

We note that $(\frac{g_{ij}}{h_{ij}},p_{ij})$ on $U_j$ is identified with $(\frac{g_{ij}}{h_{ij}}+\frac{p_{ij}h_{jk}}{h_{jk}},p_{ij})$ on $U_k$. (\ref{cocycle1}) means that $\delta(\{(\frac{g_{ij}}{h_{ij}},p_{ij})\})=0$. On the other hand, we have $[\Lambda_0+\epsilon \Lambda_i, T_i^0+\epsilon W_i]=0$. Then we have $[\Lambda_i, T_i^0]+[\Lambda_0,W_i]=0$. In other words, $[\Lambda_0, W_i]+(-1)^2[T_i^0, -\Lambda_i]=0$. Lastly we have, on $U_i$,
\begin{align*}
T_j^0+\epsilon W_j+\epsilon[p_{ij},T_j^0+\epsilon W_j]-T_i^0-\epsilon W_i+\frac{h_{ij}-\epsilon g_{ij}}{h_{ij}^2}[\Lambda_0+\epsilon \Lambda_i,h_{ij}+\epsilon g_{ij}]=0
\end{align*}
By considering the coefficient of $\epsilon$, we have
\begin{align*}
W_j-W_i+[p_{ij},T_j^0]+\frac{1}{h_{ij}}[\Lambda_0,g_{ij}]+\frac{1}{h_{ij}}[\Lambda_{i},h_{ij}]-\frac{g_{ij}}{h_{ij}^2}[\Lambda_0,h_{ij}]=0
\end{align*}
We note that $(W_i,- \Lambda_i)$ on $U_i$ is identified with $(W_i-\frac{1}{h_{ij}}[\Lambda_i,h_{ij}],-\Lambda_i)$ on $U_j$. Then we see that on $U_j$, $(W_j-W_i+\frac{1}{h_{ij}}[\Lambda_i,h_{ij}],-\Lambda_j+\Lambda_i)+([\Lambda_0,\frac{g_{ij}}{h_{ij}}]+(-1)^1[T_j^0,p_{ij}],[\Lambda_0,p_{ij}])=0$ since $[\Lambda_0,\frac{g_{ij}}{h_{ij}}]=\frac{1}{h_{ij}}[\Lambda_0,g_{ij}]-\frac{g_{ij}}{f_{ij}^2}[\Lambda_0,f_{ij}]$. So we have $\delta(\{(W_i,-\Lambda_i)\}+d((\{\frac{g_{ij}}{h_{ij}},p_{ij})\})=0$.  Hence $(\{(W_i, -\Lambda_i)\},\{(\frac{g_{ij}}{f_{ij}},p_{ij})\})$ defines an element in $\mathbb{H}^1(X,\Lambda_0,\mathcal{E}_L^\bullet)$. 

Assume that we have two equivalent first-order deformations $(\mathcal{X},\Lambda, \mathcal{L},\nabla_{\mathcal{L}})$ and $(\mathcal{X}',\Lambda',\mathcal{L}'\nabla_{\mathcal{L}'})$ of $(X,\Lambda_0,L,\nabla)$
, which are represented by $(\{(W_i,-\Lambda_i)\},\{( \frac{g_{ij}}{h_{ij}},p_{ij})\})$ and $(\{(W_i',-\Lambda_i')\},\{ (\frac{g_{ij}'}{h_{ij}},p_{ij}')\})$ respectively. Recall that there is an induced Poisson isomorphism $\alpha_i:(U_i\times Spec(k[\epsilon]),\Lambda_0+\epsilon \Lambda_i)\to (U_i\times Spec(k[\epsilon]), \Lambda_0+\epsilon \Lambda_i')$ which corresponds to $a_i\in \Gamma(U_i,T_X)$ so that $a_i-a_j=p_{ij}-p_{ij}'$ and $\Lambda_i'-\Lambda_i=[\Lambda_0,a_i]$ in the proof of Proposition \ref{3q}. Let us consider $Id+\epsilon a_i:(\mathcal{O}_X(U_i)\otimes k[\epsilon],\Lambda_0+\epsilon \Lambda_i')\to (\mathcal{O}_X(U_i)\otimes k[\epsilon],\Lambda_0+\epsilon \Lambda_i)$.
Then $T_i^0+\epsilon W_i=(Id+\epsilon a_i)(T_i^0+\epsilon W_i')=T_i^0+\epsilon [a_i,T_i^0]+\epsilon W_i'$ so that $W_i'-W_i=[T_i^0, a_i]$ so that $[\Lambda_0,0]+(-1)^1[T_i^0,a_i]=W_i-W_i'$. On the other hand, $(Id+\epsilon a_i)F_{ij}'=F_{ij}$, equivalently, $(Id+\epsilon a_i)(h_{ij}+\epsilon g_{ij}')=h_{ij}+\epsilon g_{ij}$ so that $\frac{g_{ij}}{h_{ij}}-\frac{g_{ij}'}{h_{ij}}=[a_i,\log h_{ij}]$. Since $(0,a_i)$ on $U_i$ is identified with $([a_i,\log h_{ij}],a_i)$ on $U_j$, $-\delta(\{(0,a_i)\})=\{( \frac{g_{ij}}{h_{ij}}-\frac{g_{ij}'}{h_{ij}},p_{ij}-p_{ij}')\}$. Hence $(\{(W_i, -\Lambda_i)\}$, $\{( \frac{g_{ij}}{h_{ij}},p_{ij})\})$ and $(\{(W_i', -\Lambda_i')\}$, $\{( \frac{g_{ij}'}{h_{ij}},p_{ij}')\})$ are cohomologous. This proves (1) in Proposition \ref{spo}.

Now consider the proof of $(2)$ in Proposition \ref{spo}. We keep the notations in the proof of Proposition \ref{lifting}. Consider a small extension $e:0\to (t)\to \tilde{A}\to A\to 0$ in $\bold{Art}$ and let $\eta=(\xi, \mathcal{L},\nabla_\mathcal{L})$ be an infinitesimal deformation of $(X,\Lambda_0,L,\nabla)$ over $A$, where
\begin{center}
$\xi: \begin{CD}
(X,\Lambda_0)@>>> (\mathcal{X},\Lambda)\\
@VVV @VVV\\
Spec(k)@>>> Spec(A)
\end{CD}$
\end{center}
Let $\mathcal{U}=\{U_i\}$ be an affine open cover of $X$ and $\theta_{i}:(U_i \times Spec(A),\Lambda_i)\to (\mathcal{X}|_{U_i},\Lambda|_{U_i})$ be a Poisson isomorphism so that $\theta_{ij}:=\theta_j^{-1}\theta_i:(U_{ij}\times Spec(A),\Lambda_i)\to (U_{ij}\times Spec(A),\Lambda_j)$ is an Poisson isomorphism. Let ${f}_{ij}:(\mathcal{O}_X(U_{ij})\otimes _k {A},\Lambda_j)\to (\mathcal{O}_X(U_{ij})\otimes_k A,\Lambda_i)$ corresponding to $\theta_{ij}$. We may assume that $(\mathcal{L},\nabla_\mathcal{L})$ is given by a system of transition functions $\{F_{ij}\}$, where $F_{ij}$ is a nowhere zero function on $U_{ij}\times Spec(A)$ which is restricted from $U_i\times Spec(A)$ so that $F_{ij}f_{ij}(F_{jk})=F_{ik}$ and $\{T_i\}$, where $T_i\in \Gamma(U_i,T_X)\otimes A$ with $[\Lambda_i,T_i]=0$ and $\tilde{f}_{ij}T_j-T_i+[\Lambda_i,\log F_{ij}]=0$. In order to see if a lifting $(\tilde{\xi},\tilde{\mathcal{L}},{\nabla}_{\tilde{\mathcal{L}}})$ of $(\xi,\mathcal{L},\nabla_\mathcal{L})$ to $Spec(\tilde{A})$ exists,  we choose arbitrarily a collection $\{\tilde{\theta}_{ij}\},\{\tilde{F}_{ij}\},\{\tilde{\Lambda}_i\}$ and $\{\tilde{T}_i\}$ where, for each $i,j,k$:
\begin{enumerate}
\item $\tilde{\theta}_{ij}$ is an automorphism of $U_{ij}\times Spec(\tilde{A})$ which restrict to $\theta_{ij}$ on $U_{ij}\times Spec(A)$, and corresponds to $\tilde{f}_{ij}:\mathcal{O}_X(U_{ij})\otimes \tilde{A}\to \mathcal{O}_X(U_{ij})\otimes \tilde{A}$ as in the proof of Proposition \ref{lifting}.
\item $\tilde{F}_{ij}$ is a nowhere zero function on $U_{ij}\times Spec(\tilde{A})\subset U_i \times Spec(\tilde{A})$ which restricts $F_{ij}$ on $U_{ij}\times Spec(A)\subset U_i\times Spec(A)$.
\item $\tilde{\Lambda}_i$ is a bivector field over $\tilde{A}$ on $U_{i}\times Spec(\tilde{A})$ which restricts $\Lambda_i$ on $U_{i}\times Spec(A)$ as in the proof of Proposition \ref{lifting} so that $\tilde{\Lambda}_i\in \Gamma(U_i,\wedge^2 T_X)\otimes_k \tilde{A}$.
\item $\tilde{T}_i$ is a vector field over $\tilde{A}$ on $U_{i}\times Spec(\tilde{A})$ which restricts to $T_i$ on $U_{i}\times Spec(A)$ so that $\tilde{T}_i\in\Gamma(U_i,T_X)\otimes \tilde{A}$.
\end{enumerate}
Recall that from the proof of Proposition \ref{lifting}, (1) implies that $\tilde{f}_{ij}\tilde{f}_{jk}\tilde{f}_{ik}^{-1}=Id+t\tilde{d}_{ijk}$, and (2) implies that $\tilde{F}_{ij}\tilde{f}_{ij}(\tilde{F}_{jk})\tilde{F}_{ik}^{-1}=1+tg_{ijk}$ for some $g_{ijk}\in \mathcal{O}_X(U_{ijk})$. By considering $(-2\tilde{g}_{ijk},-2\tilde{d}_{ijk})$ to be on $\mathcal{O}_X|_{U_k}\oplus T_X|_{U_k}$, we can show that $\{(-2\tilde{g}_{ijk},-2\tilde{d}_{ijk})\}\in \mathcal{C}^2(\mathcal{U},\mathcal{E}_L^0)$ and $-\delta(\{(-2\tilde{g}_{ijk},-2\tilde{d}_{ijk})\})=0$ (for the detail, see \cite{Ser06} p.149-150).
$(4)$ implies that $[\tilde{\Lambda}_i,\tilde{T}_i]=tS_i$ for some $S_i\in \Gamma(U_i, \wedge^2 T_X)$. Moreover $(2)$ and $(4)$ implies that $\tilde{f}_{ij}\tilde{T}_j-\tilde{T}_i+[\tilde{\Lambda}_i,\log \tilde{F}_{ij}]=0 \mod (t)$ so that $\tilde{f}_{ij}\tilde{T}_j-\tilde{T}_i+[\tilde{\Lambda}_i,\log \tilde{F}_{ij}]=tQ_{ij}$ for some $Q_{ij}\in \Gamma(U_{ij},T_X)$. By considering $(-2Q_{ij},2\Lambda_{ij}')$ to be on $T_X|_{U_j}\oplus \wedge^2 T_X|_{U_j}$, we can show that $\{(-2Q_{ij},2\Lambda_{ij}')\}\in \mathcal{C}^1(\mathcal{U},\mathcal{E}_L^1)$.  We claim that $\alpha:=\{(-2S_i,\Pi_i)\}\oplus \{(-2Q_{ij}, 2\Lambda_{ij}' )\}\oplus \{(-2g_{ijk},-2\tilde{d}_{ijk})\}\in \mathcal{C}^0(\mathcal{U},\mathcal{E}_L^2)\oplus \mathcal{C}^1(\mathcal{U},\mathcal{E}_L^1)\oplus \mathcal{C}^0(\mathcal{U},\mathcal{E}_L^0)$ is a $2$-cocyle in the above \v{C}ech resolution.

We show that $d(\{(-2S_i,\Pi_i)\})=\{([\Lambda_0, -2S_i]+(-1)^3[T_i^0,\Pi_i],[\Lambda_0,\Pi_i])\}=0.$ Indeed, since $t[\Lambda_0, S_i]=[\tilde{\Lambda}_i,[\tilde{\Lambda}_i,\tilde{T}_i]]=[[\tilde{\Lambda}_i,\tilde{\Lambda}_i], \tilde{T}_i]-[\tilde{\Lambda}_i,[\tilde{\Lambda}_i, \tilde{T}_i]]=[t\Pi_i,T_i^0]-t[\Lambda_0, S_i]$, we have $[\Lambda_0, 2S_i]+[-\Pi_i, T_i^0]=0$.

We show that $-\delta(\{-2S_i, \Pi_i)\})+d(\{(-2Q_{ij},2\Lambda_{ij}')\})=0$. We note that $(-2S_i,\Pi_i)$ on $U_i$ is identified with $(-2S_i+[\Pi_i,\log h_{ij}],\Pi_i)$ on $U_j$. Then on $U_j$, we have to show that $-2S_i+[\Pi_i,\log h_{ij}]-(-2S_j)+[\Lambda_0,-2Q_{ij}]+(-1)^2[T_j^0,2\Lambda_{ij}']=0$. Indeed, $2tS_i-2tS_j=2tS_i-2\tilde{f}_{ij}(tS_j)=2[\tilde{\Lambda}_i,\tilde{T}_i]-2\tilde{f}_{ij}[\tilde{\Lambda}_j,\tilde{T}_j]=2[\tilde{\Lambda}_i,\tilde{T}_i]-2[\tilde{f}_{ij}\tilde{\Lambda}_j,\tilde{f}_{ij}\tilde{T}_j]=2[\tilde{\Lambda}_i,\tilde{T}_i]-2[\tilde{\Lambda}_i+t\Lambda_{ij}',\tilde{T}_i-[\tilde{\Lambda}_i,\log\tilde{F}_{ij}]+tQ_{ij}]=-2t[\Lambda_0,Q_{ij}]+2[\tilde{\Lambda}_i,[\tilde{\Lambda}_i, \log \tilde{F}_{ij}]]-2t[\Lambda_{ij}', T_i^0-[\Lambda_0,\log h_{ij}]]=t[\Lambda_0, -2Q_{ij}]+[[\tilde{\Lambda}_i,\tilde{\Lambda}_i],\log \tilde{F}_{ij}]-t[2\Lambda_{ij}', T_j^0]=t([\Lambda_0,- 2Q_{ij}]+[\Pi_i,\log h_{ij}]+[T_j^0,2\Lambda_{ij}']).$

We show that $-\delta(\{(-2Q_{ij},2\Lambda_{ij}')\})+d(\{(-2g_{ijk},-2\tilde{d}_{ijk})\})=0$, equivalently $-\delta(\{(-Q_{ij},\Lambda_{ij}')\})+d(\{(-g_{ijk},-\tilde{d}_{ijk})\})=0$. We note that $(-g_{ijk},-\tilde{d}_{ijk})$ on $U_k$ is identified with $(-g_{ijk}-\tilde{d}_{ijk}\log h_{ki},-\tilde{d}_{ijk})$ on $U_i$. $(-Q_{ij},\Lambda_{ij}')$ on $U_j$ is identified with $(-Q_{ij}-[\Lambda_{ij}',\log h_{ij}],\Lambda_{ij}')$ on $U_i$. $(-Q_{jk}, \Lambda_{jk}')$ on $U_k$ is identified with $(-Q_{jk}-[\Lambda_{jk}', \log h_{ik}], \Lambda_{jk}')$ on $U_i$. We have to show that on $U_i$,
\begin{align*}
Q_{ij}+[\Lambda_{ij}',\log h_{ij}] +Q_{jk}+[\Lambda_{jk}',\log h_{ik}]+Q_{ki}-[\Lambda_0, g_{ijk}+\tilde{d}_{ijk}\log h_{ki}]+(-1)^1[T_i^0, -\tilde{d}_{ijk}]=0
\end{align*}

By Lemma \ref{3n}, on $U_i$, we have $t[T_i^0,\tilde{d}_{ijk}]=\tilde{T}_i-\tilde{f}_{ij}\tilde{f}_{jk}\tilde{f}_{ki}\tilde{T}_i=\tilde{T}_i-\tilde{f}_{ij}\tilde{f}_{jk}(\tilde{T}_k-[\tilde{\Lambda}_k,\log\tilde{F}_{ki}]+tQ_{ki})=\tilde{T}_i-\tilde{f}_{ij}(\tilde{T}_j-[\Lambda_j, \log \tilde{F}_{jk}]+tQ_{jk}-[\tilde{f}_{jk}\tilde{\Lambda}_k,\log \tilde{f}_{jk}\tilde{F}_{ki}]+tQ_{ki}]=\tilde{T}_i-\tilde{f}_{ij}(\tilde{T}_j-[\tilde{\Lambda}_j, \log \tilde{F}_{jk}]-[\tilde{\Lambda}_j+t\Lambda_{jk}',\log \tilde{f}_{jk}\tilde{F}_{ki}]+tQ_{ki}+tQ_{jk})=\tilde{T}_i-(\tilde{T}_i-[\tilde{\Lambda}_i,\log \tilde{F}_{ij}]-[\tilde{f}_{ij}\tilde{\Lambda}_j,\log \tilde{f}_{ij}\tilde{F}_{jk}\cdot \tilde{f}_{ij}\tilde{f}_{jk}\tilde{F}_{ki}]-t[\Lambda_{jk}',\log h_{ki}]+tQ_{ij}+tQ_{jk}+tQ_{ki})=[\tilde{\Lambda}_i,\log \tilde{F}_{ij}]+[\tilde{\Lambda}_i+t\Lambda_{ij}',\log \tilde{f}_{ij}\tilde{F}_{jk}\cdot \tilde{f}_{ij}\tilde{f}_{jk}\tilde{F}_{ki}]+t[\Lambda_{jk}',\log h_{ki}]-t(Q_{ij}+Q_{jk}+Q_{ki})=[\tilde{\Lambda}_i,\log \tilde{F}_{ij}\cdot\tilde{f}_{ij}\tilde{F}_{jk}\cdot\tilde{f}_{ij}\tilde{f}_{jk}\tilde{F}_{ki}]+t[\Lambda_{ij}',\log h_{jk}h_{ki}]+t[\Lambda_{jk}',\log h_{ki}]-t(Q_{ij}+Q_{jk}+Q_{ki})=[\tilde{\Lambda}_i,\log \tilde{F}_{ij}\cdot\tilde{f}_{ij}\tilde{F}_{jk}\cdot\tilde{f}_{ij}\tilde{f}_{jk}\tilde{f}_{ki}\tilde{F}_{ik}^{-1}]+t[\Lambda_{ij}',\log h_{ji}]+t[\Lambda_{jk}',\log h_{ki}]-t(Q_{ij}+Q_{jk}+Q_{ki})=[\tilde{\Lambda}_i,\log (\tilde{F}_{ij}\cdot \tilde{f}_{ij}\tilde{F}_{jk}\cdot \tilde{F}_{ik}^{-1}+\tilde{F}_{ij}\cdot\tilde{f}_{ij}\tilde{F}_{jk}\cdot td_{ijk} \tilde{F}_{ik}^{-1})]+t[\Lambda_{ij}',\log h_{ji}]+t[\Lambda_{jk}',\log h_{ki}]-t(Q_{ij}+Q_{jk}+Q_{ki})=[\tilde{\Lambda}_i,\log (1+tg_{ijk}+th_{ij}h_{jk}\tilde{d}_{ijk}h_{ki})]-t[\Lambda_{ij}',\log h_{ij}]-t[\Lambda_{jk}',\log h_{ik}]-t(Q_{ij}+Q_{jk}+Q_{ki})=(1-tg_{ijk}-th_{ik}\tilde{d}_{ijk}h_{ki})[\tilde{\Lambda}_i, 1+tg_{ijk}+th_{ik}\tilde{d}_{ijk} h_{ki}]-t[\Lambda_{ij}',\log h_{ij}]-t[\Lambda_{jk}',\log h_{ik}]-t(Q_{ij}+Q_{jk}+Q_{ki})=t[\Lambda_0,g_{ijk}+\tilde{d}_{ijk}\log h_{ki}]-t[\Lambda_{ij}',\log h_{ij}]-t[\Lambda_{jk}',\log h_{ik}]-t(Q_{ij}+Q_{jk}+Q_{ki})$.

We claim that given a different arbitrary choice $\{\tilde{\theta}_{ij}'\},\{\tilde{\Lambda}_i'\},\{\tilde{F}_{ij}'\}$ and $\{\tilde{T}_i'\}$ satisfying (1),(2),(3) and (4), the associated $2$-cocycle $\beta:=\{(-2S_i' ,\Pi_i')\}\oplus \{(-2Q_{ij} , 2\Lambda_{ij}'')\} \oplus \{( -2g_{ijk},-2\tilde{d}_{ijk}')\}$ is cohomologous to the $2$-cocycle $\alpha=\{(-2S_i,\Pi_i)\}\oplus \{(-2Q_{ij}, 2\Lambda_{ij}' )\}\oplus \{(-2g_{ijk},-2\tilde{d}_{ijk})\}$ associated with $\{\tilde{\theta_{ij}}\},\{\tilde{\Lambda}_i\},\{\tilde{F}_{ij}\}$ and $\{\tilde{T}_i\}$. We note that $\tilde{T}_i'=\tilde{T}+tT_i'$  for some $T_i'\in \Gamma(U_i,T_X)$ and $\tilde{F}_{ij}'=\tilde{F}_{ij}+tF_{ij}'$ for some $F_{ij}'\in \Gamma(U_{ij},\mathcal{O}_X)$. Recall the notations in the proof of Proposition \ref{lifting}. $tS_i'=[\tilde{\Lambda}_i',\tilde{T}_i']=[\tilde{\Lambda}_i+t\Lambda_i',\tilde{T}_i+tT_i']=tS_i+[\tilde{\Lambda}_i,tT_i']+t[\Lambda_i',\tilde{T}_i]=tS_i+t[\Lambda_0,T_i']+t[\Lambda_i',T_i^0]$ so that $-2S_i+2S_i'=[\Lambda_0,2T_i']+(-1)^2[T_i^0,-2\Lambda_i']$.

$1+tg_{ijk}'=\tilde{F}_{ij}'\tilde{f}_{ij}'(\tilde{F}_{jk}')\tilde{F'}_{ik}^{-1}=(\tilde{F}_{ij}+tF_{ij}')(\tilde{f}_{ij}+tp_{ij}')(\tilde{F}_{jk}+tF_{jk}')(\tilde{F}_{ik}+tF_{ik}')^{-1}=(\tilde{F}_{ij}+tF_{ij}')(\tilde{f}_{ij}(\tilde{F}_{jk})+t[p_{ij}',h_{jk}]+tF_{jk}')(\tilde{F}_{ik}^{-1}-tF_{ik}'h_{ik}^{-2})=\tilde{F}_{ij}\tilde{f}_{ij}(\tilde{F}_{jk})\tilde{F}_{ik}^{-1}+th_{ij}[p_{ij}',h_{jk}]h_{ik}^{-1}-th_{ij}h_{jk}F_{ik}'h_{ik}^{-2}+th_{ij}F_{jk}'h_{ik}^{-1}+tF_{ij}'h_{jk}h_{ik}^{-1}=1+tg_{ijk}+t\frac{1}{h_{jk}}[p_{ij}',h_{jk}]-th_{ki}F_{ik}'+th_{kj}F_{jk}'+th_{ji}F_{ij}'=1+tg_{ijk}+t\frac{1}{h_{jk}}[p_{ij}',h_{jk}]-t\frac{F_{ik}'}{h_{ik}}+t\frac{F_{jk}'}{h_{jk}}+t\frac{F_{ij}'}{h_{ij}}$ so that $-g_{ijk}-(-g_{ijk}')=\frac{1}{h_{jk}}[p_{ij}',h_{jk}]-\frac{F_{ik}'}{h_{ik}}+\frac{F_{jk}'}{h_{jk}}+\frac{F_{ij}'}{h_{ij}}$. Note that $(\frac{{F}_{ij}'}{h_{ij}},p_{ij}')$ on $U_j$ is identified with $\frac{F_{ij}'}{h_{ij}}+[p_{ij}',\log h_{jk}]$ on $U_k$.

$tQ_{ij}'=\tilde{f}_{ij}'\tilde{T}_j'-\tilde{T}_i'+[\tilde{\Lambda}_i',\log \tilde{F}_{ij}']=(\tilde{f}_{ij}+tp_{ij}')(\tilde{T}_j+tT_j')-\tilde{T}_i-tT_i'+[\tilde{\Lambda}_i+t\Lambda_i',\log (\tilde{F}_{ij}+tF_{ij}')]=\tilde{f}_{ij}\tilde{T}_j+t[p_{ij}',T_j^0]+tT_j'-\tilde{T}_i-tT_i'+\frac{\tilde{F}_{ij}-tF_{ij}'}{\tilde{F}_{ij}^2}[\tilde{\Lambda}_i+t\Lambda_i',\tilde{F}_{ij}+tF_{ij}']=tQ_{ij}+t[p_{ij}',T_j^0]+tT_j'-tT_i'-t\frac{F_{ij}'}{h_{ij}^2}[\Lambda_0,h_{ij}]+t\frac{1}{h_{ij}}[\Lambda_0,F_{ij}']+t\frac{1}{h_{ij}}[\Lambda_i',h_{ij}]$ so that $-Q_{ij}-(-Q_{ij}')=T_j'-(T_i'+\frac{1}{h_{ij}}[-\Lambda_i',h_{ij}])+[\Lambda_0,\frac{F_{ij}'}{h_{ij}}]$. Here we note that $(T_i', -\Lambda_i')$ on $U_i$ is identified with $(T_i'+\frac{1}{h_{ij}}[-\Lambda_i',h_{ij}],-\Lambda_{i}')$ on $U_j$. 

Hence $(\{2T_i'\},\{-2\Lambda_i'\})\oplus(\{2\frac{F_{ij}'}{h_{ij}}\},\{2p_{ij}'\})\in \mathcal{C}^0(\mathcal{U},\mathcal{E}_L^1)\oplus \mathcal{C}^1(\mathcal{U},\mathcal{E}_L^0)$ in the above \v{C}ech resolution is mapped to $\alpha-\beta$ so that $\alpha$ is cohomologous to $\beta$. So given a small extension $e:0\to (t)\to \tilde{A}\to A\to 0$ and an infinitesimal deformations $\eta$ of $(X,\Lambda_0,L,\nabla)$ over $A$, we can associate an element $o_\eta(e):=$ the cohomology class of $\alpha\in \mathbb{H}^2(X,\Lambda_0,\mathcal{E}_L^\bullet)$. We note that $o_\eta(e)=0$ if and only if there exists a collection of $\{\tilde{f}_{ij}\}, \{\tilde{\Lambda}_i\}, \{\tilde{F}_{ij}\},\{\tilde{Y_i}\}$ defining an infinitesimal deformation over $\tilde{A}$ which induces $\eta$.

Let's consider the proof of $(3)$ in Proposition \ref{spo}. For given $c(L,\nabla)=(\{T_i^0\}, \{d \log h_{ij}\})\in \mathbb{H}^1(\Omega_X^1\xrightarrow{i_{\Lambda_0}} T_X)$, we define in the \v{C}ech resolutions of $T_X^\bullet$ and $\mathcal{O}_X^\bullet$,
\begin{align*}
c(L,\nabla):\mathbb{H}^1(X,\Lambda_0,T_X^\bullet)&\to \mathbb{H}^2(X,\Lambda_0,\mathcal{O}_X^\bullet),k(\xi)\mapsto c(L,\nabla)(k(\xi))\\
(\{-\Lambda_i\}, \{p_{ij}\})&\mapsto (\{[-\Lambda_i, T_i^0]\},\{[-\Lambda_i,\log h_{ij}]+[T_j^0,p_{ij}]\},\{-\frac{p_{ij}h_{jk}}{h_{jk}}=-[p_{ij},\log h_{jk}]\})\\
\end{align*}
We can check this is well-defined. 
$c(L,\nabla)(k(\xi))=0$ means that there exists $(\{W_i\},\{\frac{g_{ij}}{h_{ij}}\})\in \mathcal{C}^0(\mathcal{U},T_X)\oplus \mathcal{C}^1(\mathcal{U},\mathcal{O}_X)$ such that 
\begin{center}
$\begin{CD}
[\Lambda_0,W_i]=[-\Lambda_i, T_i^0]\\
@A[\Lambda_0,-]AA\\
W_i @>\delta>> W_j-W_i+[\Lambda_0,\frac{g_{ij}}{h_{ij}}]=[-\Lambda_\alpha, \log h_{ij}]+[T_j^0,p_{ij}]\\
@. @A[\Lambda_0,-]AA\\
@. \frac{g_{ij}}{h_{ij}}@>\delta>> \frac{g_{ij}}{h_{ij}}+\frac{g_{jk}}{h_{\beta\gamma}}-\frac{g_{\alpha\gamma}}{h_{\alpha\gamma}}=-\frac{p_{ij} h_{jk}}{h_{jk}}
\end{CD}$
\end{center}
Hence $(L,\nabla)$ deforms along a first-order deformation $\xi$ over $k[\epsilon]$.

\end{proof}
\begin{remark}\label{remark11}
We keep the notations in the proof of Proposition $\ref{spo}$. When $\Lambda_0$ is symplectic, $\mathbb{H}^1(X,\Omega_X^1\xrightarrow{i_{\Lambda_0}} T_X)\cong \mathbb{H}^1(X,\Omega_X^1\xrightarrow{Id}\Omega_X^1)=0$ so that the Poisson Chern class of a Poisson invertible sheaf $(L,\nabla)$ is trivial. Hence $(L,\nabla)$ deforms along whole $\mathbb{H}^1(X,\Lambda_0,T_X^\bullet)\cong \mathbb{H}^1(X,\Omega_X^\bullet)$, where $\Omega_X^\bullet:\Omega_X^1\xrightarrow{\partial}\wedge^2 \Omega_X^1\xrightarrow{\partial}\cdots$. When $\Lambda_0$ is trivial, $\{T_i^0\}$ is a global Poisson vector field $T$, and $\{\Lambda_i\}$ is a global bivector field $\Lambda$ so that we have
\begin{align*}
\mathbb{H}^1(X,\Lambda_0,T_X^\bullet)&\xrightarrow{\cdot c(L,\nabla)} \mathbb{H}^2(X,\Lambda_0,\mathcal{O}_X^\bullet)\cong H^0(X,\wedge^2 T_X)\oplus H^1(X,T_X)\oplus H^2(X,\mathcal{O}_X)\\
(-\Lambda, \{p_{ij}\})&\mapsto ([-\Lambda,T],\{[-\Lambda, \log h_{ij}]+[T,p_{ij}]\},\{-\frac{p_{ij}h_{jk}}{h_{jk}}\})
\end{align*}

\end{remark}

\begin{remark}
Let $(X,\Lambda_0)$ be a compact K\"ahler holomorphic Poisson manifold and $(L,\nabla)$ be a Poisson invertible sheaf on $(X,\Lambda_0)$. In this remark, we describe the map $\mathbb{H}^1(X,\Lambda_0,T_X^\bullet)\xrightarrow{c(L,\nabla)} \mathbb{H}^2(X,\Lambda_0,\mathcal{O}_X^\bullet)$ in terms of Deaulbault resolution. Let the Poisson Chern class of $(L,\nabla)$ be represented by $(\{T_i\},\{d\log f_{ij}\})$ for an open covering $\mathcal{U}=\{U_i\}$ of $X$. Given a $1$-cocycle $(\{-\Lambda_i\},\{p_{ij}\})\in \mathcal{C}^0(\mathcal{U},\wedge^2 T_X)\oplus \mathcal{C}^1(\mathcal{U},T_X)$ in $\mathbb{H}^1(X,\Lambda_0,T_X^\bullet)$, we have
\begin{align*}
c(L,\nabla):\mathbb{H}^1(X,\Lambda_0,T_X^\bullet)&\to \mathbb{H}^2(X,\Lambda_0,\mathcal{O}_X^\bullet)\\
(\{-\Lambda_i\}, \{p_{ij}\})&\mapsto (\{[-\Lambda_i, T_i]\},\{[-\Lambda_i,\log f_{ij}]+[T_j,p_{ij}]\},\{-\frac{p_{ij}f_{jk}}{f_{jk}}=-[p_{ij},\log f_{jk}]\})\\
\end{align*}
There exists $\{e_i\}\in \mathcal{C}^0(\mathcal{U},\mathscr{A}^{0,0}(T_X))$ such that $-\delta\{ e_i\}=\{e_i-e_j\}=\{p_{ij}\}$. Then $(\{-\Lambda_i\},\{p_{ij}\})$ corresponds to $(-\Lambda:=\{[\Lambda_0,e_i]+\Lambda_i\}, \theta:=\bar{\partial}\{e_i\}) \in A^{0,0}(X,\wedge^2 T_X)\oplus A^{0,1}(X,T_X)$. Recall that $(\{d \log f_{ij}\},\{ T_i\})$ corresponds to $(\omega:=\{-\bar{\partial}\partial \log a_i\},T=\{-[\Lambda_0,\log a_i]-T_i\})\in A^{0,1}(X,\Omega_X^1)\oplus A^{0,0}(X,T_X)$ by Remark $\ref{remark12}$. We connect \v{C}ech reolsution with Deaulbault resolution in the following.

{\tiny
\[
\xymatrixrowsep{0.2in}
\xymatrixcolsep{0.05in}
\xymatrix{
&& & H^0(X,\wedge^2 T_X) \ar[ld]  \ar[rr]  &  & C^0(\wedge^2 T_X) \ar[ld]\\
&&A^{0,0}(X,\wedge^2 T_X)  \ar[rr] & & C^0(\mathscr{A}^{0,0}(\wedge^2 T_X))\\
&&& H^0(X,T_X) \ar@{.>}[uu] \ar@{.>}[ld] \ar@{.>}[rr] & & C^0(T_X) \ar[uu] \ar[ld] \ar[rr]^{\delta} && C^1( T_X) \ar[ld]\\
&& A^{0,0}(X,  T_X)  \ar[uu]^{[\Lambda_0,-]} \ar[ld]^{\bar{\partial}} \ar[rr] && C^0(\mathscr{A}^{0,0}( T_X)) \ar[uu]^{[\Lambda_0,-]} \ar[ld]^{\bar{\partial}} \ar[rr]^{\delta} && C^1(\mathscr{A}^{0,0}(T_X))   \\
&A^{0,1}(X, T_X) \ar[rr]  && C^0(\mathscr{A}^{0,1}(T_X)) && C^0(\mathcal{O}_X) \ar@{.>}[ld]\ar@{.>}[uu] \ar@{.>}[rr]^{-\delta} & & C^1(\mathcal{O}_X) \ar[uu] \ar[ld] \ar[rr]^{\delta}& &C^2(\mathcal{O}_X) \ar[ld]  \\
& & A^{0,0}(X,\mathcal{O}_X) \ar@{.>}[uu] \ar@{.>}[rr]  \ar@{.>}[ld]^{\bar{\partial}} && C^0(\mathscr{A}^{0,0}) \ar[uu]^{[\Lambda_0,-]} \ar[rr]^{-\delta} \ar[ld]^{\bar{\partial}} && C^1(\mathscr{A}^{0,0}) \ar[ld]^{\bar{\partial}} \ar[uu]^{[\Lambda_0,-]} \ar[rr]^{\delta} && C^2(\mathscr{A}^{0,0})\\
& A^{0,1}(X,\mathcal{O}_X) \ar[uu]^{[\Lambda_0,-]} \ar[rr] \ar[ld]^{\bar{\partial}}& & C^0(\mathscr{A}^{0,1}) \ar[uu]^{[\Lambda_0,-]} \ar[rr]^{-\delta} \ar[ld]^{\bar{\partial}} &&  C^1(\mathscr{A}^{0,1})\\
A^{0,2}(X,\mathcal{O}_X) \ar[rr]& & C^0(\mathscr{A}^{0,2})
}\]}

On the first floor, we note that $\delta(\{ [e_i, \log f_{ij}]\})=[e_i,\log f_{ij}]-[e_i,\log f_{ik}]+[e_j, \log f_{jk}]=[e_i,\log f_{kj}]+[e_j,\log f_{jk}]=-([e_i-e_j, \log f_{jk}])=\{-\frac{p_{ij} f_{jk}}{f_{jk}}\}$ and $\bar{\partial}\{[e_i,\log f_{ij}]\}=[\bar{\partial} e_i, \log f_{ij}]=[\theta, \log f_{ij}]=\theta (d ( \log f_{ij}))=\theta(\partial \log f_{ij})=\theta (\partial \log a_j)-\theta (\partial \log a_i)=-\delta (\{\theta(-\partial\log  a_i)\})$. Then $\bar{\partial}(\{\theta (-\partial\log  a_i)\})=\bar{\partial}[\theta,-\log a_i]=-[\theta,-\bar{\partial} \log a_i]=-\theta(-\partial\bar{\partial}\log a_i)=\theta(\omega)$. Hence $\{-\frac{p_{ij}f_{jk}}{f_{jk}}\}\in \mathcal{C}^2(\mathcal{U},\mathcal{O}_X)$ corresponds to $\theta(\omega)\in A^{0,2}(X,\mathcal{O}_X)$ by Appendix \ref{appendixA}.

On the second floor, we note that $[\Lambda_0,[e_i,\log f_{ij}]]-[-\Lambda_i,\log f_{ij}]-[T_j,p_{ij}]=[[\Lambda_0,e_i],\log f_{ij}]+[e_i,[\Lambda_0,\log f_{ij}]]+[\Lambda_i,\log f_{ij}]-[T_j,e_i-e_j]=[[\Lambda_0,e_i],\log f_{ij}]+[e_i,T_i-T_j]+[\Lambda_i,\log f_{ij}]-[T_j,e_i-e_j]=[-\Lambda,\log f_{ij}]-[T_i,e_i]+[T_j,e_j]=[-\Lambda, \log a_j]-[-\Lambda,\log a_i]-[T_i,e_i]+[T_j,e_j]=\delta([-\Lambda,\log a_i]+[T_i,e_i])$. Then $[\Lambda_0,[\theta,-\log a_i]]-\bar{\partial}([-\Lambda,\log a_i])-\bar{\partial}([T_i,e_i])=[[\Lambda_0,\theta], -\log a_i]-[\theta,[\Lambda_0,-\log a_i]]+[\bar{\partial}\Lambda,\log a_i]-[\Lambda,\bar{\partial} \log a_i]-[T_i,\theta]=-[\Lambda,\bar{\partial}\log a_i]-[ \theta,T]=-i_{\Lambda}(\partial\bar{\partial}\log a_i)-[\theta,T]=-i_{\Lambda}(\omega)-[\theta,T]$. Hence $\{[-\Lambda_i,\log f_{ij}]+[T_j,p_{ij}]\}$ corresponds to $-i_{\Lambda}\omega-[\theta,T]\in A^{0,1}(X,T_X)$ by Appendix \ref{appendixA}.

On the third floor, we note that $[\Lambda_0,[-\Lambda,\log a_i]]+[\Lambda_0,[T_i,e_i]]+[-\Lambda_i,T_i]=[\Lambda,[\Lambda_0,\log a_i]]+[T_i,[\Lambda_0,e_i]]+[T_i,\Lambda_i]=[\Lambda, [\Lambda_0, \log a_i]]+[T_i,-\Lambda]=[\Lambda, -T]$. Hence $[-\Lambda_i,T_i]$ corresponds to $[\Lambda,T]\in A^{0,0}(X,\wedge^2 T_X)$ by Appendix \ref{appendixA}.

Hence for given $c(L,\nabla)=(\omega, T)$, we have the following map
\begin{align}\label{deault}
c(L,\nabla):\mathbb{H}^1(X,\Lambda_0,T_X^\bullet)\to \mathbb{H}^2(X,\Lambda_0,\mathcal{O}_X^\bullet),
(-\Lambda, \theta)\mapsto ([\Lambda, T],- i_{\Lambda}(\omega)+[T,\theta],\theta(\omega))
\end{align}

\end{remark}
\begin{example}
Let $(X,\Lambda_0)$ be a complex $K3$ surface with trivial Poisson structure $\Lambda_0=0$. Then $\mathbb{H}^1(X,\Lambda_0,T_X^\bullet)=H^1(X,T_X)\oplus H^0(X,\wedge^2 T_X)$ and $\mathbb{H}^2(X,\Lambda_0,\mathcal{O}_X^\bullet)=H^2(X,\mathcal{O}_X)\oplus H^1(X,T_X)\oplus H^2(X,\mathcal{O}_X)$. In this case, any global vector field define a Poisson $\mathcal{O}_X$-module structure on an invertible sheaf $L$ on $(X,0)$. However, there is no nonzero global vector field on $X$. Hence there is only trivial Poisson $\mathcal{O}_X$-module structure on $L$ of $(X,0)$. We denote this Poisson invertible sheaf by $(L,0)$. In this case,  from $(\ref{deault})$, we have a map $(-\Lambda,\theta)\mapsto ([\Lambda,0],-i_{\Lambda}(\omega)+[0,\theta],\theta(\omega))=(0,-i_{\Lambda}(\omega),\theta(\omega))$. When $\Lambda=0$, we get $(0,0,\theta(\omega))$ so that $(L,0)$ deforms along whole $H^1(X,T_X)=20$ when $c(L)$ is trivial and along $19$ dimensional subspace of $H^1(X,T_X)$ when $c(L)$ is nontrivial since $H^1(X,T_X)\to H^2(X,\mathcal{O}_X)$ is surjective. When $\Lambda\ne 0$, in other words, $\Lambda$ is symplectic, in this case, we get $(0,-i_{\Lambda}(\omega),\theta(\omega))$. Since $i_\Lambda:\Omega_X^1\to T_X$ is an isomorphism when $\Lambda\ne 0$, in order that the cohomology class of $-i_{\Lambda}(\omega)=0$, the Chern class $\omega$ of the invertible sheaf $L$ should be trivial.

We conclude that when $L$ has a trivial Chern class, then $(L,0)$ deforms along whole $\mathbb{H}^1(X,\Lambda_0,T_X^\bullet)$. When $L$ has a nontrivial Chern class and $\Lambda\ne 0$, there is no first order deformation  of $(L,0)$ along $(-\Lambda, \theta)$. Hence $(L,0)$ deforms along $19$ dimensional  subspace of $H^1(X,T_X)$ when $c(L)$ is nontrivial.
\end{example}

\begin{example}
Let $X$ be a complex abelian variety induced from a $n$-dimensional vector space $V$ over $\mathbb{C}$ by a lattice and $L$ be an ample invertible sheaf on $X$. We consider a trivial Poisson structure $\Lambda_0=0$ on $X$. Then any global vector field $T=\sum_{i=1}^n c_i\frac{\partial}{\partial z_i}$, $c_i\in \mathbb{C}$ define a Poisson $\mathcal{O}_X$-module structure on $L$. We denote the Poisson invertible sheaf by $(L,T)$. Let $\sum_{p,q=1}^n a_{pq}dz_p\wedge d\bar{z}_q, a_{pq}\in \mathbb{C}$ be the associated $(1,1)$-form of $L$ so that $(a_{pq})$ is positive definite. Then $\mathbb{H}^1(X,\Lambda_0,T_X^\bullet)\xrightarrow{c(L,\nabla)} \mathbb{H}^2(X,\Lambda_0,\mathcal{O}_X^\bullet)$ is described in the following way by Remark $\ref{remark11}$ and $(\ref{deault})$,
\begin{align*}
&H^0(X,\wedge^2 T_X)\oplus H^1(X,T_X)\to H^0(X,\wedge^2 T_X)\oplus H^1(X,T_X)\oplus H^2(X,\mathcal{O}_X)\\
(-\Lambda=&\sum_{i,j=1}^n -x_{ij}\frac{\partial}{\partial z_i}\wedge\frac{\partial}{\partial z_j}, \sum_{i,j=1}^nb_{ij}\frac{\partial}{\partial z_i}\wedge d\bar{z}_j)
\\&\mapsto ( [\Lambda,T], \sum_{i,p,q=1}^n 2x_{ip}a_{pq}\frac{\partial}{\partial z_i}\wedge d\bar{z}_q+[T,\sum_{i,j=1}^n b_{ij}\frac{\partial}{\partial z_i}]\wedge d\bar{z}_j,\sum_{q,i,j=1}^n -a_{qi}b_{ij}d\bar{z}_q\wedge d\bar{z}_j)\\
&=(0,\sum_{i,p,q=1}^n 2x_{ip}a_{pq}\frac{\partial}{\partial z_i}\wedge d\bar{z}_q,\sum_{q,i,j=1}^n -a_{qi}b_{ij}d\bar{z}_q\wedge d\bar{z}_j)
\end{align*}
where $x_{ij}, b_{ij}\in \mathbb{C}$ and $x_{ij}=-x_{ji}$. In order for the second component to be $0$, $\Lambda$ should be $0$ since $(a_{ij})$ is positive definite. Hence $(L,T)$ deforms only in the subspace of $H^1(X,T_X)$ which is $n^2-\binom{n}{2}$ $($see \cite{Ser06}  p.151$)$
\end{example}

\begin{remark}[compare \cite{hor76}]
We can rephrase simultaneous deformations of $(X,\Lambda_0,L,\nabla)$ in the holomorphic setting in the following way. Let $(X,\Lambda_0)$ be a compact holomorphic Poisson manifold. By a family of deformations of Poisson invertible sheaves, we mean a pair of Poisson analytic family $(\mathcal{X},\Lambda, M)$ with $p:\mathcal{X}\to M$ and a Poisson invertible sheaf $(\mathcal{L},\nabla_{\mathcal{L}})$ on $\mathcal{X}$ over $M$. Let $\{U_i\}$ be a finite open covering of $X$, where $z=(z_1,...,z_n)$ is a coordinate of $U_i$ such that $\mathcal{X}$ is locally covered by a finite number of coordinate systems $U_i\times M$ with coordinate change $z_j=f_{jk}(z_k,t)$ and Poisson structures $\Lambda_i(t):=\sum_{p,q=1}^n g_{pq}^i(z_i,t)\frac{\partial}{\partial z_i^p}\wedge \frac{\partial}{\partial z_i^q}$ on $U_i\times M$ with $g_{pq}^i(z_i,t)=-g_{qp}^i(z_i,t)$ $($see \cite{Kim15}$)$ and $(\mathcal{L},\nabla_{\mathcal{L}})$ is represented by transition functions $\{\Psi_{ij}(z_j,t)\}$ and Poisson vector fields $\{ T_i(t):=\sum_{l=1}^n T_i^l(z_i,t)\frac{\partial}{\partial z_i^l}\}$. Then we have 
\begin{align}
&\Psi_{ij}(f_{jk}(z_k,t),t)\cdot \Psi_{jk}(z_k,t)=\Psi_{ik}(z_k,t), \label{ij1} \\
&[\sum_{p,q=1}^ng_{pq}^i \frac{\partial}{\partial z_i^p}\wedge \frac{\partial}{\partial z_i^q},\sum_{l=1}^n T_i^l(z_i,t)\frac{\partial}{\partial z_i^l}]=0 \label{ij2}\\
& \sum_{p=1}^n T_k^p(z_k,t)\frac{\partial}{\partial z_k^p} -\sum_{q=1}^n T_j^q(z_j,t)\frac{\partial}{\partial z_j^q}+[\sum_{p,q=1}^n g_{pq}^j(z_j,t)\frac{\partial}{\partial z_j^p}\wedge \frac{\partial}{\partial z_j^q}, \log \Psi_{jk}(z_k,t)]\label{ij3}\\
&=\sum_{p,q=1}^n T_k^p(z_k,t)\frac{\partial f_{jk}^q}{\partial z_k^p}\frac{\partial}{\partial z_j^q} -\sum_{q=1}^n T_j^q(f_{jk}(z_k,t),t)\frac{\partial}{\partial z_j^q}+\sum_{p,q=1}^n 2\frac{g_{pq}^j(z_j,t)}{\Psi_{jk}}\frac{\partial\Psi_{jk}}{\partial z_j^p}\frac{\partial}{\partial z_j^q}=0.\notag
\end{align}

By considering the coefficient of $\frac{\partial}{\partial z_j^q}$ in $(\ref{ij3})$, we get 
\begin{align*}
\sum_ {p=1}^nT_k^p(z_k,t)\frac{\partial f_{jk}^q}{\partial z_k^p}-T_j^q(f_{jk}(z_k,t),t)+\sum_{p=1}^n 2\frac{g_{pq}^j(z_j,t)}{\Psi_{jk}}\frac{\partial\Psi_{jk}}{\partial z_j^p}=0
\end{align*}
By taking the derivative with respect to $t$, we get
\begin{align}\label{ij4}
\sum_{p=1}^n \frac{\partial T_k^p}{\partial t}\frac{\partial f_{jk}^q}{\partial z_k^p}+&\sum_{p=1}^n T_k^p\frac{\partial}{\partial z_k^p}\left( \frac{\partial f_{jk}^q}{\partial t}\right)-\sum_{p=1}^n \frac{\partial T_j^q}{\partial z_j^p}\frac{\partial f_{jk}^p}{\partial t}-\frac{\partial T_j^q}{\partial t}\\&+\sum_{p=1}^n 2\frac{1}{\Psi_{jk}^2}(\frac{\partial g_{pq}^j}{\partial t}\Psi_{jk}-g_{pq}^j\frac{\partial \Psi_{jk}}{\partial t})\frac{\partial \Psi_{jk}}{\partial z_j^p}
+\sum_{p=1}^n 2\frac{g_{pq}^j}{\Psi_{jk}}\frac{\partial}{\partial z_j^p}\left( \frac{\partial \Psi_{jk}}{\partial t} \right)=0\notag
\end{align}
We claim that $(\{ T'_j=\sum_{p=1}^n \frac{\partial T_j^p(z_j,t)}{\partial t}\frac{\partial}{\partial z_j^p},-\Lambda_j'=\sum_{p,q=1}^n-\frac{\partial g_{pq}^j(z_j,t)}{\partial t}\frac{\partial}{\partial z_j^p}\wedge\frac{\partial}{\partial z_j^q}\},\{\Psi'_{jk}=\frac{1}{\Psi_{jk}}\frac{\partial \Psi_{jk}}{\partial t},-\theta_{jk}=-\sum_{p=1}^n\frac{\partial{f}_{jk}^p(z_k,t)}{\partial t}\frac{\partial}{\partial z_j^p}\}) \in \mathcal{C}^0(\mathcal{U},\mathcal{E}_L^1)\oplus \mathcal{C}^1(\mathcal{U},\mathcal{E}_L^0)$ define a $1$-cocycle. By taking the derivative of $(\ref{ij1})$, we can show that $\Psi_{jk}'-\Psi_{ik}'+\Psi_{ij}'=-[\theta_{jk},\log \Psi_{ij}]$ so that $\delta(\{\Psi_{ij}'\})=0$ $($see \cite{hor76}$)$. By taking the derivative of $(\ref{ij2})$,  we have $[\Lambda_j',T_j]+[\Lambda, T_j']=0$ so that $[\Lambda,T_j']+(-1)^2[T_j,-\Lambda_j']=0$.
We note that $(T_k',-\Lambda_k')$ on $U_k$ is translated to $(T_k'+[\Lambda_k',\log \Psi_{jk}],-\Lambda_k')$ on $U_j$. So it remains to show that $ T_k'+[\Lambda_k',\log \Psi_{jk}]-T_j'+[\Lambda, \Psi_{jk}']+(-1)^1[T_j, -\theta_{jk}]=0$, equivalently,
\begin{align*}
\sum_{p=1}^n \frac{\partial T_k^p}{\partial t}\frac{\partial}{\partial z_k^p}+&[\sum_{p,q=1}^n \frac{\partial g_{pq}^k}{\partial t}\frac{\partial}{\partial z_k^p}\wedge \frac{\partial }{\partial z_k^q},\log \Psi_{jk}]-\sum_{q=1}^n \frac{\partial T_j^q}{\partial t}\frac{\partial}{\partial z_j^q}\\
&+[\sum_{p,q=1}^n g_{pq}^j\frac{\partial}{\partial z_j^p}\wedge \frac{\partial}{\partial z_j^q},\frac{1}{\Psi_{jk}}\frac{\partial \Psi_{jk}}{\partial t}]-[\sum_{p=1}^n\frac{\partial f_{jk}^p}{\partial t}\frac{\partial}{\partial z_j^p},\sum_{q=1}^n T_j^q\frac{\partial}{\partial z_j^q}]=0
\end{align*}
which follows from $(\ref{ij4})$.
Hence we have a characteristic map $\rho:T_t(B)\to \mathbb{H}^1(X_t,\Lambda_t,\mathcal{E}_L^\bullet)$.
\end{remark}

\section{Deformations of sections of a Poisson invertible sheaf $(L,\nabla)$ in flat Poisson deformations}\label{section7}

The formalism of deformations of section of an invertible sheaf in flat deformations presented in \cite{Ser06} (see p.152-153) can be extended to Poisson deformations.  Let $(X,\Lambda_0)$ be a nonsingular projective Poisson variety and $(L,\nabla)$ be a Poisson invertible sheaf on $(X,\Lambda_0)$. Let $(L,\nabla)$ be given by transition functions $\{f_{ij}\}\in \mathcal{C}^1(\mathcal{U},\mathcal{O}_X^*)$ and Poisson vector fields $\{T_i\}\in \mathcal{C}^0(\mathcal{U},T_X)$ for an affine open covering $\mathcal{U}=\{U_i\}$ of $X$. We define a homomorphism of complex of sheaves in the following way (see Remark \ref{lco})
\begin{center}
$\begin{CD}
\mathcal{E}_L^0@>>> \mathcal{E}_L^1@>>>\mathcal{E}_L^2\\
@VM_0VV @VM_1VV@VM_2VV\\
\mathbb{H}^0(X,\Lambda_0,L^\bullet,\nabla)^\vee\otimes L@>v_0:=\nabla>> \mathbb{H}^0(X,\Lambda_0,L^\bullet,\nabla)^\vee\otimes L\otimes T_X@>v_1>> \mathbb{H}^0(X,\Lambda_0,L^\bullet,\nabla)^\vee\otimes L\otimes \wedge^2 T_X \\
\end{CD}$
\end{center}
Consider a section $\eta \in \Gamma(U,\mathcal{E}_L^p)$, where $U\subset X$ is an open set. It is given by a system $(\{a_i,b_i\})$ where $a_i\in \Gamma(U\cap U_i,\wedge^pT_X)$, $b_i\in \Gamma(U\cap U_i, \wedge^{p+1} T_X)$ such that $b_j=b_i$ and $a_j-a_i=\frac{1}{f_{ij}}[b_i, f_{ij}]$. Then for every $s=\{s_i\}\in \mathbb{H}^0(X,\Lambda_0,L^\bullet,\nabla)$ so that $-[\Lambda_0,s_i]+s_i T_i=0$, we let
\begin{align*}
M_p(\eta)(s_i)=s_i a_i+[b_i,s_i]
\end{align*}
We show that this is well-defined, in other words, $M_p(\eta)(s_i)$ and $M_p(\eta)(s_j)$ define a same element so that $f_{ij}M_p(\eta)(s_j)=M_p(\eta)(s_i)$. Indeed, on $U\cap U_i\cap U_j$,
\begin{align*}
f_{ij}M_p(\eta)(s_j)=f_{ij}(s_j a_j+[b_j, s_j])=s_i a_j+f_{ij} [b_j, s_j]=s_i(a_i+\frac{[b_i,f_{ij}]}{f_{ij}})+f_{ij}[b_j, s_j]\\
=s_i a_i+s_j[b_i,f_{ij}]+f_{ij}[b_j,s_j]=s_i a_i+[b_i,f_{ij}s_j]=s_i a_i+[b_i, s_i]=M_p(\eta)(s_i)
\end{align*}
 
Therefore the functions $M_p(\eta)(s_i)\in \Gamma(U\cap U_i,\wedge^pT_X)$ patch together to define a section $M_p(\eta)(s)\in \Gamma(U, L\otimes \wedge^i T_X)$. This defines $M_i$. 
Consider the induced linear map
\begin{align*}
M_1:\mathbb{H}^1(X,\Lambda_0,\mathcal{E}_L^\bullet)\to Hom(\mathbb{H}^0(X,\Lambda_0,L^\bullet,\nabla),\mathbb{H}^1(X,\Lambda_0,L^\bullet,\nabla))
\end{align*}
Let $\eta\in \mathbb{H}^1(X,\Lambda_0,\mathcal{E}_L)$ be represented by the \v{C}ech cocyle $(\{b_i, c_i\})\oplus(\{a_{ij},d_{ij}\})\in \mathcal{C}^0(\mathcal{U},\mathcal{E}_L^1)\oplus\mathcal{C}^1(\mathcal{U},\mathcal{E}_L^0)$. We consider $(a_{ij},d_{ij})$ in $\mathcal{O}_X|_{U_j}\oplus T_X|_{U_j}$. Then
\begin{align*}
M_1(\eta):\mathbb{H}^0(X,\Lambda_0,L^\bullet,\nabla)&\to \mathbb{H}^1(X,\Lambda_0,L^\bullet,\nabla)\\
                    \{s_i\}&\mapsto \overline{\{s_i b_i+[c_i, s_i]\}\oplus \{s_j a_{ij}+[d_{ij},s_j]\}}
\end{align*}

\begin{definition}
Let $A\in \bold{Art}$ and $(\mathcal{X},\Lambda,\mathcal{L},\nabla_{\mathcal{L}})$ be an infinitesimal deformation of $(X,\Lambda_0,X,L)$ over $Spec(A)$. Then we say that a section $s\in \mathbb{H}^0(X,\Lambda_0,L^\bullet,\nabla)$ extends to $(\mathcal{L},\nabla_{\mathcal{L}})$ if
\begin{align*}
s\in Im[\mathbb{H}^0(\mathcal{X},\Lambda,\mathcal{L}^\bullet,\nabla_\mathcal{L})\to \mathbb{H}^0(X,\Lambda_0,L^\bullet,\nabla)]
\end{align*}
\end{definition}

\begin{proposition}\label{sec2}
Let $(X,\Lambda_0)$ be a nonsingular projective Poisson variety, $(L,\nabla)$ a Poisson invertible sheaf on $(X,\Lambda_0)$, and $(\mathcal{X},\Lambda,\mathcal{L},\nabla_{\mathcal{L}})$ be a first order simultaneous deformation of $(X,\Lambda_0)$ and $(L,\nabla)$ over $Spec(k[\epsilon])$ defined by a cohomology class $\eta\in \mathbb{H}^1(X,\Lambda_0,\mathcal{E}_L^\bullet)$. Then a section $s\in \mathbb{H}^0(X,\Lambda_0,L^\bullet,\nabla)$ extends to a section $\tilde{s}\in\mathbb{H}^0(\mathcal{X},\Lambda,\mathcal{L}^\bullet,\nabla_{\mathcal{L}})$ of $(\mathcal{L},\nabla_{\mathcal{L}})$ if and only if $s\in ker(M_1(\eta))$.
\end{proposition}

\begin{proof}
As above, let $(L,\nabla)$ be represented by $(\{f_{ij}\},\{T_i\})\in \mathcal{C}^1(\mathcal{U},\mathcal{O}_X^*)\oplus \mathcal{C}^0(\mathcal{U},T_X)$ and $\eta=(\mathcal{X},\Lambda, \mathcal{L},\nabla_{\mathcal{L}})$ by $\{(b_i,c_i)\}\oplus \{(a_{ij},d_{ij})\}\in \mathcal{C}^0(\mathcal{U},\mathcal{E}_L^1)\oplus \mathcal{C}^1(\mathcal{U},\mathcal{E}_L^0)$ for an affine open covering $\mathcal{U}=\{U_i\}$ of $X$ so that we have Poisson isomorphisms $\theta_{ij}:=Id+\epsilon d_{ij}:( \mathcal{O}_X(U_i),\Lambda_0-\epsilon c_j)\to (\mathcal{O}_X(U_i),\Lambda_0-\epsilon c_i)$, and $(\mathcal{L},\nabla)$ is represented by $F_{ij}=f_{ij}+\epsilon a_{ij}f_{ij}$ and $Y_i=T_i+\epsilon b_i$ (see the proof of (1) in Proposition \ref{spo}). Let's assume that $s\in \mathbb{H}^0(X,\Lambda_0,L^\bullet,\nabla)$ is represented by $\{s_i\},s_i\in \Gamma(U_i,\mathcal{O}_X)$, such that $s_i=f_{ij}s_j$ and $-[\Lambda_0,s_i]+s_iT_i=0$. Then $M_1(\eta)(s)$ is represented by $\{s_i b_i+[c_i, s_i]\}\oplus \{s_j a_{ij}+[d_{ij},s_j]\}\in \mathcal{C}^0(\mathcal{U},L)\oplus \mathcal{C}^1(\mathcal{U},L\otimes T_X)$.

In order for $s$ to extend to a section $\tilde{s}\in \mathbb{H}^0(\mathcal{X},\Lambda,\mathcal{L},\nabla_{\mathcal{L}})$, it is necessary and sufficient that there exist $\{t_i\}, t_i\in \Gamma(U_i,\mathcal{O}_X)$ such that $\theta_{ji}(F_{ij})\cdot (s_j+\epsilon t_j)=\theta_{ji}(s_i +\epsilon t_i)$ on $U_j$, and $-[\Lambda_0-\epsilon c_i, s_i+\epsilon t_i]+(s_i+\epsilon t_i)(T_i+\epsilon b_i)=0$. Then we have $(1-\epsilon d_{ij})(f_{ij}+\epsilon f_{ij} a_{ij})\cdot (s_j+\epsilon t_j)=(1-\epsilon d_{ij})(s_i+\epsilon t_i)$ so that $(f_{ij}-\epsilon d_{ij}f_{ij}+\epsilon f_{ij} a_{ij})(s_j+\epsilon t_j)=(1-\epsilon d_{ij})(s_i+\epsilon t_i)$. By considering the coefficient of $\epsilon$, we get $f_{ij}t_j-s_j d_{ij}f_{ij}+s_i a_{ij}=t_i-d_{ij} s_i$. Then $d_{ij}s_{i}-s_{j}d_{ij} f_{ij}+s_{i} a_{ij}=t_{i}-f_{ij}t_j$ and so $d_{ij}(f_{ij}s_{j})-s_{j}d_{ij} f_{ij}+s_{i} a_{ij}=t_{i}-f_{ij}t_j$. Hence $f_{ij} d_{ij} s_j+s_i a_{ij}=t_{i}-f_{ij} t_j$ so that $d_{ij}s_j+s_j a_{ij}=s_ja_{ij}+[d_{ij},s_j]=f_{ij}t_{i}-t_j$ on $U_j$.

On the other hand, $-[\Lambda_0-\epsilon c_i, s_i+\epsilon t_i]+(s_i+\epsilon t_i)(T_i+\epsilon b_i)=0$, which means that $-[\Lambda_0,t_i]+[c_i, s_i]+s_i b_i+t_i T_i=0 $. Then $-[\Lambda_0, -t_i]-t_i T_i=s_i b_i+[c_i, s_i].$ Hence $\{s_i b_i+[c_i, s_i]\}\oplus \{s_j a_{ij}+[d_{ij},s_j]\}$ is a coboundary so that $M_1(\eta)(s)=0$.
\end{proof}

\appendix

\section{Isomorphism of the second cohomology groups between \v{C}ech and Deaulbault resolutions}\label{appendixA}
Let $(X,\Lambda_0)$ be a compact holomorphic Poisson manifold. In this appendix, we explicitly describe an isomorphism of the second cohomology groups between \v{C}ech and Deaulbault resolutions of the complex of sheaves $\mathcal{O}_X\xrightarrow{[\Lambda_0,-]} T_X\xrightarrow{[\Lambda_0,-]} \wedge^2 T_X\xrightarrow{[\Lambda_0,-]}\cdots$. We connect \v{C}ech reolsution with Deaulbault resolution in the following (here $\delta$ is the \v{C}ech map).

{\tiny
\[
\xymatrixrowsep{0.2in}
\xymatrixcolsep{0.05in}
\xymatrix{
&& & H^0(X,\wedge^2 T_X) \ar[ld]  \ar[rr]  &  & C^0(\wedge^2 T_X) \ar[ld]\\
&&A^{0,0}(X,\wedge^2 T_X)  \ar[rr] & & C^0(\mathscr{A}^{0,0}(\wedge^2 T_X))\\
&&& H^0(X,T_X) \ar@{.>}[uu] \ar@{.>}[ld] \ar@{.>}[rr] & & C^0(T_X) \ar[uu] \ar[ld] \ar[rr]^{\delta} && C^1( T_X) \ar[ld]\\
&& A^{0,0}(X,  T_X)  \ar[uu]^{[\Lambda_0,-]} \ar[ld]^{\bar{\partial}} \ar[rr] && C^0(\mathscr{A}^{0,0}( T_X)) \ar[uu]^{[\Lambda_0,-]} \ar[ld]^{\bar{\partial}} \ar[rr]^{\delta} && C^1(\mathscr{A}^{0,0}(T_X))   \\
&A^{0,1}(X, T_X) \ar[rr]  && C^0(\mathscr{A}^{0,1}(T_X)) && C^0(\mathcal{O}_X) \ar@{.>}[ld]\ar@{.>}[uu] \ar@{.>}[rr]^{-\delta} & & C^1(\mathcal{O}_X) \ar[uu] \ar[ld] \ar[rr]^{\delta}& &C^2(\mathcal{O}_X) \ar[ld]  \\
& & A^{0,0}(X,\mathcal{O}_X) \ar@{.>}[uu] \ar@{.>}[rr]  \ar@{.>}[ld]^{\bar{\partial}} && C^0(\mathscr{A}^{0,0}) \ar[uu]^{[\Lambda_0,-]} \ar[rr]^{-\delta} \ar[ld]^{\bar{\partial}} && C^1(\mathscr{A}^{0,0}) \ar[ld]^{\bar{\partial}} \ar[uu]^{[\Lambda_0,-]} \ar[rr]^{\delta} && C^2(\mathscr{A}^{0,0})\\
& A^{0,1}(X,\mathcal{O}_X) \ar[uu]^{[\Lambda_0,-]} \ar[rr] \ar[ld]^{\bar{\partial}}& & C^0(\mathscr{A}^{0,1}) \ar[uu]^{[\Lambda_0,-]} \ar[rr]^{-\delta} \ar[ld]^{\bar{\partial}} &&  C^1(\mathscr{A}^{0,1})\\
A^{0,2}(X,\mathcal{O}_X) \ar[rr]& & C^0(\mathscr{A}^{0,2})
}\]}

Let $(\{a_{ijk}\},\{b_{ij}\},\{c_i\})\in \mathcal{C}^2(\mathcal{U},\mathcal{O}_X)\oplus \mathcal{C}^1(\mathcal{U},T_X)\oplus \mathcal{C}^0(\mathcal{U},\wedge^2 T_X)$ be a $2$-cocycle, where $\mathcal{U}=\{U_i\}$ is an open cover of $X$ and $U_i$ is an open ball.

 On the first floor, since $-\delta(\{a_{ijk}\})=0$, there exists $\{d_{ij}\}\in \mathcal{C}^1(\mathcal{U},\mathscr{A}^{0,0})$ such that $\delta(\{d_{ij}\})=\{a_{ijk}\}$. Since $-\delta(\{\bar{\partial} d_{ij}\})=-\bar{\partial}(\delta(\{d_{ij}\})=\{-\bar{\partial}a_{ijk}\}=0,$ there exists $\{u_i\}\in \mathcal{C}^0(\mathcal{U},\mathscr{A}^{0,1})$ such that $-\delta(\{u_i\})=\{\bar{\partial} d_{ij}\}$. Since $-\delta(\{\bar{\partial}u_i\})=\{\bar{\partial}\bar{\partial}d_{ij}\}=0$, we have $\{\bar{\partial} u_i\}\in A^{0,2}(X,\mathcal{O}_X)$.
 
 On the second floor, since $\delta(\{[\Lambda_0,d_{ij}]-b_{ij}])=\{[\Lambda_0,a_{ijk}]\}-\delta(\{b_{ij}\})=0$, there exists $\{f_i\}\in\mathcal{C}^0(\mathcal{U},\mathscr{A}^{0,0}(T_X))$ such that $\delta(\{f_i\})=\{[\Lambda_0,d_{ij}]-b_{ij}]\}$. Since $\delta(\{[\Lambda_0,u_i]-\bar{\partial}f_i\})=\{-[\Lambda_0,\bar{\partial}d_{ij}]-\bar{\partial}[\Lambda_0,d_{ij}]\}=0$, we have $\{[\Lambda_0,u_i]-\bar{\partial}f_i\}\in A^{0,1}(X,T_X)$.
 
 On the third floor, since $\delta(\{-c_i-[\Lambda_0,f_i]\})=-\delta(\{c_i\})+\{[\Lambda_0,b_{ij}]\}=0$, we have $\{-c_i-[\Lambda_0,f_i]\}\in A^{0,0}(X,\wedge^2 T_X)$.
 
 We claim that $(\{\bar{\partial}u_i\},\{[\Lambda_0,u_i]-\bar{\partial}f_i\},\{-c_i-[\Lambda_0,f_i]\})\in A^{0,2}(X,\mathcal{O}_X)\oplus A^{0,1}(X,T_X)\oplus A^{0,0}(X,\wedge^2 T_X)$ define a $2$-cocycle in the Deaulbault reolution. Indeed, $\bar{\partial}\bar{\partial}u_i=0, [\Lambda_0,\bar{\partial}u_i]+\bar{\partial}([\Lambda_0,u_i]-\bar{\partial} f_i)=[\Lambda_0,\bar{\partial} u_i]-[\Lambda_0,\bar{\partial}u_i]=0$. $[\Lambda_0,[\Lambda_0,u_i]-\bar{\partial}f_i]-\bar{\partial}c_i-\bar{\partial}[\Lambda_0, f_i]=-[\Lambda_0, \bar{\partial} f_i]+[\Lambda_0, \bar{\partial}f_i]=0$, and $[\Lambda_0, -c_i-[\Lambda_0,f_i]]=[\Lambda_0, -c_i]=0$.
 
 We define an isomorphism of the second cohomology groups between \v{C}ech and Deaulbault resolutions by the correspondence
 \begin{align}\label{correspondence}
 (\{a_{ijk}\},\{b_{ij}\},\{c_i\})\to (\{\bar{\partial}u_i\},\{[\Lambda_0,u_i]-\bar{\partial}f_i\},\{-c_i-[\Lambda_0,f_i]\})
 \end{align}
We claim that this is well-defined. In other words, equivalent $2$-cocycles in \v{C}ech resolution corresponds to equivalent $2$-cocycles in Deaulbault resolution. Let $ (\{a_{ijk}'\},\{b_{ij}'\},\{c_i'\})$ be an equivalent $2$-cocycle so that there exists $\{t_{ij}\}\in \mathcal{C}^1(\mathcal{U},\mathcal{O}_X)$ and $\{r_i\}\in \mathcal{C}^0(\mathcal{U},T_X)$ such that $\delta(\{t_{ij}\})=\{a_{ijk}-a_{ijk}'\}$, $\delta(\{r_i\})+\{[\Lambda_0,t_{ij}]\}=\{b_{ij}-b_{ij}'\}$, and $\{[\Lambda_0,r_i]\}=\{c_i-c_i'\}$. Since $\delta(\{d_{ij}-d_{ij}'-t_{ij}\})=\{a_{ijk}-a_{ijk}'\}-\delta(\{t_{ij}\})=0$, there exists $\{s_i\}\in \mathcal{C}^0(\mathcal{U},\mathscr{A}^{0,0})$ such that $-\delta(\{s_i\})=\{d_{ij}-d_{ij}'-t_{ij}\}$. Since $-\delta(\{u_i-u_i'-\bar{\partial} s_i\})=\{\bar{\partial} (d_{ij}-d_{ij}')-\bar{\partial}(d_{ij}-d_{ij}'-t_{ij})\}=0$, $\{u_i-u_i'-\bar{\partial} s_i\}\in A^{0,1}(X,\mathcal{O}_X)$.

On the other hand, since $\delta(\{f_i'-f_i-[\Lambda_0,s_i]-r_i\})=\{[\Lambda_0,d_{ij}'-d_{ij}]-(b_{ij}'-b_{ij}) +[\Lambda_0,d_{ij}-d_{ij}'-t_{ij}]\}-\delta(\{r_i\})=\{b_{ij}-b_{ij}'-[\Lambda_0,t_{ij}]\}-\delta(\{r_i\})=0$, we have $\{f_i'-f_i-[\Lambda_0,s_i]-r_i\}\in A^{0,0}(X,T_X)$.

We claim that $(\{u_i-u_i'-\bar{\partial}s_i\},\{f_i'-f_i-[\Lambda_0,s_i]-r_i\})$ is mapped to $(\{\bar{\partial}(u_i-u_i')\}, \{[\Lambda_0,u_i-u_i']-\bar{\partial}(f_i-f_i')\}, \{-(c_i-c_i')-[\Lambda_0,f_i-f_i']\})$. Indeed, $[\Lambda_0,f_i'-f_i-[\Lambda_0,s_i]-r_i]=-[\Lambda_0,f_i-f_i']-[\Lambda_0,r_i]=-[\Lambda_0,f_i-f_i']-(c_i-c_i')$, and $[\Lambda_0, u_i-u_i'-\bar{\partial} s_i]+\bar{\partial}(f_i'-f_i-[\Lambda_0,s_i]-r_i)=[\Lambda_0,u_i-u_i']-\bar{\partial}(f_i-f_i')-[\Lambda_0, \bar{\partial} s_i]-\bar{\partial}[\Lambda_0, s_i]=0$. Hence $(\{a_{ijk}\},\{b_{ij}\},\{c_i\})$ is cohomologous to $(\{a_{ijk}'\},\{b_{ij}'\},\{c_i'\})$ so that the correspondence $(\ref{correspondence})$ is well-defined.

\section{Deformations of Poisson vector bundles}\label{appendixB}

\begin{definition}\label{ve3}
Let $(X,\Lambda_0)$ be an algebraic Poisson scheme. A Poisson vector bundle $F$ on $(X,\Lambda_0)$ is a Poisson $\mathcal{O}_X$-module which is locally free of finite rank so that $F$ is equipped with a flat Poisson connection $\nabla$ $($see Definition $\ref{lco2})$. In this case, we denote the Poisson vector bundle by $(F,\nabla)$. Assume that $(X,\Lambda_0)$ is a Poisson scheme over $S$. Then $(F,\nabla)$ is called a Poisson vector bundle over $S$ if the associated Poisson connection $\{-,-\}_F:\mathcal{O}_X\otimes_k F\to F$ is $\mathcal{O}_S$-linear.
\end{definition}

\begin{remark}\label{ve1}
Let $(X,\Lambda_0)$ be a nonsingular Poisson variety and $(F,\nabla)$ be a Poisson vector bundle of rank $n$. Let $\,\,\mathcal{U}=\{U_\alpha\}$ be an open covering of $X$ such that $F|_{U_\alpha}\cong \oplus_{i=1}^n \mathcal{O}_X|_{U_\alpha} $ with the transition matrices $\{F^{\alpha\beta}=(F_{ij}^{\alpha\beta})|F_{ij}^{\alpha\beta}\in \Gamma(U_{\alpha\beta},\mathcal{O}_X)\}$ where each matrix $F^{\alpha\beta}$ defines a linear transformation $c^{\alpha\beta}:\oplus_{i=1}^n \mathcal{O}_X|_{U_\beta}\xrightarrow{(F_{ij}^{\alpha\beta})} \oplus_{i=1}^n \mathcal{O}_X|_{U_\alpha}$ by the rule $c^{\alpha\beta}(e_j^{\beta})=\sum_{i=1}^n F_{ij}^{\alpha\beta}e_i^\alpha$. Here $e_i^\alpha=(0,\cdots,\overset{i-th}{1},\cdots,0)\in \oplus_{i=1}^n \Gamma(U_\alpha, \mathcal{O}_X)$ and $e_j^\beta=(0,\cdots,\overset{j-th}{1},\cdots,0)\in \oplus_{i=1}^n \Gamma(U_\beta,\mathcal{O}_X)$. Then 
\begin{enumerate}
\item $F^{\alpha\beta}F^{\beta\gamma}=F^{\alpha\gamma}$, equivalently for any pair $(i,j)$, $\sum_{j=1}^nF_{ij}^{\alpha\beta}F_{jk}^{\beta\gamma}=F_{ik}^{\alpha\gamma}$.
\item
Since $F|_{U_\alpha}\cong\oplus_{i=1}^n \mathcal{O}_X|_{U_\alpha}$, the flat Poisson connection $\nabla$ can be locally described as $\nabla=v_0:\bigoplus_{i=1}^n \mathcal{O}_X|_{U_\alpha}\to \bigoplus_{i=1}^n T_X|_{U_\alpha}$. Let $v_0(e_j^\alpha)=\sum_{i=1}^n T_{ij}^\alpha e_i^\alpha=(T_{1j}^\alpha,..., T_{nj}^\alpha)\in \oplus_{i=1}^n \Gamma(U_\alpha,T_X)$. Since $v_0$ is flat, by Remark $\ref{lco}$, ${v}_1(v_0(e_j^\alpha))={v}_1(\sum_{i=1}^n T_{ij}^\alpha e_i^\alpha)=\sum_{i=1}^n-[T_{ij}^\alpha,\Lambda_0]e_i^\alpha-\sum_{i=1}^n(T_{ij}^\alpha\wedge v_0(e_i^\alpha))=\sum_{i=1}^n-[T_{ij}^\alpha,\Lambda_0]e_i^\alpha-\sum_{i,k=1}^n(T_{ij}^\alpha\wedge T_{ki}^\alpha)e_k^\alpha=0$. By considering the coefficient of $e_i^\alpha$, we have $-[T_{ij}^\alpha,\Lambda_0]-\sum_{k=1}^n T_{kj}^\alpha\wedge T_{ik}^\alpha=0$. Hence, for any pair $(i,j)$, we have 
\begin{align}\label{ve8}
[\Lambda_0,T_{ij}^\alpha]+\sum_{k=1}^n T_{ik}^\alpha \wedge T_{kj}^\alpha=0.
\end{align}
\item We will find a relation between $(T_{ij}^\alpha)$ and $(T_{ij}^\beta)$. Let $\{-,-\}$ be the Poisson bracket induced form $\Lambda_0$ and $\{-,-\}_\alpha$ be the Poisson connection on $F|_{U_\alpha}\cong \oplus_{i=1}^n \mathcal{O}_X|_{U_\alpha}$. For any local section $f\in \mathcal{O}_X$, $\{f, c^{\alpha\beta}(e_j^\beta)\}_\alpha=\sum_{i=1}^n \{f, F_{ij}^{\alpha\beta}e_i^\alpha\}_\alpha=\sum_{i=1}^n \{f,F_{ij}^{\alpha\beta}\}e_i^\alpha+\sum_{i,k=1}^n F_{ij}^{\alpha\beta}T_{ki}^\alpha(f)e_k^\alpha=c^{\alpha\beta}(\{f,e_j^\beta\}_\beta)=c^{\alpha\beta}(\sum_{k=1}^n T_{kj}^\beta(f)e_k^\beta)=\sum_{k=1}^n T_{kj}^\beta(f)c^{\alpha\beta}(e_k^\beta)=\sum_{k,r=1}^n T_{kj}^\beta(f)F^{\alpha\beta}_{rk}e_r^\alpha$ so that we have $\sum_{i=1}^n -[\Lambda_0, F_{ij}^{\alpha\beta}](f)e_i^\alpha+\sum_{i,k=1}^n F_{ij}^{\alpha\beta}T_{ki}^\alpha(f) e_k^\alpha=\sum_{k,r=1}^n T_{kj}^\beta(f)F^{\alpha\beta}_{rk}e_r^\alpha$. Hence, for any pair $(i,j)$, we have 
\begin{align}\label{ve9}
-[\Lambda_0,F^{\alpha\beta}_{ij}]+\sum_{k=1}^n T_{ik}^\alpha F_{kj}^{\alpha\beta} =\sum_{k=1}^n  F^{\alpha\beta}_{ik}T_{kj}^\beta
\end{align}
\end{enumerate}
\end{remark}

\begin{remark}\label{ve5} 
Let $(X,\Lambda_0)$ be a nonsingular Poisson variety and $(F,\nabla)$ be a Poisson vector bundle with the associated flat Poisson connection $\{-,-\}_F:\mathcal{O}_X\otimes F\to F$. Then $\mathscr{H}om(F,F)$ is a Poisson vector bundle via $\{f,\phi\}_{\mathscr{H}om(F,F)}(x)=\{f,\phi(x)\}_{F}-\phi(\{f,x\}_F)$, where $f\in \mathcal{O}_X,x\in F,\phi\in \mathscr{H}om(F,F)$ $($see \cite{Pol97}$)$. Let us denote the induced Poisson $\mathcal{O}_X$-module structure on $\mathscr{H}om(\mathcal{F},\mathcal{F})$ by $\nabla_{\mathscr{H}om(F,F)}$. Then $\nabla_{\mathscr{H}om(F,F)}:\mathscr{H}om(F,F)\to T_X\otimes\mathscr{H}om(F,F)$ is defined by 
$\nabla_{\mathscr{H}om(F,F)}(\phi)(f)=\{f,\phi\}_{\mathscr{H}om(F,F)}$ and we have a complex of sheaves $($see Remark $\ref{lco})$
\begin{align*}
\mathscr{H}om(F,F)^\bullet:\mathscr{H}om(F,F)\xrightarrow{v_0:=\nabla_{\mathscr{H}om(F,F)}}T_X\otimes \mathscr{H}om(F,F)\xrightarrow{v_1}\wedge^2 T_X\otimes\mathscr{H}om(F,F)\xrightarrow{v_2} \cdots
\end{align*}
We will denote the $i$-th hypercohomology group by $\mathbb{H}^i(X,\Lambda_0, \mathscr{H}om(F,F)^\bullet,\nabla_{\mathscr{H}om(F,F)})$. Let $\mathcal{U}=\{U_\alpha\}$ be an open cover of $X$  such that $F|_{U_\alpha}\cong\oplus_{i=1}^n \mathcal{O}|_{U_\alpha}$ as in Remark $\ref{ve1}$. We keep the notations in Remark $\ref{ve1}$. Then $\mathscr{H}om(F,F)|_{U_\alpha}$ can be identified with a sheaf of $\mathcal{O}_X|_{U_\alpha}$-valued matrices and $T_X\otimes\mathscr{H}om(F,F)|_{U_\alpha}$ can be identified with a sheaf of $T_X|_{U_\alpha}$-valued matrices. Let $E_{ij}^\alpha$ be the matrix whose $(i,j)$ component is $1$ and other components are $0$. Then $E_{ij}^\alpha$ can be considered as a linear transformation from $\oplus_{i=1}^n\mathcal{O}_X|_{U_\alpha}$ to $\oplus_{i=1}^n \mathcal{O}_X|_{U_\alpha}$ which defines an element of $\Gamma(U_\alpha,\mathscr{H}om(F,F))$. Now we compute $\nabla_{\mathscr{H}om(F,F)}(E_{ij}^\alpha)$ explicitly. Set $\nabla_{\mathscr{H}om(F,F)}(E_{ij}^\alpha)=\sum_{p,q=1}^n c^\alpha_{pq}\cdot E_{pq}^\alpha$, where $c^\alpha_{pq}\in \Gamma(U_\alpha,T_X)$. We explicitly find $c^\alpha_{pq}$. For any local section $f\in \mathcal{O}_X$, we have $\nabla_{\mathscr{H}om(F,F)}(E_{ij}^\alpha)(f)=\{f,E_{ij}^\alpha\}_{\mathscr{H}om(F,F)}=\sum_{p,q=1}^n c^\alpha_{pq}(f)\cdot E^\alpha_{pq}$. Then by applying $e_k^\alpha$, we have $\{f,E_{ij}^\alpha(e_k^\alpha)\}_F-E_{ij}^\alpha(\{f,e_k^\alpha\}_F)=\{f,E^\alpha_{ij}(e_k^\alpha)\}_F-E_{ij}^\alpha(\sum_{r=1}^n T_{rk}^\alpha(f)e_r^\alpha)=\{f,E_{ij}^\alpha(e_k^\alpha)\}_F-\sum_{r=1}^n T_{rk}^\alpha(f)E^\alpha_{ij}(e_r^\alpha)=\{f,E_{ij}^\alpha(e_k^\alpha)\}_F-T_{jk}^\alpha(f)e_i^\alpha$. On the other hand, $\sum_{p,q=1}^n c_{pq}(f)E_{pq}^\alpha(e_k^\alpha)=\sum_{p=1}^n c_{pk}^\alpha(f)e_p^\alpha$. Then we have
\begin{align*}
\{f,E_{ij}^\alpha(e_k^\alpha)\}_F-T_{jk}^\alpha(f)e_i^\alpha=\sum_{p=1}^n c_{pk}^\alpha(f)e_p^\alpha
\end{align*}
\begin{enumerate}
\item When $k\ne j$, $-T_{jk}^\alpha(f)e_i^\alpha=\sum_{p=1}^n c_{pk}^\alpha(f)e_p^\alpha$ so that $-T_{jk}^\alpha=c_{ik}^\alpha$, and $c_{pk}^\alpha=0$ for $p\ne i$.
\item When $k=j$, $\{f,e_i^\alpha\}_{F}-T_{jj}^\alpha(f)e_i^\alpha=\sum_{r=1}^n T_{ri}^\alpha(f) e_r^\alpha-T_{jj}^\alpha(f)e_i^\alpha=\sum_{p=1}^n c^\alpha_{pj}(f)e_p^\alpha$ so that $T_{ii}^\alpha-T_{jj}^\alpha=c_{ij}^\alpha$, and $ T^\alpha_{pi}=c_{pj}^\alpha$ for $p\ne i$.
\end{enumerate}
Then $\nabla_{\mathscr{H}om(F,F)}(E_{ij}^\alpha)=(T_{ii}^\alpha -T_{jj}^\alpha) E_{ij}^\alpha+\sum_{k\ne j} -T_{jk}^\alpha E_{ik}^\alpha+\sum_{p\ne i} T_{pi}^\alpha E_{pj}^\alpha$ so that for any pair $(i,j)$,
\begin{align}\label{ve7}
\nabla_{\mathscr{H}om(F,F)}(E_{ij}^\alpha)=v_0(E_{ij}^\alpha)=\sum_{p=1}^n T_{pi}^\alpha E_{pj}^\alpha-\sum_{k=1}^n T_{jk}^\alpha E_{ik}^\alpha
\end{align}

\end{remark}

\begin{definition}\label{ve4}
Let $(X,\Lambda_0)$ be a nonsingular Poisson variety and $(F,\nabla)$ be a Poisson vector bundle on $(X,\Lambda_0)$. Let $A$ be in $\bold{Art}$. An infinitesimal deformation of $(F,\nabla)$ over $Spec(A)$ is a Poisson vector bundle $(\mathcal{F},\nabla_{\mathcal{F}})$ over $A$ on $(X\times_{Spec(k)} Spec(A),\Lambda_0)$ such that $(F,\nabla)=(\mathcal{F},\nabla_{\mathcal{F}})|_{(X,\Lambda_0)}$. Two deformations $(\mathcal{F},\nabla_{\mathcal{F}})$ and $(\mathcal{F}',\nabla_{\mathcal{F}'})$ of $(F,\nabla)$ over $A$ are isomorphic if there is an isomorphism $(\mathcal{F},\nabla_\mathcal{F})\cong (\mathcal{F}',\nabla_{\mathcal{F}'})$ of Poisson $\mathcal{O}_{X\times_{Spec(k)} Spec(A)}$-modules. Then we can define a functor of Artin rings
\begin{align*}
Def_{(F,\nabla)}:\bold{Art}&\to (sets)\\
                                  A&\mapsto \{\text{infinitesimal deformations of $(F,\nabla)$ over $A$}\}/isomorphism
\end{align*}
\end{definition}

\begin{proposition}\label{ve2}
Let $(X,\Lambda_0)$ be a nonsingular Poisson variety and $(F,\nabla)$ be a Poisson vector bundle of rank $n$ on $(X,\Lambda_0)$. Then
\begin{enumerate}
\item There is a $1-1$ correspondence
\begin{align*}
Def_{(\mathcal{F},\nabla)}(k[\epsilon])\cong \mathbb{H}^1(X,\Lambda_0,\mathscr{H}om(F,F)^\bullet,\nabla_{\mathscr{H}om(F,F)})
\end{align*}
\item Let $A\in \bold{Art}$ and $\eta=(\mathcal{F},\nabla_{\mathcal{F}})$ be an infinitesimal Poisson deformation of $(F,\nabla)$ over $A$. Then, to every small extension $e:0\to (t)\to \tilde{A}\to A\to 0$, we can associate an element $o_\eta(e)\in \mathbb{H}^2(X,\Lambda_0,\mathscr{H}om(F,F)^\bullet, \nabla_{\mathscr{H}om(F,F)})$ called the obstruction lifting of $\eta$ to $\tilde{A}$, which is $0$ if and only if a lifting of $\eta$ exists.
\end{enumerate}
\end{proposition}

\begin{proof}
Let $\mathcal{U}=\{U_\alpha\}$ be an open covering of $X$ such that $F|_{U_\alpha}\cong \oplus_{i=1}^n \mathcal{O}_X|_{U_\alpha} $ as in Remark \ref{ve1}. We keep the notations in Remark \ref{ve1} and Remark \ref{ve5}. Let $(\mathcal{F},\nabla_{\mathcal{F}})$ be a first-order deformation of $(F,\nabla)$ over $k[\epsilon]$. Then $\mathcal{F}|_{U_\alpha}\cong \oplus_{i=1}^n \mathcal{O}_{X\times_{Spec(k)}Spec(k[\epsilon])}|_{U_\alpha}=\oplus_{i=1}^n \mathcal{O}_X|_{U_\alpha}\otimes k[\epsilon]$ with transition matrices $\{F^{\alpha\beta}+\epsilon G^{\alpha\beta}\}$ for some matrix $G^{\alpha\beta}:=(g_{ij}^{\alpha\beta}), g_{ij}^{\alpha\beta}\in \Gamma(U_{\alpha\beta},\mathcal{O}_X)$ and $\nabla_\mathcal{F}(e_j^\alpha)=\sum_{i=1}^n (T_{ij}^\alpha+\epsilon W_{ij}^\alpha) e_i^\alpha$ for some $W_{ij}^\alpha \in \Gamma(U_\alpha, T_X)$. Then $\{G^{\alpha\beta}\}$ can be considered as an element of $C^1(\mathcal{U},\mathscr{H}om(F,F))$. Let $W^\alpha$ be the $\Gamma(U_\alpha, T_X)$-valued matrix whose $(i,j)$ component is $W_{ij}^\alpha$. In other words, $W^\alpha=\sum_{i,j=1}^nW_{ij}^\alpha E_{ij}^\alpha$. Then $\{W^\alpha\}$ can be considered as an element of $C^0(\mathcal{U},T_X\otimes\mathscr{H}om(F,F))$. We claim that $(\{-W^\alpha\},\{G^{\alpha\beta}\})\in C^0(\mathcal{U},T_X\otimes\mathscr{H}om(F,F))\oplus C^1(\mathcal{U},\mathscr{H}om(F,F))$ defines a $1$-cocycle in the following \v{C}ech resolution. 
\begin{center}
$\begin{CD}
\mathcal{C}^0(\mathcal{U},\wedge^3 T_X\otimes\mathscr{H}om(F,F))\\
@Av_2AA\\
\mathcal{C}^0(\mathcal{U},\wedge^2 T_X\otimes\mathscr{H}om(F,F))@>-\delta>>\mathcal{C}^1(\mathcal{U},\wedge^2 T_X\otimes\mathscr{H}om(F,F))\\
@Av_1AA @Av_1AA\\
C^0(\mathcal{U},T_X\otimes\mathscr{H}om(F,F))@>\delta>> C^1(\mathcal{U},T_X\otimes\mathscr{H}om(F,F))@>-\delta>>\mathcal{C}^2(\mathcal{U},T_X\otimes\mathscr{H}om(F,F))\\
@Av_0=\nabla_{\mathscr{H}om(F,F)} AA @Av_0AA @Av_0AA\\
C^0(\mathcal{U},\mathscr{H}om(F,F))@>-\delta>>C^1(\mathcal{U},\mathscr{H}om(F,F))@>\delta>>C^2(\mathcal{U},\mathscr{H}om(F,F))@>-\delta>>
\end{CD}$
\end{center}
Indeed, $v_1(W^\alpha)=v_1(\sum_{i,j=1}^n W_{ij}^\alpha E_{ij}^\alpha)= \sum_{i,j=1}^n [\Lambda_0,W^\alpha_{ij}]\cdot E_{ij}^\alpha-\sum_{i,j=1}^n W_{ij}^\alpha\wedge v_0(E_{ij}^\alpha)=\sum_{i,j=1}^n [\Lambda_0,W^\alpha_{ij}]\cdot E_{ij}^\alpha-\sum_{i,j,k=1}^n W_{ij}^\alpha\wedge T_{ki}^\alpha E_{kj}^\alpha+\sum_{i,j,k=1}^n W_{ij}^\alpha\wedge T_{jk}^\alpha E_{ik}^\alpha$ by (\ref{ve7}). On the other hand, since $[\Lambda_0,T_{ij}^\alpha+\epsilon W_{ij}^\alpha]+\sum_{k=1}^n (T_{ik}^\alpha+\epsilon W_{ik}^\alpha)\wedge (T_{kj}^\alpha+\epsilon W_{kj}^\alpha)=0$ by (\ref{ve8}), we have 
\begin{align*}
[\Lambda_0,W_{ij}^\alpha]+\sum_{k=1}^n T_{ik}^\alpha\wedge W_{kj}^\alpha+\sum_{k=1}^n W^\alpha_{ik}\wedge T_{kj}^\alpha=0
\end{align*}
Hence $v_1(W^\alpha)=0$ so that $v_1(\{-W^\alpha\})=0$.

Next we will show $\delta(\{-W^\alpha\})+v_0(\{G^{\alpha\beta}\})=0$. Let us consider $F^{\beta\alpha}=(F_{ij}^{\beta\alpha})$ so that $F^{\beta\alpha} F^{\alpha\beta}=I_n$. First $\delta(\{-W^\alpha\})=\{-W^\beta+ F^{\beta\alpha}W^\alpha F^{\alpha\beta}\}$ since $W^{\alpha}:\oplus_{i=1}^n \mathcal{O}_X|_{U_\alpha}\to \oplus_{i=1}^n T_X|_{U_\alpha}$ becomes $F^{\beta\alpha}W^{\alpha}F^{\alpha\beta}$ in terms of $\oplus_{i=1}^n\mathcal{O}_X|_{U_\beta}\to \oplus_{i=1}^n T_X|_{U_\beta}$. Second let us compute $v_0(\{G^{\alpha\beta}\})$. We note that $G^{\alpha\beta}:\oplus_{i=1}^n\mathcal{O}_X|_{U_\beta}\to \oplus_{i=1}^n \mathcal{O}_X|_{U_\alpha}$ becomes $F^{\beta\alpha}G^{\alpha\beta}$ In terms of $\oplus_{i=1}^n \mathcal{O}_X|_{U_\beta} \to \oplus_{i=1}^n \mathcal{O}_X|_{U_\beta}$. Then from (\ref{ve7}),
\begin{align*}
&v_0(F^{\beta\alpha}G^{\alpha\beta})=v_0(\sum_{i,j,k=1}^n F_{ik}^{\beta\alpha}g^{\alpha\beta}_{kj} E_{ij}^\beta)=\sum_{i,j,k=1}^n -[\Lambda_0, F^{\beta\alpha}_{ik}g^{\alpha\beta}_{kj}]E_{ij}^\beta+\sum_{i,j,k=1}^n F^{\beta\alpha}_{ik}g^{\alpha\beta}_{kj}v_0(E_{ij}^\beta)\\
&=\sum_{i,j,k=1}^n -F^{\beta\alpha}_{ik}[\Lambda_0,g_{kj}^{\alpha\beta}]E_{ij}^\beta+\sum_{i,j,k=1}^n -g_{kj}^{\alpha\beta}[\Lambda_0, F_{ik}^{\beta\alpha}]E_{ij}^\beta+\sum_{i,j,k,p=1}^n F_{ik}^{\beta\alpha}g^{\alpha\beta}_{kj}T_{pi}^\beta E_{pj}^\beta-\sum_{i,j,k,p=1}^nF_{ik}^{\beta\alpha}g^{\alpha\beta}_{kj}T_{jp}^\beta E_{ip}^\beta.
\end{align*}
Then the coefficient of $E_{ij}^\beta$ is as follows (here we use $-[\Lambda_0, F_{ik}^{\beta\alpha}]+\sum_{p=1}^n  T_{ip}^\beta F_{pk}^{\beta\alpha}=\sum_{p=1}^n F_{ip}^{\beta\alpha}T_{pk}^\alpha$ from (\ref{ve9})).
\begin{align*}
&-\sum_{k=1}^n  F_{ik}^{\beta\alpha}[\Lambda_0, g^{\alpha\beta}_{kj}]-\sum_{k=1}^n g^{\alpha\beta}_{kj}[\Lambda_0, F_{ik}^{\beta\alpha}]+\sum_{p,k=1}^n F^{\beta\alpha}_{pk}g^{\alpha\beta}_{kj}T_{ip}^\beta-\sum_{p,k=1}^n F^{\beta\alpha}_{ik}g^{\alpha\beta}_{kp}T_{pj}^\beta\\
&=-\sum_{k=1}^n  F^{\beta\alpha}_{ik}[\Lambda_0, g^{\alpha\beta}_{kj}]+\sum_{p,k=1}^n g^{\alpha\beta}_{kj} F^{\beta\alpha}_{ip}T_{pk}^\alpha-\sum_{p,k=1}^n g^{\alpha\beta}_{kj} T_{ip}^\beta F^{\beta\alpha}_{pk}+\sum_{p,k=1}^n F^{\beta\alpha}_{pk}g^{\alpha\beta}_{kj}T_{ip}^\beta -\sum_{p,k=1}^n F^{\beta\alpha}_{ik}g^{\alpha\beta}_{kp}T_{pj}^\beta \\
&=-\sum_{k=1}^n  F^{\beta\alpha}_{ik}[\Lambda_0, g^{\alpha\beta}_{kj}]+\sum_{p,k=1}^n  F^{\beta\alpha}_{ip} T_{pk}^\alpha  g_{kj}^{\alpha\beta}-\sum_{p,k=1}^n F^{\beta\alpha}_{ik}g^{\alpha\beta}_{kp}T_{pj}^\beta
\end{align*}

Hence 
\begin{align}\label{ve11}
&\delta(\{-W^\alpha\})+v_0(\{G^{\alpha\beta}\})\\
&=\{(-W_{ij}^\beta+\sum_{p,k=1}^n F^{\beta\alpha}_{ik}W_{kp}^\alpha F^{\alpha\beta}_{pj} -\sum_{k=1}^n F_{ik}^{\beta\alpha}[\Lambda_0, g_{kj}^{\alpha\beta}]+\sum_{p,k=1}^n  F^{\beta\alpha}_{ip} T_{pk}^\alpha  g_{kj}^{\alpha\beta}-\sum_{p,k=1}^n F^{\beta\alpha}_{ik}g^{\alpha\beta}_{kp}T_{pj}^\beta) \}\notag
\end{align}
On the other hand, from (\ref{ve9}), we have $-[\Lambda_0,F^{\alpha\beta}_{kj}+\epsilon g_{kj}^{\alpha\beta}]+\sum_{p=1}^n (T_{kp}^\alpha+\epsilon W_{kp}^\alpha)( F_{pj}^{\alpha\beta}+\epsilon g_{pj}^{\alpha\beta}) =\sum_{p=1}^n  (F^{\alpha\beta}_{kp}+\epsilon g^{\alpha\beta}_{kp})(T_{pj}^\beta+\epsilon W_{pj}^\beta)$ so that we have
\begin{align*}
-[\Lambda_0, g_{kj}^{\alpha\beta}]+\sum_{p=1}^n T_{kp}^\alpha g_{pj}^{\alpha\beta}+\sum_{p=1}^n W_{kp}^\alpha F_{pj}^{\alpha\beta}=\sum_{p=1}^n F_{kp}^{\alpha\beta} W_{pj}^\beta+\sum_{p=1}^n g_{kp}^{\alpha\beta} T_{pj}^\beta.
\end{align*}

By multiplying $\sum_{k=1}^n F_{ik}^{\beta\alpha}$, we get
\begin{align*}
-\sum_{k=1}^n F^{\beta\alpha}_{ik}[\Lambda_0, g^{\alpha\beta}_{kj}]+\sum_{k,p=1}^n F^{\beta\alpha}_{ik}T_{kp}^\alpha g^{\alpha\beta}_{pj}+\sum_{k,p=1}^n F^{\beta\alpha}_{ik}W_{kp}^\alpha F^{\alpha\beta}_{pj}-W_{ij}^\beta-\sum_{k,p=1}^n F^{\beta\alpha}_{ik} g^{\alpha\beta}_{kp} T_{pj}^\beta=0
\end{align*}

From (\ref{ve11}), we have $\delta(\{-W^\alpha\})+v_0(\{G^{\alpha\beta}\})=0$.

Lastly, $\delta(\{G^{\alpha\beta}\})=\{F^{\alpha\beta} G^{\beta\gamma}-G^{\alpha\gamma}+G^{\alpha\beta}F^{\beta\gamma}\}$. Since $(F^{\alpha\beta}+\epsilon G^{\alpha\beta})(F^{\beta\gamma}+\epsilon G^{\beta\gamma})=F^{\alpha\gamma}+\epsilon G^{\alpha\gamma}$, we have $F^{\alpha\beta}G^{\beta\gamma}+G^{\alpha\beta} F^{\beta\gamma}= G^{\alpha\gamma}$ so that $\delta(\{G^{\alpha\beta}\})=0$. Hence $(\{-W^{\alpha}\},\{G^{\alpha\beta}\})\in C^0(\mathcal{U},T_X\otimes \mathscr{H}om(F,F))\oplus C^1(\mathcal{U},\mathscr{H}om(F,F))$ defines a $1$-cocycle in the above \v{C}ech resolution.

Assume that we have two equivalent first-order deformations $(\mathcal{F},\nabla_\mathcal{F})$ and $(\mathcal{F}',\nabla_{\mathcal{F}'})$ of $(F,\nabla)$ over $Spec(k[\epsilon])$, which are represented by $(\{F^{\alpha\beta}+\epsilon G^{\alpha\beta}=(F_{ij}^{\alpha\beta}+\epsilon g_{ij}^{\alpha\beta})\},\{T^\alpha+\epsilon W^\alpha=(T_{ij}^{\alpha}+\epsilon W_{ij}^\alpha)\})$ and $(\{F^{\alpha\beta}+\epsilon G'^{\alpha\beta}=(F_{ij}^{\alpha\beta}+\epsilon {g'}_{ij}^{\alpha\beta})\},\{T^\alpha+\epsilon W'^\alpha=( T_i^\alpha+\epsilon W'^\alpha_i)\})$ respectively so that there exist matrixces $\{A^\alpha:=(a_{ij}^\alpha)| a_{ij}^\alpha\in \Gamma(U_\alpha, \mathcal{O}_X)\}$ which define an element of $C^0(\mathcal{U},\mathscr{H}om(F,F))$ such that $(I+\epsilon A^\alpha)(F^{\alpha\beta}+\epsilon G^{\alpha\beta})=(F^{\alpha\beta}+\epsilon G'^{\alpha\beta})(I+\epsilon A^\beta)$. Then $A^\alpha F^{\alpha\beta}+G^{\alpha\beta}=G'^{\alpha\beta}+F^{\alpha\beta}A^\beta$ so that $-\delta(\{A^\alpha\})=(\{G'^{\alpha\beta}-G^{\alpha\beta}\})\in C^1(\mathcal{U},\mathscr{H}om(F,F))$. On the other hand, since $Id+\epsilon A^\alpha$ define a morphism of Poisson $\mathcal{O}_X$-modules from $(\mathcal{F},\nabla_\mathcal{F})$ to $(\mathcal{F}',\nabla_{\mathcal{F}'})$, we have $-[\Lambda_0,\delta_{ij}+\epsilon a_{ij}^{\alpha}]+\sum_{k=1}^n (T_{ik}^\alpha+\epsilon {W'}_{ik}^\alpha)(\delta_{kj}+\epsilon a_{kj}^\alpha)=\sum_{k=1}^n (\delta_{ik}+\epsilon a_{ik}^\alpha)(T_{kj}^\alpha+\epsilon W_{kj}^\alpha)$. Then  $-[\Lambda_0, a_{ij}^\alpha]+\sum_{k=1}^n T_{ik}^\alpha a_{kj}^\alpha +{W'}_{ij}^\alpha=\sum_{k=1}^n a_{ik}^\alpha T_{kj}^\alpha+W_{ij}^\alpha$ so that $v_0(\{A^\alpha\})=(\{W^\alpha-W'^\alpha\})$. Hence $a=(\{-W^\alpha\}, \{G^{\alpha\beta}\})$ is cohomologous to $b=(\{-W'^{\alpha}\},\{G'^{\alpha\beta}\})$. This proves $(1)$ in Proposition \ref{ve2}. 

Now we identify obstructions. Let us consider a small extension $e:0\to (t)\to \tilde{A}\to A\to 0$ in $\bold{Art}$ and let $\eta=(\mathcal{F},\nabla_\mathcal{F})$ be an infinitesimal deformation of $(F,\nabla)$ over $Spec(A)$. Let $\mathcal{U}=\{U_\alpha\}$ be an open covering  of $X$ such that $(\mathcal{F},\nabla_\mathcal{F})$ be represented by transition matrices $\{(H_{ij}^{\alpha\beta})\}$ with $H_{ij}^{\alpha\beta}\in \Gamma(U_{ij},\mathcal{O}_{X\times Spec(A)}) $ and $\sum_{j=1}^n H_{ij}^{\alpha\beta}H_{jk}^{\beta\gamma}=H_{ik}^{\alpha\gamma}$, and $\{(Y_{ij}^\alpha)\}$ where $Y_{ij}^\alpha\in \Gamma(U_\alpha, T_X)\otimes A$, $[\Lambda_0,Y_{ij}^\alpha]+\sum_{k=1}^n Y_{ik}^\alpha\wedge Y_{kj}^\alpha=0$, and $-[\Lambda_0,H_{ij}^{\alpha\beta}]+\sum_{k=1}^n Y_{ik}^\alpha H_{kj}^{\alpha\beta}=\sum_{k=1}^n H_{ik}^{\alpha\beta}Y_{kj}^\beta$. In order to see if a lifting $\tilde{\eta}$ of $\eta$ to $Spec(\tilde{A})$ exists, we choose an arbitrary collection $\{\tilde{F}_{ij}^{\alpha\beta}\}$ and $\{\tilde{Y}_{ij}^\alpha\}$ where $\tilde{F}_{ij}^{\alpha\beta}\in \Gamma(U_{\alpha\beta},\mathcal{O}_{X\times Spec(\tilde{A})})$ restricts to $H_{ij}^{\alpha\beta}$, and $\tilde{Y}_{ij}^{\alpha}\in \Gamma(U_{\alpha},T_X)\otimes \tilde{A}$ restricts to $Y_{ij}^{\alpha}$.  

Let $[\Lambda_0, \tilde{Y}_{ij}^\alpha]+\sum_{k=1}^n \tilde{Y}_{ik}^\alpha\wedge \tilde{Y}_{kj}^\alpha=tS_{ij}^\alpha$, where $S_{ij}^\alpha\in \Gamma(U_\alpha, \wedge^2 T_X)$. Then the $\Gamma(U_\alpha, \wedge^2 T_X)$-valued matrices $S^\alpha:=(S_{ij}^\alpha)$ define an element in $C^0(\mathcal{U},\wedge^2 T_X\otimes \mathscr{H}om(F,F))$ in terms of $\oplus_{i=1}^n \mathcal{O}_X|_{U_\alpha}\to \oplus_{i=1}^n \wedge^2 T_X|_{U_\alpha}$. Let $-[\Lambda_0, \tilde{F}_{ij}^{\alpha\beta}]+\sum_{k=1}^n \tilde{Y}_{ik}^\alpha \tilde{F}_{kj}^{\alpha\beta}-\sum_{k=1}^n \tilde{F}_{ik}^{\alpha\beta} \tilde{Y}_{kj}^\beta=tQ_{ij}^{\alpha\beta}$ where $Q_{ij}^{\alpha\beta}\in \Gamma(U_{\alpha\beta},T_X)$. Then the $\Gamma(U_{\alpha\beta},T_X)$-valued matrices $Q^{\alpha\beta}:=(Q_{ij}^{\alpha\beta})$ define an element in $C^1(\mathcal{U},T_X\otimes\mathscr{H}om(F,F))$ in terms of $\oplus_{i=1}^n \mathcal{O}_X|_{U_\beta}\to \oplus_{i=1}^n T_X|_{U_\alpha}$. Let $\sum_{k,p=1}^n\tilde{F}_{ik}^{\alpha\beta}\tilde{F}_{kp}^{\beta\gamma}\tilde{F}_{pj}^{\gamma\alpha}=\delta_{ij}+tg_{ij}^{\alpha\beta\gamma}$ where $g_{ij}^{\alpha\beta\gamma}\in \Gamma(U_{\alpha\beta\gamma},\mathcal{O}_X)$. Then the $\Gamma(U_{\alpha\beta\gamma},\mathcal{O}_X)$-valued matrices $G^{\alpha\beta\gamma}:=(g_{ijk}^{\alpha\beta\gamma})$ define an element in $C^2(\mathcal{U},\mathscr{H}om(F,F))$ in terms of $\oplus_{i=1}^n \mathcal{O}_X|_{U_\alpha}\to \oplus_{i=1}^n \mathcal{O}_X|_{U_\alpha}$. Putting $\tilde{F}^{\alpha\beta}:=(\tilde{F}_{ij}^{\alpha\beta})$, we have $\tilde{F}^{\beta\alpha}\tilde{F}^{\alpha\beta}=I_n,\tilde{F}^{\alpha\beta}\tilde{F}^{\beta\gamma}\tilde{F}^{\gamma\alpha}=I_n+t G^{\alpha\beta\gamma}$ and $\tilde{F}^{\alpha\gamma}\tilde{F}^{\gamma\beta}\tilde{F}^{\beta\alpha}=I_n-tG^{\alpha\beta\gamma}$.

We claim that $(\{-S^\alpha\},\{Q^{\alpha\beta}\},\{G^{\alpha\beta\gamma}\})\in C^0(\mathcal{U},\wedge^2 T_X\otimes \mathscr{H}om(F,F))\oplus C^1(\mathcal{U},T_X\otimes\mathscr{H}om(F,F))\oplus C^2(\mathcal{U},\mathscr{H}om(F,F))$ defines an $2$-cocycle in the above \v{C}ech resolution.

First, we show that $v_2(\{S^{\alpha}\})=0$, equivalently $v_2(tS^\alpha)=0$. Indeed, from (\ref{ve7}), 
\begin{align*}
&v_2(tS^\alpha)=v_2(\sum_{i,j=1}^n tS_{ij}^\alpha E_{ij}^\alpha)=\sum_{i,j=1}^n -[\Lambda_0, tS_{ij}^\alpha]E_{ij}^\alpha+\sum_{i,j=1}^ntS_{ij}^\alpha\wedge v_0(E_{ij}^\alpha)\\
&=\sum_{i,j,k=1}^n -[\Lambda_0,\tilde{Y}_{ik}^\alpha\wedge \tilde{Y}_{kj}^\alpha]E_{ij}^\alpha+\sum_{i,j,p=1}^n tS_{ij}^\alpha\wedge T_{pi}^\alpha E_{pj}^\alpha-\sum_{i,j,p=1}^n tS_{ij}^\alpha \wedge T_{jp}^\alpha E_{ip}^\alpha\\
&=\sum_{i,j,k=1}^n -\tilde{Y}_{ik}^\alpha \wedge [\Lambda_0, \tilde{Y}_{kj}^\alpha] E_{ij}^\alpha+\sum_{i,j,k=1}^n[\Lambda_0,\tilde{Y}_{ik}^\alpha]\wedge \tilde{Y}_{kj}^\alpha E_{ij}^\alpha+\sum_{i,j,p=1}^n tS_{pj}^\alpha\wedge T_{ip}^\alpha E_{ij}^\alpha-\sum_{i,j,p=1}^n tS_{ip}^\alpha\wedge T_{pj}^\alpha E_{ij}^\alpha\\
&=\sum_{i,j,k=1}^n -\tilde{Y}_{ik}^\alpha\wedge (tS_{kj}^\alpha-\sum_{p=1}^n \tilde{Y}_{kp}^\alpha\wedge \tilde{Y}_{pj}^\alpha)E_{ij}^\alpha+\sum_{i,j,k=1}^n (tS_{ik}^\alpha-\sum_{p=1}^n \tilde{Y}_{ip}^\alpha\wedge \tilde{Y}_{pk}^\alpha)\wedge\tilde{Y}_{kj}^\alpha E_{ij}^\alpha\\
&\,\,\,\,\,\,\,\,\,\,\,\,\,\,\,\,\,\,+\sum_{i,j,k=1}^n T_{ik}^\alpha \wedge tS_{kj}^\alpha  E_{ij}^\alpha-\sum_{i,j,k=1}^n tS_{ik}^\alpha\wedge T_{kj}^\alpha E_{ij}^\alpha=0\\
&=\sum_{i,j,k,=1}^n -\tilde{Y}_{ik}^\alpha\wedge tS_{kj}^\alpha E_{ij}^\alpha+\sum_{i,j,k=1}^n tS^\alpha_{ik}\wedge \tilde{Y}_{kj}^\alpha E_{ij}^\alpha+\sum_{i,j,k=1}^n T_{ik}^\alpha \wedge tS_{kj}^\alpha  E_{ij}^\alpha-\sum_{i,j,k=1}^n tS_{ik}^\alpha\wedge T_{kj}^\alpha E_{ij}^\alpha=0.
\end{align*}

Second, we show that $-\delta(\{-S^\alpha\})+v_1(\{Q^{\alpha\beta}\})=0$. We compute  $-\delta(\{S^\alpha\})=\{F^{\beta\alpha}S^\alpha F^{\alpha\beta}-S^\beta \}$ in terms of $\oplus_{i=1}^n \mathcal{O}_X|_{U_\beta}\to \oplus_{i=1}^n \wedge^2 T_X|_{U_\beta}$.
For this, we compute $-\delta(\{tS^\alpha\})=\{tS^{\alpha}\tilde{F}^{\alpha\beta}-\tilde{F}^{\alpha\beta}tS^\beta\}$ in terms of $\oplus_{i=1}^n \mathcal{O}_X|_{U_\beta}\to \oplus_{i=1}^n \wedge^2 T_X|_{U_\alpha}$ in the following (here we use that $[\Lambda_0, tQ_{rj}^{\alpha\beta}]=\sum_{k=1}^n [\Lambda_0, \tilde{Y}_{rk}^{\alpha}\tilde{F}_{kj}^{\alpha\beta}]-\sum_{k=1}^n [\Lambda_0, \tilde{F}_{rk}^{\alpha\beta}\tilde{Y}_{kj}^{\beta}]$).

\begin{align*}
&\sum_{k=1}^n tS_{rk}^\alpha \tilde{F}_{kj}^{\alpha\beta}-\sum_{k=1}^n \tilde{F}_{rk}^{\alpha\beta} tS_{kj}^\beta\\
&=\sum_{k=1}^n [\Lambda_0,\tilde{Y}_{rk}^\alpha]\tilde{F}_{kj}^{\alpha\beta}+\sum_{k,p=1}^n \tilde{Y}_{rp}^\alpha\wedge(\tilde{Y}_{pk}^\alpha\tilde{F}_{kj}^{\alpha\beta})-\sum_{k,p=1}^n\tilde{F}_{rk}^{\alpha\beta}[\Lambda_0,\tilde{Y}_{kj}^\beta]-\sum_{k,p=1}^n(\tilde{F}_{rk}^{\alpha\beta}\tilde{Y}_{kp}^\beta)\wedge \tilde{Y}_{pj}^\beta\\
&=\sum_{k=1}^n [\Lambda_0,\tilde{Y}_{rk}^\alpha]\tilde{F}_{kj}^{\alpha\beta}+\sum_{p=1}^n\tilde{Y}_{rp}^\alpha\wedge[\Lambda_0, \tilde{F}_{pj}^{\alpha\beta}]+\sum_{p,k=1}^n \tilde{Y}_{rp}^\alpha\wedge \tilde{F}_{pk}^{\alpha\beta}\tilde{Y}_{kj}^{\beta}+t\sum_{p=1}^n T_{rp}^\alpha\wedge Q_{pj}^{\alpha\beta}\\
&-\sum_{k,p=1}^n\tilde{F}_{rk}^{\alpha\beta}[\Lambda_0,\tilde{Y}_{kj}^\beta]+\sum_{p=1}^n[\Lambda_0,\tilde{F}_{rp}^{\alpha\beta}]\wedge \tilde{Y}_{pj}^\beta-\sum_{k,p=1}^n(\tilde{Y}_{rk}^\alpha\tilde{F}_{kp}^{\alpha\beta})\wedge \tilde{Y}_{pj}^\beta+t\sum_{p=1}^n Q_{rp}^{\alpha\beta}\wedge T_{pj}^\beta\\
&=\sum_{k=1}^n [\Lambda_0,\tilde{Y}_{rk}^{\alpha}]\tilde{F}_{kj}^{\alpha\beta}+\sum_{p=1}^n\tilde{Y}_{rp}^\alpha\wedge[\Lambda_0, \tilde{F}_{pj}^{\alpha\beta}]+t\sum_{p=1}^n T_{rp}^\alpha\wedge Q_{pj}^{\alpha\beta}-\sum_{k,p=1}^n\tilde{F}_{rk}^{\alpha\beta}[\Lambda_0,\tilde{Y}_{kj}^\beta]+\sum_{p=1}^n[\Lambda_0,\tilde{F}_{rp}^{\alpha\beta}]\wedge \tilde{Y}_{pj}^\beta\\&+t\sum_{p=1}^n Q_{rp}^{\alpha\beta}\wedge T_{pj}^\beta=[\Lambda_0, tQ_{rj}^{\alpha\beta}]+t\sum_{p=1}^n T_{rp}^\alpha\wedge Q_{pj}^{\alpha\beta}+t\sum_{p=1}^n Q_{rp}^{\alpha\beta}\wedge T_{pj}^\beta
\end{align*}
Hence we have $\sum_{k=1}^n S_{rk}^\alpha F_{kj}^{\alpha\beta}-\sum_{k=1}^n F_{rk}^{\alpha\beta} S_{kj}^\beta=[\Lambda_0, Q_{rj}^{\alpha\beta}]+\sum_{p=1}^n T_{rp}^\alpha\wedge Q_{pj}^{\alpha\beta}+\sum_{p=1}^n Q_{rp}^{\alpha\beta}\wedge T_{pj}^\beta.$ By multiplying $\sum_{r=1}^n F_{ir}^{\beta\alpha}$, we have
\begin{align}\label{ve13}
\sum_{r,k=1}^n F^{\beta\alpha}_{ir}S_{rk}^\alpha F_{kj}^{\alpha\beta}- S_{ij}^\beta=\sum_{r=1}^n F_{ir}^{\beta\alpha}[\Lambda_0, Q_{rj}^{\alpha\beta}]+\sum_{r,p=1}^n F_{ir}^{\beta\alpha}T_{rp}^\alpha\wedge Q_{pj}^{\alpha\beta}+\sum_{r,p=1}^n F_{ir}^{\beta\alpha}Q_{rp}^{\alpha\beta}\wedge T_{pj}^\beta
\end{align}

On the other hand, $Q^{\alpha\beta}$ becomes $F^{\beta\alpha}Q^{\alpha\beta}$ in terms of $\oplus_{i=1}^n\mathcal{O}_X|_{U_\beta}\to \oplus_{i=1}^n T_X|_{U_\beta}$ so that from (\ref{ve7}), we have 
\begin{align*}
&v_1(F^{\beta\alpha}Q^{\alpha\beta})=v_1(\sum_{i,j,k=1}^n F_{ik}^{\beta\alpha}Q^{\alpha\beta}_{kj} E_{ij}^\beta)\\
&=\sum_{i,j,k=1}^n [\Lambda_0, F^{\beta\alpha}_{ik}Q^{\alpha\beta}_{kj}]E_{ij}^\beta+(-1)^1\sum_{i,j,k=1}^n F^{\beta\alpha}_{ik}Q^{\alpha\beta}_{kj}\wedge v_0(E_{ij}^\beta)\\
&=\sum_{i,j,k=1}^n F^{\beta\alpha}_{ik}[\Lambda_0,Q_{kj}^{\alpha\beta}]E_{ij}^\beta+\sum_{i,j,k=1}^n Q_{kj}^{\alpha\beta}\wedge [\Lambda_0, F_{ik}^{\beta\alpha}]E_{ij}^\beta-\sum_{i,j,k,p=1}^n (F_{ik}^{\beta\alpha}Q^{\alpha\beta}_{kj}\wedge T_{pi}^\beta) E_{pj}^\beta+\sum_{i,j,k,p=1}^n(F_{ik}^{\beta\alpha}Q^{\alpha\beta}_{kj}\wedge T_{jp}^\beta) E_{ip}^\beta.
\end{align*}
Then the coefficient of $E_{ij}^\beta$ is as follows (here we use $-[\Lambda_0, F_{ik}^{\beta\alpha}]+\sum_{p=1}^n  T_{ip}^\beta F_{pk}^{\beta\alpha}=\sum_{p=1}^n F_{ip}^{\beta\alpha}T_{pk}^\alpha$ from (\ref{ve9})).
\begin{align}\label{ve12}
&\sum_{k=1}^n  F_{ik}^{\beta\alpha}[\Lambda_0, Q^{\alpha\beta}_{kj}]+\sum_{k=1}^n Q^{\alpha\beta}_{kj}\wedge [\Lambda_0, F_{ik}^{\beta\alpha}]-\sum_{p,k=1}^n F^{\beta\alpha}_{pk}Q^{\alpha\beta}_{kj}\wedge T_{ip}^\beta+\sum_{p,k=1}^n F^{\beta\alpha}_{ik}Q^{\alpha\beta}_{kp}\wedge T_{pj}^\beta\\\notag
&=\sum_{k=1}^n  F^{\beta\alpha}_{ik}[\Lambda_0, Q^{\alpha\beta}_{kj}]-\sum_{p,k=1}^n Q^{\alpha\beta}_{kj} \wedge F^{\beta\alpha}_{ip}T_{pk}^\alpha+\sum_{p,k=1}^n Q^{\alpha\beta}_{kj} \wedge T_{ip}^\beta F^{\beta\alpha}_{pk}-\sum_{p,k=1}^n F^{\beta\alpha}_{pk}Q^{\alpha\beta}_{kj}\wedge T_{ip}^\beta +\sum_{p,k=1}^n F^{\beta\alpha}_{ik}Q^{\alpha\beta}_{kp}\wedge T_{pj}^\beta \\\notag
&=\sum_{k=1}^n  F^{\beta\alpha}_{ik}[\Lambda_0, Q^{\alpha\beta}_{kj}]+\sum_{p,k=1}^n  F^{\beta\alpha}_{ip} T_{pk}^\alpha \wedge Q_{kj}^{\alpha\beta}+\sum_{p,k=1}^n F^{\beta\alpha}_{ik}Q^{\alpha\beta}_{kp}\wedge T_{pj}^\beta
\end{align}

Hence from (\ref{ve13}) and (\ref{ve12}), we have $-\delta(\{-S^\alpha\})+v_1(\{Q^{\alpha\beta}\})=0$.

Next we show that $v_0(\{G^{\alpha\beta\gamma}\})-\delta(\{Q^{\alpha\beta}\})=0$, equivalently $v_0(tG^{\alpha\beta\gamma})=\delta(tQ^{\alpha\beta})$. Indeed, from (\ref{ve7}),
\begin{align*}
v_0(tG^{\alpha\beta})=v_0(\sum_{i,j=1}^n tg_{ij}^{\alpha\beta\gamma}E_{ij}^\alpha)&=\sum_{i,j=1}^n-[\Lambda_0, tg_{ij}^{\alpha\beta\gamma}]E_{ij}^\alpha+\sum_{i,j=1}^ntg_{ij}^{\alpha\beta\gamma}v_0(E_{ij}^\alpha)\\
&=\sum_{i,j=1}^n-[\Lambda_0, tg_{ij}^{\alpha\beta\gamma}]E_{ij}^\alpha+\sum_{i,j,k=1}^n tg_{ij}^{\alpha\beta\gamma} T_{ki}^\alpha E_{kj}^\alpha-\sum_{i,j,k=1}^n tg_{ij}^{\alpha\beta\gamma}T_{jk}^\alpha E_{ik}^\alpha\\
&=\sum_{i,j=1}^n-[\Lambda_0, tg_{ij}^{\alpha\beta\gamma}]E_{ij}^\alpha+\sum_{i,j,k=1}^n tg_{ij}^{\alpha\beta\gamma} \tilde{Y}_{ki}^\alpha E_{kj}^\alpha-\sum_{i,j,k=1}^n tg_{ij}^{\alpha\beta\gamma}\tilde{Y}_{jk}^\alpha E_{ik}^\alpha
\end{align*}
The coefficient of $E_{ij}^\alpha$ is $-[\Lambda_0, tg_{ij}^{\alpha\beta\gamma}]+\sum_{k=1}^n tg_{kj}^{\alpha\beta\gamma} \tilde{Y}_{ik}^\alpha -\sum_{k=1}^n g_{ik}^{\alpha\beta\gamma}\tilde{Y}_{kj}^\alpha$. Then
\begin{align*}
&-[\Lambda_0, tg_{ij}^{\alpha\beta\gamma}]+\sum_{k=1}^n \tilde{Y}_{ik}^\alpha tg_{kj}^{\alpha\beta\gamma} -\sum_{k=1}^n tg_{ik}^{\alpha\beta\gamma}\tilde{Y}_{kj}^\alpha\\
&=-[\Lambda_0, \sum_{k,p=1}^n\tilde{F}_{ik}^{\alpha\beta}\tilde{F}_{kp}^{\beta\gamma}\tilde{F}_{pj}^{\gamma\alpha}]+\sum_{k,p,r=1}^n \tilde{Y}_{ik}^\alpha\tilde{F}_{kp}^{\alpha\beta} \tilde{F}_{pr}^{\beta\gamma}\tilde{F}_{rj}^{\gamma\alpha}-\sum_{k=1}^n \tilde{Y}_{ik}^\alpha \delta_{kj} -\sum_{k,p,r=1}^n \tilde{F}_{ik}^{\alpha\beta}\tilde{F}_{kp}^{\beta\gamma}\tilde{F}_{pr}^{\gamma\alpha} \tilde{Y}_{rj}^\alpha+\sum_{k=1}^n\delta_{ik} \tilde{Y}_{kj}^{\alpha}\\
&=-[\Lambda_0, \sum_{k,p=1}^n\tilde{F}_{ik}^{\alpha\beta}\tilde{F}_{kp}^{\beta\gamma}\tilde{F}_{pj}^{\gamma\alpha}]+\sum_{k,p,r=1}^n \tilde{Y}_{ik}^\alpha\tilde{F}_{kp}^{\alpha\beta} \tilde{F}_{pr}^{\beta\gamma}\tilde{F}_{rj}^{\gamma\alpha} -\sum_{k,p,r=1}^n \tilde{F}_{ik}^{\alpha\beta}\tilde{F}_{kp}^{\beta\gamma}\tilde{F}_{pr}^{\gamma\alpha} \tilde{Y}_{rj}^\alpha\\
&=-\sum_{k,p,r=1}^n \tilde{Y}_{ir}^\alpha \tilde{F}_{rk}^{\alpha\beta}\tilde{F}_{kp}^{\beta\gamma}\tilde{F}_{pj}^{\gamma\alpha}+\sum_{k,p,r=1}^n \tilde{F}_{ir}^{\alpha\beta}\tilde{Y}_{rk}^\beta\tilde{F}_{kp}^{\beta\gamma}\tilde{F}_{pj}^{\gamma\alpha}+\sum_{k,p=1}^n tQ_{ik}^{\alpha\beta}\tilde{F}_{kp}^{\beta\gamma}\tilde{F}_{pj}^{\gamma\alpha}\\
&-\sum_{k,p,r=1}^n\tilde{F}_{ik}^{\alpha\beta}\tilde{Y}_{kr}^\beta \tilde{F}_{rp}^{\beta\gamma}\tilde{F}_{pj}^{\gamma\alpha}+\sum_{k,p,r=1}^n \tilde{F}_{ik}^{\alpha\beta}\tilde{F}_{kr}^{\beta\gamma}\tilde{Y}_{rp}^\gamma\tilde{F}_{pj}^{\gamma\alpha}+\sum_{k,p=1}^n\tilde{F}_{ik}^{\alpha\beta}tQ_{kp}^{\beta\gamma}\tilde{F}^{\gamma\alpha}_{pj}\\
&-\sum_{r,k,p=1}^n \tilde{F}_{ik}^{\alpha\beta} \tilde{F}_{kp}^{\beta\gamma}\tilde{Y}_{pr}^\gamma \tilde{F}_{rj}^{\gamma\alpha}+\sum_{r,k,p=1}^n \tilde{F}_{ik}^{\alpha\beta} \tilde{F}_{kp}^{\beta\gamma}\tilde{F}_{pr}^{\gamma\alpha}\tilde{Y}_{rj}^{\alpha}+\sum_{p,k=1}^n\tilde{F}^{\alpha\beta}_{ik}\tilde{F}_{kp}^{\beta\gamma}tQ_{pj}^{\gamma\alpha}\\
&+\sum_{k,p,r=1}^n \tilde{Y}_{ik}^\alpha\tilde{F}_{kp}^{\alpha\beta} \tilde{F}_{pr}^{\beta\gamma}\tilde{F}_{rj}^{\gamma\alpha} -\sum_{k,p,r=1}^n \tilde{F}_{ik}^{\alpha\beta}\tilde{F}_{kp}^{\beta\gamma}\tilde{F}_{pr}^{\gamma\alpha} \tilde{Y}_{rj}^\alpha=\sum_{k,p=1}^n tQ_{ik}^{\alpha\beta}\tilde{F}_{kp}^{\beta\gamma}\tilde{F}_{pj}^{\gamma\alpha}+\sum_{k,p=1}^n\tilde{F}_{ik}^{\alpha\beta}tQ_{kp}^{\beta\gamma}\tilde{F}^{\gamma\alpha}_{pj}+\sum_{p,k=1}^n\tilde{F}^{\alpha\beta}_{ik}\tilde{F}_{kp}^{\beta\gamma}tQ_{pj}^{\gamma\alpha}\\
&=\sum_{k=1}^n tQ_{ik}^{\alpha\beta}F_{kj}^{\beta\alpha}+\sum_{p,k=1}^nF_{ik}^{\alpha\beta}tQ_{kp}^{\beta\gamma}F_{pj}^{\gamma\alpha}+\sum_{p=1}^nF_{ip}^{\alpha\gamma}tQ_{pj}^{\gamma\alpha}
\end{align*}
 
On the other hand, $\delta(\{tQ^{\alpha\beta}\})=(\{ tQ^{\alpha\beta} F^{\beta\alpha}+F^{\alpha\beta}tQ^{\beta\gamma}F^{\gamma\alpha}+F^{\alpha\gamma}tQ^{\gamma\alpha}\})$.
Hence we have $-\delta(\{Q^{\alpha\beta}\})+v_0(G^{\alpha\beta\gamma})=0$.

Lastly, $\delta(G^{\alpha\beta\gamma})= F^{\alpha\beta}G^{\beta\gamma\delta}F^{\beta\alpha}-G^{\alpha\gamma\delta}+G^{\alpha\beta\delta}-G^{\alpha\beta\gamma}$. We note that
\begin{align*}
&I_n-tG^{\alpha\gamma\delta}-tG^{\alpha\beta\gamma}+F^{\alpha\beta}tG^{\beta\gamma\delta}tF^{\beta\alpha}+tG^{\alpha\beta\delta}\\
&=(I_n-tG^{\alpha\gamma\delta})(I_n-t G^{\alpha\beta\gamma})(I_n+\tilde{F}^{\alpha\beta}t G^{\beta\gamma\delta}\tilde{F}^{\beta\alpha})(I_n+t G^{\alpha\beta\delta})\\
&=(I_n-tG^{\alpha\gamma\delta})(I_n-t G^{\alpha\beta\gamma})F^{\alpha\beta}(I_n+t G^{\beta\gamma\delta})F^{\beta\alpha}(I_n+t G^{\alpha\beta\delta})\\
&=(\tilde{F}^{\alpha\delta}\tilde{F}^{\delta\gamma}\tilde{F}^{\gamma\alpha})(\tilde{F}^{\alpha\gamma}\tilde{F}^{\gamma\beta}\tilde{F}^{\beta\alpha})\tilde{F}^{\alpha\beta}(\tilde{F}^{\beta\gamma}\tilde{F}^{\gamma\delta}\tilde{F}^{\delta\beta})\tilde{F}^{\beta\alpha}(\tilde{F}^{\alpha\beta}\tilde{F}^{\beta\delta}\tilde{F}^{\delta\alpha})\\
&=\tilde{F}^{\alpha\delta}\tilde{F}^{\delta\gamma}\tilde{F}^{\gamma\beta}\tilde{F}^{\beta\gamma}\tilde{F}^{\gamma\delta}\tilde{F}^{\delta\beta}\tilde{F}^{\beta\delta}\tilde{F}^{\delta\alpha}=\tilde{F}^{\alpha\delta}\tilde{F}^{\delta\gamma}\tilde{F}^{\gamma\delta}\tilde{F}^{\delta\alpha}=\tilde{F}^{\alpha\delta}\tilde{F}^{\delta\alpha}=I_n
\end{align*}
So we have $-\delta(G^{\alpha\beta\gamma})=0$. Hence $a:=(\{-S^\alpha\},\{Q^{\alpha\beta}\},\{G^{\alpha\beta\gamma}\})$ defines a $2$-cocycle in the above \v{C}ech resolution.

Now we claim that given another arbitrary collection $\{\tilde{F}_{ij}^{\alpha\beta}\}$ and $\{\tilde{Y}'^{\alpha}_{ij}\}$, the associated $2$-cocycle $b:=(\{-S'^{\alpha}\},\{Q'^{\alpha\beta}\},\{G'^{\alpha\beta\gamma}\})$ is cohomologous to the $2$-cocycle $a=(\{-S^\alpha\},\{Q^{\alpha\beta}\},\{G^{\alpha\beta\gamma}\})$ associated with $\{\tilde{F}^{\alpha\beta}\}$ and $\{\tilde{Y}^\alpha\}$. Then $\tilde{F}'^{\alpha\beta}=\tilde{F}^{\alpha\beta}+t F'^{\alpha\beta}$ for some $\Gamma(U_{\alpha\beta}, \mathcal{O}_X)$-valued matrices $F'^{\alpha\beta}$, and $\tilde{Y}'^{\alpha}=\tilde{Y}^\alpha+tY'^\alpha$ for some $\Gamma(U_\alpha,T_X)$-valued matrices $Y'^\alpha$. $\{F'^{\alpha\beta}\}$ defines an element in $C^1(\mathcal{U},\mathscr{H}om(F,F))$ and $\{Y'^\alpha\}$ defines an element in $C^0(\mathcal{U}, T_X\otimes \mathscr{H}om(F,F))$. Then
\begin{enumerate}
\item $tS'^{\alpha}_{ij}=[\Lambda_0,\tilde{Y}'^{\alpha}_{ij}]+\sum_{k=1}^n \tilde{Y}'^{\alpha}_{ik}\wedge \tilde{Y}'^{\alpha}_{kj}=tS_{ij}^\alpha+t[\Lambda_0,Y'^\alpha_{ij}]+t(\sum_{k=1}^n Y'^{\alpha}_{ik}\wedge T^{\alpha}_{kj}+\sum_{k=1}^n Y'^{\alpha}_{ik}\wedge T_{kj}^\alpha) $ so that $S_{ij}'^\alpha-S_{ij}^\alpha=[\Lambda_0, Y_{ij}'^{\alpha}]+Y'^{\alpha}_{ik}\wedge T_{kj}^\alpha+\sum_{k=1}^n Y_{ik}'^\alpha\wedge T_{kj}^\alpha$. Hence $v_0(\{-Y'^\alpha\})=\{-S'^\alpha-(-S^\alpha)\}$.
\item  $tQ_{ij}'^{\alpha\beta}=-[\Lambda_0, \tilde{F}_{ij}'^{\alpha\beta}]+\sum_{k=1}^n \tilde{Y}_{ik}'^\alpha \tilde{F}_{kj}'^{\alpha\beta}-\sum_{k=1}^n \tilde{F}_{ik}'^{\alpha\beta} \tilde{Y}_{kj}'^\beta=tQ_{ij}^{\alpha\beta}+t(-[\Lambda_0, F_{ij}'^{\alpha\beta}]+\sum_{k=1}^n Y_{ik}'^\alpha F_{kj}^{\alpha\beta}+\sum_{k=1}^n T_{ik}^\alpha F_{kj}'^{\alpha\beta}-\sum_{k=1}^n F_{ik}'^{\alpha\beta}T_{kj}^\beta-\sum_{k=1}^n F_{ik}^{\alpha\beta} Y_{kj}'^\beta)$ so that $Q_{ij}'^{\alpha\beta}-Q_{ij}^{\alpha\beta}=\sum_{k=1}^n Y_{ik}'^\alpha F_{kj}^{\alpha\beta}-\sum_{k=1}^n F_{ik}^{\alpha\beta} Y_{kj}'^\beta-[\Lambda_0, F_{ij}'^{\alpha\beta}]+\sum_{k=1}^n T_{ik}^\alpha F_{kj}'^{\alpha\beta}-\sum_{k=1}^n F_{ik}'^{\alpha\beta}T_{kj}^\beta$. Hence $\delta(\{-Y'^\alpha\})+v_0(\{F'^{\alpha\beta}\})=\{Q'^{\alpha\beta}-Q^{\alpha\beta}\}$
\item $\delta_{ij}+t G_{ij}'^{\alpha\beta\gamma}=\sum_{k,p=1}^n \tilde{F'}_{ik}^{\alpha\beta}\tilde{F'}_{kp}^{\beta\gamma}\tilde{F'}_{pj}^{\gamma\alpha}=\delta_{ij}+tG_{ij}^{\alpha\beta\gamma}+t(\sum_{k,p=1}^n{F'}_{ik}^{\alpha\beta}F_{kp}^{\beta\gamma}F_{pj}^{\gamma\alpha}+\sum_{k,p=1}^n F_{ik}^{\alpha\beta}{F'}_{kp}^{\beta\gamma}F_{pj}^{\gamma\alpha}+\sum_{k,p=1}^nF_{ik}^{\alpha\beta}F_{kp}^{\beta\gamma}{F'}_{pj}^{\gamma\alpha})=\delta_{ij}+tG_{ij}^{\alpha\beta\gamma}+t(\sum_{k=1}^n {F'}_{ik}^{\alpha\beta} F_{kj}^{\beta\alpha}+\sum_{k,p=1}^nF_{ik}^{\alpha\beta}{F'}_{kp}^{\beta\gamma}F_{pj}^{\gamma\alpha}+\sum_{p=1}^n F_{ip}^{\alpha\gamma}F_{pj}'^{\gamma\alpha})$ so that $G'^{\alpha\beta\gamma}_{ij}-G^{\alpha\beta\gamma}_{ij}=\sum_{k=1}^n F_{ik}'^{\alpha\beta}F_{kj}^{\beta\alpha} +\sum_{k,p=1}^nF_{ik}^{\alpha\beta}{F'}_{kp}^{\beta\gamma}F_{pj}^{\gamma\alpha}+\sum_{p=1}^n F_{ip}^{\alpha\gamma}F_{pj}'^{\gamma\alpha}$. Hence $\delta(\{F'^{\alpha\beta}\})=\{G'^{\alpha\beta\gamma}_{ij}-G^{\alpha\beta\gamma}_{ij}\}$.
\end{enumerate}
Hence $(\{-Y'^\alpha\},\{F'^{\alpha\beta}\})\in \mathcal{C}^0(\mathcal{U},T_X\otimes \mathscr{H}om(F,F))\oplus \mathcal{C}^1(\mathcal{U},\mathscr{H}om(F,F))$ in the above \v{C}ech resolution is mapped to $b-a$ so that $a$ is cohomologous to $b$. So given a small extension $e:0\to (t)\to \tilde{A}\to A\to 0$ and an infinitesimal deformation $\eta$ of $(F,\nabla)$ over $A$, we can associate an element $o_\eta(e):=$ the cohomology class of $a\in \mathbb{H}^2(X,\Lambda_0,\mathscr{H}om(F,F)^\bullet,\nabla_{\mathscr{H}om(F,F)})$. We note that $o_\eta(e)=0$ if and only if there exists a collection of $\{\tilde{F}_{ij}^{\alpha\beta}\},\{\tilde{Y}_{ij}^{\alpha}\}$ defining an infinitesimal deformation over $\tilde{A}$ which induces $\eta$. This proves $(2)$ in Proposition \ref{ve2}.
\end{proof}

\bibliographystyle{amsalpha}
\bibliography{References-Rev9}

\end{document}